\documentclass[a4paper,11pt]{amsart}

\usepackage{amsmath}
\usepackage{amsthm}
\usepackage{amsfonts}
\usepackage{amssymb}
\usepackage{latexsym}
\usepackage{mathrsfs}
\usepackage{dutchcal} 
\usepackage{enumitem}

\usepackage[section]{placeins}

\usepackage[abs]{overpic}

\usepackage[titletoc]{appendix}

\usepackage[usenames]{color}
\usepackage[all]{xy}

\usepackage{graphicx}
\usepackage{latexsym}
\usepackage[cp850]{inputenc}
\usepackage{epsfig}
\usepackage{psfrag}

\usepackage{thmtools}
\usepackage{thm-restate}

\usepackage[hidelinks]{hyperref}
\hypersetup{ 
  colorlinks   = true,            
  urlcolor     = {blue!50!black}, 
  linkcolor = black,
  citecolor = red,
}

\newtheorem{theorem}{Theorem}[section]{\bf}{\it}
\newtheorem{lemma}[theorem]{Lemma}{\bf}{\it}
\newtheorem{proposition}[theorem]{Proposition}{\bf}{\it}
\newtheorem{corollary}[theorem]{Corollary}{\bf}{\it}
{\bf}{\it} 
{\bf}{\it}
\newtheorem*{theorem*}{Theorem}

\newtheorem{remark}[theorem]{Remark}
\newtheorem*{remark*}{Remark}
{\bf}{\it}
{\bf}{\it}

{\bf}{\it}
{\bf}{\it}
\newtheorem{convention}[theorem]{Convention}{\bf}{\it}

\newtheorem{definition}[theorem]{Definition}

\newtheorem{standing}[theorem]{Standing assumptions}{\bf}{\it}
{\bf}{\it}

\theoremstyle{remark}

\theoremstyle{definition}

\theoremstyle{remark}

\newcommand{\R}{\mathbb R}

\newcommand{\C}{\mathbb C}
\newcommand{\N}{\mathbb N}

\newcommand{\loc}{{\operatorname{loc}}}

\newcommand{\dist}{{\operatorname{dist}\,}}
\newcommand{\id}{{\operatorname{id}}}

\newdimen\vintkern\vintkern11pt
\def\vint{-\kern-\vintkern\int}

\newcommand{\norm}[1]{\lVert #1 \rVert}

\newcommand{\bS}{\mathbb{S}}

\newcommand{\cL}{\mathcal{L}}

\newcommand{\cP}{\mathcal{P}}

\newcommand{\interior}{\mathrm{int}}

\newcommand{\Star}{\mathrm{St}}

\newcommand{\cD}{\mathcal{D}}

\newcommand{\cT}{\mathcal{T}}

\newcommand{\cC}{\mathcal{C}}

\newcommand{\cR}{\mathcal{R}}

\newcommand{\sfK}{\mathsf{K}}

\newcommand{\Wedge}{\mathrm{Wedge}}

\newcommand{\Tunnel}{\mathsf{Tunnel}}
\newcommand{\partialtop}{\partial_{\mathrm{top}}}
\newcommand{\tunnel}{\mathsf{tunnel}}

\newcommand{\cA}{\mathcal{A}}

\newcommand{\fT}{\mathfrak{T}}

\newcommand{\cl}{\mathrm{cl}}

\newcommand{\D}{\mathsf{D}}

\newcommand{\cS}{\mathcal S}

\newcommand{\sK}{\mathsf{K}}

\renewcommand{\emptyset}{\varnothing}

\newcommand{\Refine}{\mathrm{Ref}}

\newcommand{\Core}{\mathsf{Core}}

\newcommand{\Comp}{\mathsf{Comp}}
\newcommand{\Real}{\mathsf{Lift}}

\newcommand{\Realdented}{\mathsf{Lift}^{\mathsf{dented}}}

\newcommand{\Rec}{\mathrm{Comp}^{\mathsf{rec}}}

\newcommand{\Urec}{U^{\mathsf{rec}}}

\newcommand{\Gammacut}{\Gamma^{\mathsf{cut}}}

\newcommand{\Tr}{\mathsf{Tr}}

\newcommand{\Span}{\mathrm{Span}}
\newcommand{\sfJ}{\mathsf{J}}

\newcommand{\sfE}{\mathsf{E}}
\newcommand{\sfC}{\mathsf{C}}

\newcommand{\sfR}{\mathsf{R}}
\newcommand{\sfT}{\mathsf{T}}

\newcommand{\sfRC}{\mathsf{RC}}

\newcommand{\Channel}{\mathsf{Ch}}

\newcommand{\Dent}{\mathsf{D}}

\newcommand{\RC}{\mathsf{RC}}

\title[On Lakes of Wada]{On Lakes of Wada}

\author{Pekka Pankka}
\address{Department of Mathematics and Statistics, P.O. Box 68 (Pietari Kalmin katu 5), FI-00014 University of Helsinki, Finland}

\email{pekka.pankka@helsinki.fi}

\author{Jang-Mei Wu}
\address{Department of Mathematics, University of Illinois,  1409 West Green Street, Urbana, IL 61822, USA}
\email{jmwu@illinois.edu}

\thanks{This work was supported in part by the Research Council of Finland (RCF) projects \#256228, \#297258 and \#332671, RCF Center of Excellence FiRST, and a grant from the Simons Foundation \#353435. }


\makeindex

\begin{document}

\begin{abstract}
There exist Lakes of Wada in $\bS^n, n\ge 3,$ which are quasiconformally equivalent to a Euclidean ball and are John domains.
\end{abstract}

\maketitle

\setcounter{tocdepth}{1}  

\tableofcontents

\section{Introduction}

Classically, Lakes of Wada are three or more open sets in $\bS^2$ which have a common boundary. The common boundary is necessarily a continuum and it is called a Wada continuum. This definition generalizes immediately to higher dimensional spheres and further to connected manifolds (possibly with boundary):
\emph{A continuum $X$ on a connected manifold $M$ is a \emph{Wada continuum} if $M\setminus X$ has at least three connected components and each point in $X$ is a common boundary point of all connected components of $M\setminus X$. In this case, the connected components of $M\setminus X$ are called \emph{Lakes of Wada}}.

In $\bS^2$, Lakes of Wada were constructed by Yoneyama \cite{Yoneyama} in 1917, who credited the idea to Takeo Wada. They were also known previously by  Brouwer \cite{Brouwer} in 1910. In 1951, Luba\'nski \cite{Lubanski} proved the existence of Lakes of Wada in $\bS^n$ for $n \ge 3$.

While Lakes of Wada in $\bS^2$ are conformally equivalent to the unit disk, by the Riemann Mapping Theorem,
they do not have the internal geometry of the unit disk. For example, Lakes of Wada are never John domains (Proposition \ref{prop:Wada_2dim}). This  changes when we pass from two to higher dimensions.

In this article, we construct Wada continua in $\bS^n$ for $n\geq 3$, whose complementary components are both quasiconformally equivalent to the unit ball and John domains.

\begin{restatable}[Lakes of Wada]{theorem}{introthmWadasphere}
\label{intro-thm:Wada-on-sphere}
For $n\ge 3$ and $m \ge 3$, there exists a Wada continuum in $\bS^n$ having exactly $m$ complementary components, each of which is both quasiconformal to the Euclidean ball $B^n(0,1)$ and a John domain.
\end{restatable}

This result on Wada continua on spheres is an immediate consequence of the following existence result for Wada continua on compact Riemannian manifolds with boundary. 

In what follows, for an $n$-manifold $M$ with boundary, we denote $\interior M$ the set of manifold points of $M$, that is, the set of points having neighborhoods homeomorphic to $\R^n$, and denote $\partial M$ the manifold boundary of $M$, that is, non-manifold points of $M$; recall that $\partial M$ is an $(n-1)$-manifold in its relative topology. The sets $\interior M$ and $\partial M$ are called, respectively, the manifold interior and manifold boundary of $M$. To distinguish, we denote by $\partialtop E$ the topological boundary of a subset $E$ of a topological space.

\begin{restatable}{theorem}{introthmWadamanifold}
\label{intro-thm:Wada-Riemannian-manifold} 
\index{Lakes of Wada!quasiconformally controlled} 
Let $n\ge 3$, $m\geq 3$, and $M$ be a compact connected Riemannian $n$-manifold with $m$ boundary components. Then there exist a constant $\sK=\sK(n,M)>1$ and a Wada continuum $X\subset \interior M$ having the following properties:
\begin{enumerate}
\item $M\setminus X$ has $m$ connected components, each of which contains a connected component of $\partial M$; and \label{item:Wada}
\item each component $\Omega$ of $M \setminus X$ is a John domain in $M$ and is $\sK$-quasiconformal to the product $(\Omega \cap \partial M) \times [0,1)$. \label{item:John+QC}
\end{enumerate}
\end{restatable}

\begin{remark}
In the case of two boundary components, the proof of Theorem \ref{intro-thm:Wada-Riemannian-manifold} yields a continuum, which is the the common boundary of two complementary domains, satisfying properties \eqref{item:Wada} and \eqref{item:John+QC}; see Theorem \ref{thm:Wada-Riemannian-manifold} for a statement.
We follow here the tradition and do not call continua having less than three complementary components Wada continua. 

Continua $X\subset \interior M$ having no interior points and satisfying properties \eqref{item:Wada} and \eqref{item:John+QC},  exist also on compact manifolds $M$ without boundary or having only one boundary component. In particular, for a closed and connected Riemannian $n$-manifold $M$, there exists a continuum $X$, having no interior points, for which $M\setminus X$ is a John domain  quasiconformal to an open Euclidean unit $n$-ball. Recall that a compact manifold without boundary is called closed. We refer to Corollary \ref{cor:one_none_boundary} for the constructions of these continua. 
\end{remark}

In order to prove Theorem \ref{intro-thm:Wada-on-sphere}, it suffices to fix $m$ mutually disjoint closed PL balls $B_1,\ldots, B_m$ in $\bS^n$ and  apply Theorem \ref{intro-thm:Wada-Riemannian-manifold} to manifold $M = \bS^n \setminus \interior (B_1 \cup \cdots \cup B_m)$.

We recall now the definitions in these statements. A homeomorphism $f\colon M \to N$ between two oriented Riemannian $n$-manifolds without boundary ($n\geq 2$) is \emph{$\sK$-quasiconformal} for $\sK \ge 1$, if $f$ belongs to the Sobolev space $W^{1,n}_\loc(M;N)$ and satisfies the distortion inequality
\begin{equation}
\label{eq:dist}
\norm{Df}^n \le \sK J_f\quad \text{a.e.}\ M,
\end{equation}
where $\norm{Df}$ is the norm of the weak differential $Df$ of $f$ and $J_f$ is the Jacobian determinant of $f$. 
We say that a homeomorphism $f\colon M\to N$ between Riemannian manifolds with boundary is \emph{quasiconformal} if there exist open Riemannian manifolds $M'$ and $N'$ containing manifolds $M$ and $N$ as smooth submanifolds with boundary, respectively, and $f$ extends to a quasiconformal homeomorphism $f' \colon M' \to N'$.

\begin{remark*}
This definition of quasiconformal homeomorphism between Riemannian manifolds with boundary stems from the corresponding definition of diffeomorphisms. Note that, a diffeomorphism between compact Riemannian manifolds with boundary is a quasiconformal homeomorphism. 
Note also that a homeomorphism between two compact Riemannian manifolds with boundary, if quasiconformal with respect to some Riemannian metrics, is quasiconformal with respect to all Riemannian metrics.
The definition in the case of manifolds with boundary given here is equivalent to the requirement that, for a homeomorphism $M\to N$ between manifolds with boundary, the restriction $\interior M\to \interior N$ is quasiconformal. This follows from classical boundary extension of quasiconformal homeomorphis and theorems of Tukia and V\"ais\"al\"a \cite{Tukia-Vaisala-AASF,Tukia-Vaisala-Annals}. 
There are also other equivalent definitions, in particular, a metric definition and a geometric definition using modulus of curve families; we refer to monographs of Ahlfors \cite{Ahlfors-book}  and V\"ais\"al\"a \cite{Vaisala-book} for the quasiconformal theory. 
\end{remark*}

The concept of John domains is due to F. John \cite{John}, and the term John domain was coined by  Martio and Sarvas in  \cite{Martio-Sarvas}. 
John domains occur frequently in the study of elasticity and geometric analysis.
A proper subdomain $D$ of $\R^n, n\geq 2$ is a $C$-\emph{John domain for $C\ge 1$}, if any two points $a, b \in D$ can be joined by a rectifiable curve $\gamma \subset D$ satisfying the length-distance estimate
\begin{equation}
\label{eq:John-intro-Rn}
\min \{s(\gamma(a,x)), s(\gamma(b,x))\} \leq C  \, \dist (x,\R^n \setminus D)\quad \text{for all}\,\, x\in \gamma, 
\end{equation}
where $\gamma(a,x)$ (resp.~$\gamma(b,x)$) is the part of $\gamma$ between $a$ and $x$ (resp.~between $b$ and $x$) and  $s(\cdot)$ is the length. 

An open and connected proper subset $D$ of a Riemannian manifold $M$ (possibly with boundary) is a \emph{$C$-John domain in $M$} if it satisfies an length-distance estimate  \eqref{eq:John-intro-Rn} for rectifiable curves with respect to $\dist(x, M\setminus D)$.

Roughly speaking  any two points in a John domain $D$ can be connected by a  twisted double-cone in $D$ having vertices at these two points. 
If $f \colon B^n(0,r) \to M$ is a quasiconformal embedding of a Euclidean ball $B^n(0,r)$ into a Riemannian manifold, then the image $f(B^n(0,r/2))$ remains somewhat round and has diameter roughly the distance to the complement 
of the image. While many quasiconformal images of Euclidean balls are John domains, these two classes are not the same.
For example, for $n\ge 3$, a domain in $\R^n$ that contains an inward-directed spike can not be mapped to a ball by a quasiconformal map, but it can be a John domain. On the other hand, a domain, which contains an outward-directed spike, is not a John domain, but it can be mapped by a quasiconformal map to a ball. We refer to \cite{Gehring-Vaisala} and \cite{Heinonen-book} for the explanation.

\smallskip

Apart from the intrinsic geometry of higher dimensional Lakes of Wada, our interest in the Wada continua stems from their roles in the construction of quasiregular mappings 
and in complex dynamics. Recall that a continuous mapping $f\colon M \to N$ between oriented Riemannian manifolds is \emph{quasiregular} if $f$ belongs to the Sobolev space $W^{1,n}_\loc(M;N)$ and satisfies the distortion inequality \eqref{eq:dist} for some $\sK\ge 1$.

Rickman \cite{Rickman_Acta} used two dimensional complexes -- called map complexes -- in his construction of a quasiregular map $\R^3 \to \bS^3$ omitting a given finite set in $\bS^3$. Analogous $2$-dimensional complexes were later used by Heinonen and Rickman \cite{Heinonen-Rickman_Topology, Heinonen-Rickman_Duke} -- first under the name of separating polyhedra in \cite{Heinonen-Rickman_Topology} and then under the name separating complexes in \cite{Heinonen-Rickman_Duke} -- to construct, for example, quasiregular maps $\bS^3 \to \bS^3$, whose branch sets  contain Antoine's necklaces, and quasiregular maps $B^3 \to B^3$ with no radial limits in the limit set of certain Kleinian groups. 
The common property of these constructions is that the distortion of the map is controlled by a quantity depending only on the dimension
and topological data associated to the construction. In particular, in these constructions the distortion does not put restriction on the (global) degree of the mappings.
We refer to \cite{Heinonen-Rickman_Topology}, \cite{Heinonen-Rickman_Duke}, and \cite{Rickman_Acta} for the terminology and precise statements.
In higher dimensions $n\ge 4$, Drasin and the first named author used a cubical version of separating complexes (with respect to standard cubical structure of $\R^n$) in \cite{Drasin-Pankka} to construct a quasiregular map $\R^n \to \bS^n$ omitting a given finite set in $\bS^n$;  see the end of the introduction in \cite{Drasin-Pankka} for discussion.

In complex dynamics,  Mart\'i--Pete, Rempe, and Waterman recently gave a positive answer  in \cite{Marti-Rempe-Waterman} to a question of Fatou from 1926 on the existence of a Wada continuum in $\bS^2$ which is the common boundary of infinitely many Fatou components of a transcendental entire function $\C \to \C$. In dimensions $n\ge 3$, an analog of Fatou's question asks: \emph{Does there exist, for $n\ge 3$, a Wada continuum in $\bS^n$ which is a common boundary of Fatou components of a uniformly quasiregular mapping $f\colon \R^n \to \R^n$?} We refer to \cite{Marti-Rempe-Waterman} for the history of Fatou's problem and examples of Wada continua in dynamics. 

\subsubsection*{Outline of the proof} The  Wada continuum $X$ in Theorem \ref{intro-thm:Wada-Riemannian-manifold} is found by an iterative subdivision process starting from a cubical structure $K$ on the manifold $M$.

We begin by passing from the Riemannian manifold  $(M,g)$ to a cubical $n$-complex $K$ having space $|K|=M$. The cubical complex carries a natural polyhedral metric $d_K$ 
for which the metric space $(|K|, d_K)$ is quasisimilar to $(M,d_g)$, where $d_g$ is the length metric induced by the Riemannian metric $g$; see Proposition \ref{prop:Riemannian-to-cubical}. 

We take next a monotone sequence $K_0=K, K_1, K_2, \ldots$ of cubical subdivisions of $K$ in which each complex $K_{\ell+1}$ is obtained by subdividing the cubes in $K_\ell$ with a fixed rule. This subdivision respects the metric $d_K$ in the sense that cubes in $K_\ell$ have side length $3^{-\ell}$ with respect to metric $d_K$.

The crux of the proof is to construct  
a sequence $Z_0 \subset K_0, Z_1 \subset K_1, Z_2\subset K_2, \ldots$ of $(n-1)$-dimensional (separating) subcomplexes in the interior of $|K|$ 
having the property that each connected component of $|K|\setminus |Z_\ell|$ contains exactly one boundary component of $|K|$ and is quasiconformal to a collar of that boundary component. 
The Wada continuum $X$ is obtained as the Hausdorff limit of the spaces $|Z_\ell|\subset M$ of complexes $Z_\ell$.

Heuristically, the inductive step
from $Z_\ell$ to $Z_{\ell+1}$ may be viewed as a controlled way of trading $n$-dimensional cubes in $K_{\ell+1}$ between 
components of $|K|\setminus |Z_\ell|$ near $Z_\ell$. The aims of this trading are threefold: 
\begin{enumerate}
\item at each step, components of $|K|\setminus |Z_{\ell+1}|$ remain in one-to-one correspondence with respect to the components of $|\partial K|$, 
\item the maximum distance from each cube in $Z_{\ell}$ to any of the components of $|K|\setminus |Z_\ell|$ tends to zero as $\ell \to \infty$, and 
\item each $|K|\setminus |Z_{\ell+1}|$ is $\sK$-quasiconformal to a fixed collar of $|\partial K|$ for a distortion constant $\sK$ independent of $\ell$.
\end{enumerate}

One difficulty of the construction stems from condition (2), which may be viewed as a Wada-type condition for complexes $Z_\ell$. 
More precisely, each $(n-1)$-cube $q$ in $Z_\ell$ is  a face of exacty two $n$-cubes in $K_\ell$. Thus $q$ is
 on the boundary of at most two of the components of $|K|\setminus |Z_\ell|$. By (2), however, $q$ is required to be close to all components of $|K|\setminus |Z_\ell|$.

\subsubsection*{Organization of the article}

We discuss, in Section \ref{sec:From-Riem-To-Cubical}, cubical structures on Riemannian manifolds and the polyhedral metrics on cubical complexes.  We also introduce some terminology related to cubical complexes. 

In Section \ref{sec:adjacency-graphs}, we discuss first adjacency graphs, realizations of subgraphs of adjacency graphs as complexes, and cut-graphs. We define next a special class of complexes, called tunnels, and prove a tunnel contracting lemma which is applied repeatedly in our proof. Finally, we introduce the notion of a good complex, which comprises all essential properties of a cubical complex that are needed in the construction.

Separating complexes are defined in Section \ref{sec:Separating-complexes}, along with the statement of a theorem on the evolution of separating complexes, which is used to guide our construction.

In Section \ref{sec:Indentation}, we introduce the notion of indentation and prove a flattening theorem for indentations, which is used to control the quasiconformality constant of the domains.

The iterative
construction of the sequence $Z_0,Z_1,Z_2,\ldots$ of separating complexes is given in Sections \ref{sec:reservoir-canal-system} and \ref{sec:evolution-seq}.
In Section  \ref{sec:reservoir-canal-system} we discuss the building blocks -- reservoir-canal systems -- for our construction and the method of channeling. In Section \ref{sec:evolution-seq} we discuss the iterative processes.
Quasiconformal stability of the sequence $Z_0,Z_1,Z_2,\ldots$  is proved in Section \ref{sec:quasiconformality}. Finally Theorem \ref{intro-thm:Wada-Riemannian-manifold} is proved in Section \ref{sec:Wada}. In this final section, we also show that two dimensional Lakes of Wada are not John-domains.


\section{Cubical structures on Riemannian manifolds}
\label{sec:From-Riem-To-Cubical}

In this section, we discuss cubical structures on Riemannian manifolds. Recall that every Riemannian $n$-manifold $(M,g)$ admits a triangulation $T$ in which every $n$-simplex $(\sigma, g|_{\sigma})$ is bilipschitz to an affine Euclidean $n$-simplex. We refer to Cairns \cite{Cairns}, Whitehead \cite{WhiteheadC1}, and Munkres \cite{Munkres} for constructions; see also e.g.~Dyer, Vegter, and Wintraecken \cite{Dyer-Vegter-Wintraecken}, Peltonen \cite{Peltonen}, or Saucan \cite{Saucan} for triangulations with controlled geometry.  

In what follows, we call a simplicial complex $T$ having space $M$ a triangulation of $M$. We also denote $d_T$ the simplicial metric on $M$ under which simplices of $T$ are similar to the standard simplices $\Delta^k \subset \R^{k+1}$, for $k=0,\ldots, n$, in the Euclidean space.

Barycentric subdivision of a triangulation induces a cubical structure on $M$ (see e.g.~Shtan'ko and Shtogrin \cite{Shtanko-Shtogrin}), which carries a natural polyhedral metric. Although this reduction of a Riemannian manifold to a metric cubical complex is known to the experts, 
we devote 
this section to establish this connection. We discuss this in three parts: cubical complexes and related notations, polyhedral metrics on cubical complexes, and finally the change of metric on barycentric subdivisions from cubical to simplicial complexes,  then from simplicial to cubical complexes.

The following proposition allows us to pass from compact Riemannian manifolds with boundary to finite cubical complexes in our main theorem (Theorem \ref{intro-thm:Wada-Riemannian-manifold}).

\begin{proposition}
\label{prop:Riemannian-to-cubical}
Associated to each compact Riemannian manifold with boundary $(M,g)$, there exist a cubical complex $K$ with a  metric $d_K$, in which each cube of $K$ is isometric a Euclidean unit cube, and a quasisimilar homeomorphism $(M,g) \to (K,d_K)$.
\end{proposition}

Recall that a map $f\colon X\to Y$ is \emph{$L$-bilipschitz for $L\ge 1$}, if  \index{quasisimilarity}
\[
L^{-1} \,d(x,y) \le d(f(x),f(y)) \le L  \,d(x,y)  \quad \text{for all}\,\,\,x,x'\in X.
\]
A  map $f\colon X\to Y$ is \emph{$L$-quasisimilar for $L\ge 1$}, 
if there exists $\lambda>0$ for which \index{quasisimilarity}
\[
L^{-1} \lambda \,d(x,y) \le d(f(x),f(y)) \le L \lambda \,d(x,y)  \quad \text{for all}\,\,\,x,x'\in X.
\]

The proof of this proposition is concluded in Section \ref{sec:from-s-to-c}. We first introduce the polyhedral metric; see e.g.~Bridson and Haefliger \cite[Chapter I.7]{Bridson-Haefliger} for an analogous discussion. After the proof of Proposition \ref{prop:Riemannian-to-cubical}, we introduce the terminology of refinements of a cubical complex in Section \ref{sec:refinement-metric}, and the notion of good complexes in Section \ref{sec:good-complexes}.

\subsection{Preliminaries on cubical complexes}\label{sec:prelim-cubical-complex}

As usual, a \emph{simplicial complex} is a complex whose elements are cells with the standard simplicial structure.

\begin{definition}
\label{def:cell-complex}
A collection $K$ of cells is a \emph{cell complex} if, for $C, C' \in K$, $C \cap C'\in K\cup \{\emptyset\}$.
\end{definition}
It is immediate that the restriction $K|_C=\{C'\subset C \colon C'\in K\}$ of a cell complex $K$ to its cell $C\in K$ is a cell-complex.

For $k\in \N$, we call the complex $\sfC_k=\{\{0\},\{1\}, [0,1]\}^k$ a \emph{standard cubical structure on $[0,1]^k$}; we also set $\sfC_0 = \{\{0\},\{1\}, [0,1]\}^0=\{\{0\}\}$ and $\sfC_{-1}=\{\{0\},\{1\}, [0,1]\}^{-1}=\emptyset$. 

Given a $k$-cell $Q$ and a homeomorphism $\phi \colon Q\to [0,1]^k$, we call the complex 
\[
\sfC_\phi(Q) = \{ \phi^{-1}q \colon q\in \sfC_k\}
\]
a \emph{standard cubical structure on $Q$ (induced by $\phi$)}. If there is no ambiguity, we simply write $\sfC(Q) = \sfC_\phi(Q)$. The $(k-1)$-cubes and $(k-2)$-cubes in $\sfC_\phi(Q)$ are called \emph{faces and edges of $Q$}, respectively.

Cubical complexes are cell complexes, whose cells have cubical structure. In what follows, we use the following slightly more restrictive definition.

\begin{definition}\label{def:cubical-complex}\index{cubical complex}
A cell complex $K$ is a \emph{cubical $n$-complex for $n\ge 0$} if 
\begin{enumerate}
\item for each $Q\in K$, the restriction $K|_Q=\sfC(Q)$ is a standard cubical structure on $Q$, 
\item for each pair $Q,Q'\in K$ of cells, $\sfC(Q)|_{Q\cap Q'} = \sfC(Q')|_{Q\cap Q'}$ and the composition $(\phi_Q|_{Q\cap Q'}) \circ (\phi_{Q'}|{Q\cap Q'})^{-1} \colon \phi_{Q'}(Q\cap Q') \to \phi_Q(Q\cap Q')$ is an isometry,
and \label{eq:def-cub-comp-2}
\item each cell in $K$ is contained in an $n$-cell in $K$. \label{eq:def-cub-comp-3}
\end{enumerate}
The $k$-cells in a cubical $n$-complex $K$ are called $k$-cubes. A collection of cells $P\subset K$ is a \emph{cubical $k$-subcomplex of $K$} if $P$ is a cubical $k$-complex.
\end{definition}

In this definition, we tacitly assume that, for each $n$-cell $Q$ in $K$, we have fixed a homeomorphism $\phi_Q \colon Q\to [0,1]^n$.

\begin{remark}
Property \eqref{eq:def-cub-comp-3} is not typically part of the definition of a complex, and complexes satisfying \eqref{eq:def-cub-comp-3} are called in the literature either homogeneous or pure complexes.  The stronger condition holds for cubical complexes whose spaces are manifolds.
We emphasize that, in what follows, we consider cubical complexes and their subcomplexes satisfying this additional assumption.
\end{remark}

Also the second part of property \eqref{eq:def-cub-comp-2} is not typically part of the definition of a cubical complex. This condition, however, yields a well-defined polyhedral metric $d_K$ on a cubical $n$-complex $K$ having connected space $|K|$. Before discussing construction of polyhedral metrics on cubical complexes, we finish this general discussion with two definitions.

\begin{definition}
Given a homeomorphism $\phi \colon X\to |K|$, where $K$ is a cubical $n$-complex, we call
\[
\phi^*K = \{ \phi^{-1}Q \colon Q\in K \}
\]
the \emph{pull-back of the complex $K$ under $\phi$}; clearly $|\phi^*K| = X$. A homeomorphism $\phi \colon |K|\to |K'|$, between spaces of cubical complexes $K$ and $K'$, is called a \emph{cubical isomorphism} if $\phi^*(K') = K$.  
\end{definition}

\begin{definition}\label{def:essential-partition}
We say that $n$-subcomplexes $K_1,\ldots, K_\ell$ of a cubical $n$-complex $K$ are \emph{essentially disjoint} if $K_i$ and $K_j$ do not have $n$-cubes in common for $i\ne j$.
 We say that $\{K_1,\ldots, K_\ell\}$ is an \emph{essential partition of a cubical $n$-complex $K$} if $K= K_1\cup \cdots \cup K_\ell$, and $K_1,\ldots, K_\ell$ are essentially disjoint.
\end{definition}

\subsubsection{Polyhedral metric of a cubical complex}

We define first, for each $n$-cube $Q$ in $K$, a metric $d_Q \colon Q\times Q\to [0,\infty)$ by the formula $d_Q(x,y) = |\phi_Q(x)-\phi_Q(y)|$ for $x,y\in Q$. By \eqref{eq:def-cub-comp-2} in Definition \ref{def:cubical-complex}, $d_Q(x,y)=d_{Q'}(x,y)$ for $x,y\in |Q\cap Q'|$.

We call a sequence $(x_0,\ldots, x_m)$ in $|K|$ a $K$-chain if $x_i$ and $x_{i-1}$ belong to the same $n$-cube in $K$. Since $|K|$ is connected, any two points $x$ and $y$ in $|K|$ may be connected by a $K$-chain. Thus we may define the polyhedral metric $d_K \colon |K| \times |K|\to [0,\infty)$ of $|K|$ by the formula 
\[
d_K(x,y) = \inf_{(x_0,\ldots, x_m)} \sum_{i=1}^m d_{Q_i}(x_{i-1},x_i),
\]
where the infimum is taken over $K$-chains connecting $x$ and $y$ and each $Q_i$ is an $n$-cube containing $x_{i-1}$ and $x_i$.

\begin{remark}
With respect to the polyhedral metric $d_K$ on a cubical $n$-complex $K$, each $n$-cube $Q$ in $K$ is isometric to $[0,1]^n$ and $d_Q(x,y)= d_K(x,y)$ for $x,y\in |Q|$.
\end{remark}

\subsubsection{Subcomplexes of cubical complexes}

Let $K$ be a cubical $n$-complex and $0\leq k \leq n$. We denote 
$K^{[k]} = \{ q \in K \colon q \text{ is a $k$-cube}\}$ the collection of all $k$-cubes in $K$ and $K^{(k)} = \bigcup_{\ell \le k} K^{[\ell]}$ the \emph{$k$-skeleton of $K$}.

\begin{definition}\label{def:span}
Let $S$ be a subset of a cubical $n$-complex $K$. The subcomplex $\Span_K(S)$ \emph{spanned by $S$} is the smallest subcomplex of $K$ that contains $S$. 
\end{definition}

We say that an $(n-1)$-cube $q$ in a cubical $n$-complex $K$ is a \emph{one-sided} if there exists only one $n$-cube in $K$ containing $q$. 
The boundary $\partial K$ of a cubical $n$-complex is the subcomplex of $K$ spanned by all one-sided cubes, that is, 
\[
\partial K = \Span_K\left( \{ q\in K^{[n-1]} \colon q \text{ is one-sided}\}\right).
\]
A subcomplex $\Sigma \subset K$ is a \emph{boundary component of $K$} if $|\Sigma|$ is a component of $|\partial K|$. Note that, for a cubical complex $K$ having a manifold with boundary as its space $|K|$, the space $|\partial K|$ is the manifold boundary of $|K|$.

Let $P$  and $N$ be subcomplexes of $K$. We call the complex  
\[
P - N =  \Span_P(\{ q \in P \colon q\not \in N\})
\]
the \emph{difference of $P$ and $N$}. Note that, under this definition, $P-N=P$ when the dimension of $P\cap N$ is strictly lower than the dimension of $P$.

Our definition of star is more restrictive than the usual definition; in the usual definition, a star is spanned by all cubes meeting $q$. 

\begin{definition}\label{def:star}
The \emph{star $\Star_K(q)$ of a $k$-cube $q$ in $K$} is the smallest subcomplex spanned by all cubes in $K$ containing $q$, that is,
\[
\Star_K(q)=\Span_K(\{Q\in K\colon q\subset Q\}).
\]
The \emph{star $\Star_K(S)$ of a subset $S\subset K$} is \index{$\Star_K(S)$} 
$\Star_K(S) = \bigcup_{q\in S} \Star_K(q)$.
\end{definition}

The star of a subcomplex yields the notion of locally Euclidean complexes. For the definition, let $\sfC(\R^n)$ be the standard cubical structure on $\R^n$ in which every cube is of unit size and corners in the integer lattice $\mathbb Z^n$.

\begin{definition}\label{def:locally-Euclidean}
Let $K$ be a cubical $n$-complex and $S\subset K$ a
subcomplex.

We say that \emph{$K$ is locally Euclidean at $S$} provided that  
 $(\Star_K(S),S)$ is  isomorphic to a pair of subcomplexes $(\Star(P),P)$ in $\sfC(\R^n)$.
\end{definition}

\begin{remark}
Let $Q$ be an $n$-cube, $q$ an $(n-1)$-cube, and $\xi$ an $(n-2)$-cube in $\R^n$ with $n\geq 2$. Then, by Definition \ref{def:star}, $\Star_{\sfC(\R^n)}(Q)$, $\Star_{\sfC(\R^n)}(q)$, and $\Star_{\sfC(\R^n)}(\xi)$, 
contain one, two, and four $n$-cubes, respectively. Thus the same holds for locally Euclidean cubical $n$-complexes.
\end{remark}

\subsubsection{Refinement of cubical complexes}
\label{sec:refinement-metric}

The standard cubical structure $\sfC_n$ of the Euclidean $n$-cube $[0,1]^n$ admits a natural subdivision into $3^n$ congruent $n$-cubes.

\index{cubical complex!refinement} 
For each $v\in \{0,1,2\}^n$, let $q_v = \frac{1}{3}\left( v+[0,1]^n \right)$ and $\iota_v \colon q_v \to [0,1]^n$  be the congruence $x\mapsto 3^n(x-v)$; we denote $\sfC(q_v) = \sfC_{\iota_v}(q_v)$. Then cubes $q_v$, $v\in \{0,1,2\}^n$, are congruent Euclidean cubes which cover $[0,1]^n$, have mutually disjoint interiors, and satisfy $q_v \cap q_{v'} \in \sfC(q_v) \cap \sfC(q_{v'})$ for $v,v'\in \{0,1,2\}^n$. Then
\[
\sfR_n = \bigcup_{v\in \{0,1,2\}^n} \sfC(q_v)
\]
is a cubical $n$-complex with space $[0,1]^n$, and is called the \emph{standard refinement of $\sfC_n$}.

Let now $K$ be a cubical $n$-complex and, for each $Q\in K^{[n]}$, let $\phi_Q \colon Q \to [0,1]^n$ be the fixed cubical isomorphism in the definition of cubical $n$-complex. 
For each $Q\in K^{[n]}$, we take   
\[
\Refine(Q) = \phi_Q^*(\sfR_n).
\]
Since the transition maps for $n$-cubes having a common face $q=Q\cap Q'\in K^{[n-1]}$ are isometries, we have that $\Refine(Q)|_q = \Refine(Q')|_q$. Further, since each $n$-cube $Q'\in \Refine(K)^{[n]}$ is contained in an $n$-cube $Q\in K^{[n]}$, we have $\phi_Q(Q')=q_v$ for some $v\in \{0,1,2\}^n$ and we may define $\phi_{Q'}\colon Q' \to [0,1]^n$ to be the map $\phi_{Q'} = \iota_v \circ \phi_Q|_{Q'}$.
Hence
\[
\Refine(K) = \bigcup_{Q\in K^{[n]}} \Refine(Q)
\]
is a well-defined cubical $n$-complex. We call $\Refine(K)$ is the \emph{standard refinement of $K$}. \index{$\Refine^k(K)$} For $k\ge 1$, we call the cubical $n$-complex $\Refine^k(K) = \Refine(\Refine^{k-1}(K))$ the \emph{$k$th iterated refinement of $K$}. 

For $x, x' \in |\Refine^k(K)|=|K| $, we have 
\[
d_{\Refine^k(K)} (x,x')= 3^k d_K(x,x') 
\]
under the two different metrics. To be consistent in all refinement scales, we equip all refinements $\Refine^k(K)$, $k\geq 1$, with the metric $d_K$.

\begin{convention}[Standard metric on $\Refine^k(K)$]\label{convention:metric-subcomplex}
Let $K$ be a cubical $n$-complex and $d_K$ the polyhedral metric of $K$. We equip $\Refine^k(K)$, for $k\geq1$, with the metric $d_K=3^{-k} d_{\Refine^k(K)}$. We call \emph{$d_K$ the standard metric on $\Refine^k(K)$}.
\end{convention}


\subsection{From simplicial to cubical structure}
\label{sec:from-s-to-c}

We construct now, using the barycentric subdivision, a cubical complex from a barycentric subdivision of a simplicial complex. The general case stems from the case of a single simplex.
Let
$ 
\Delta_n = [e_1,\ldots, e_{n+1}] \subset \R^{n+1}
$
be the standard $n$-simplex.

In the following statement, an $n$-cube $Q$ is a polyhedron with a cubical structure $\sfC(Q)$ isomorphic to $\sfC_n$, and two $n$-cubes $Q$ and $Q'$ are \emph{congruent} if there exists a cubical isomorphism $\psi \colon Q \to Q'$ which is an isometry.

\begin{lemma}
\label{lemma:congruent-cubulation} 
There exists a cubical complex $\Delta_n^\square$ on $\Delta_n$ consisting of  $(n+1)$ congruent $n$-cubes $Q_0,\ldots, Q_n$, for which the restriction $\Delta_n^\square|_\sigma$  to each face $\sigma$ of $\Delta_n$ is a well-defined cubical $(n-1)$-complex consisting of $n$ congruent $(n-1)$-cubes. 
\end{lemma}

\begin{proof}
We denote $\Delta_n= [v_0,v_1,\ldots, v_n]$ and let $X$ be the barycentric subdivision of $\Delta_n$ for which the simplices of the same dimension are congruent.  For each vertex $v_i$ of $\Delta_n$, let $\Star_X(v_i)$ be the star of $v_i$ in the complex $X$ and let $Q_i=|\Star(v_i)|$.

We claim  that each $Q_i$ admits a cubical structure $\sfC(Q_i)$ isomorphic to  $\sfC_n$ on $[0,1]^n$ for which each  $k$-cube in $\sfC(Q_i)$, $0\leq k\leq n$, 
is a union of  $k$-simplices in $X$. The argument is an induction on dimension. 
The claim clearly holds for $n=1$. 

Suppose that the claim holds for dimension $n-1$. For dimension $n$, it suffices to prove the claim for the vertex $v_0$. 

By the induction assumption,  for each $i\neq 0$,  the restriction of  $\Star_X(v_0)$ to the face $[v_0,v_1,\ldots,v_{i-1}, \hat v_i, v_{i+1}, \ldots, v_n]$  admits a cubical structure $K_{0\, i}$ which has one $(n-1)$-cube and whose $k$-cubes, $k\in[0,  n-1],$ are unions of $k$-simplices in $X|_{Q_0}$. Thus $K_0 = K_{0\, 1}\cup \cdots \cup K_{0\, n}$ is the cubical $(n-1)$-complex $\Star_{X^{(n-1)}}(v_0)$.

Let  $v$ be the unique vertex of $X$ in the interior of $\Delta_n$. Let $L_0$ be the link of vertex $v_0$ in $X$ and $\tau_0$ be the face of $\Delta_n$ opposite to  $v_0$. Since $X$ is the barycentric subdivision of $\Delta_n$,  link $L_0$ is isomorphic  to the barycentric subdivision $X|_{\tau_0}$ of $\tau_0$. 
By the induction assumption,  $X|_{\tau_0}$ admits a cubical $(n-1)$-complex consisting of $n$ cubes of dimension $(n-1)$ and  is the star of the unique vertex in the interior of $\tau_0$.
Thus the same holds true for  $L_0$ and $v$. We denote the corresponding cubical $(n-1)$-complex on $|L_0|$ by $K'_0$.

Then the union $K_0 \cup K'_0$ is a cubical complex on the boundary of $\Star_X(v_0)$, obtained by taking unions of simplices in $X$.  Moreover, $|K_0|$ and $|K'_0|$ are $(n-1)$-cells and $\partial Q_0= |K_0\cup K'_0|$ is an $(n-1)$-sphere. Thus  $Q_0$ is an $n$-cell and is the space of the cubical complex $\sfC(Q_0) = K_0\cup K'_0 \cup \{Q_0\}$ which is isomorphic to $\sfC_n$.

We fix, for each $i=1,\ldots, n$, an isometry $\rho_i \colon \Delta_n \to \Delta_n$ which maps $v_i$ to $v_0$.  Then $\rho_i$ fixes the barycenter $v$ of $\Delta_n $ and satisfies $\rho_i(Q_i)=Q_0$. Thus  $Q_i$ admits a standard cubical structure $\sfC(Q_i)$ for which $\rho_i^*(\sfC(Q_i)) = \sfC(Q_0)$. 
The $n$-cubes $Q_0,\ldots, Q_n$ are congruent by definition.  The fact that  $\sfC(Q_i)|_{Q_i\cap Q_j} = \sfC(Q_j)|_{Q_i\cap Q_j}$ for $i\ne j$ follows  from the induction assumption and the Euclidean barycentric subdivision of $\Delta_n$. 

We set $\Delta_n^\square$ to be the union of complexes $\sfC(Q_0),\ldots, \sfC(Q_n)$ and for brevity, denote  complex $\sfC(Q_i)$ by $Q_i$.

It remains to show the compatibility of the standard cubical structures. 
Let $h \colon Q_0 \to [0,1]^n$ be a cubical isomorphism for which $h(v_0) = 0$, $h(v) = e_1 + \cdots +e_n$, and $h$ conjugates 
each isometry of $Q_0$ that fixes $v_0$ and $v$ to an isometry of $[0,1]^n$.
 Set  $h_0 = h \colon Q_0 \to [0,1]^n$ and, for $i>0$, $h_i = h \circ \rho_i \colon Q_i \to [0,1]^n$.

Since $\rho_j(Q_j\cap Q_i)$ and $\rho_i(Q_j\cap Q_i)$ are unions of faces of $Q_0$ in $K_0'$, which differ by an isometry, 
maps $h_i \cap h_j^{-1}|_{h_j(Q_j\cap Q_i)} \colon h_j(Q_j\cap Q_i) \to h_i(Q_j\cap Q_i)$ are Euclidean isometries. \end{proof}

An immediate consequence of the previous lemma is the existence of a cubical complex $T^\square$ on the space $|T|$ of a simplicial complex $T$.

\begin{proposition}
\label{prop:PL-to-cubical}
Let $T$ be a simplicial $n$-complex for which $(|T|, d_T)$ is a Riemannian manifold. Then there exists a cubical complex $K=T^\square$ on $|T|$ for which the restriction $T^\square|_\sigma$ to  each $n$-simplex $\sigma\in T$ is isomorphic to $\Delta_n^\square$. 
Moreover, the polyhedral metric $d_{T^\square}$ and the Riemannian metric $d_T$ are $L(n,T)$-quasisimilar.
\end{proposition}

\begin{proof}
We fix, for each $n$-simplex $\sigma\in T$, an affine simplicial homeomorphism $\psi_\sigma \colon \sigma \to \Delta_n$. Since the transition maps $\psi_{\sigma'} \circ \psi_{\sigma}^{-1}$ are Euclidean isometries on common faces, the pull-back structures $\psi_\sigma^*(\Delta_n^\square)$ and $\psi_{\sigma'}^*(\Delta_n^\square)$ agree on $\sigma\cap \sigma'$. We define 
\[
T^\square = \bigcup_{\sigma \in T^{[n]}} \,\psi_\sigma^*(\Delta_n^\square).
\]

Recall that, by the proof of Lemma \ref{lemma:congruent-cubulation}, 
all $n$-cubes $Q_0,\ldots, Q_n$ in $\Delta_n^\square$ are congruent and that maps $h_i \colon Q_i \to [0,1]^n$ are cubical isomorphisms.  

Let  $\sigma$ be an $n$-simplex in $T$ and  $Q$ be an $n$-cube in $ T^\square|_\sigma$. Let $i_Q \in \{0,1\ldots, n\}$ be the index for which $\psi_\sigma(Q) = Q_{i_Q}$. Set $\phi_Q = h_{i_Q} \circ \psi_\sigma|_Q \colon Q\to [0,1]^n$ for each $Q\in T^\square$. It is now straightforward to check that transition maps $\phi_{Q'} \circ \phi_Q^{-1}$ are Euclidean isometries on common faces.

Note that, for each $n$-simplex $\sigma$ in $T$ and  $n$-cube $Q\in T^\square|_\sigma$,  $\psi_\sigma$ is a quasisimilarity  with a constant depending on $\sigma$ and $h_{i_Q}$ is a bilipschitz map in Euclidean metric with a constant depending on the dimension $n$.
Thus the map $\phi_Q \colon (Q, d_T|_Q)\to [0,1]^n$ is an $L(n,\sigma)$-quasisimilarity.
From this, the $L(n,T)$-quasisimilarity between the two metric follows.
\end{proof}

\begin{proof}[Proof of Proposition \ref{prop:Riemannian-to-cubical}]
A Riemannian $n$-manifold $M$ admits a simplicial triangulation $T$ in which every $n$-simplex is bilipschitz to an affine Euclidean $n$-simplex. The claim now follows from Proposition \ref{prop:PL-to-cubical}.
\end{proof}

\begin{remark}
The cubical complex in Proposition \ref{prop:Riemannian-to-cubical} may be chosen to have additional properties which are suitable for the Wada construction; see Remark \ref{rmk:Riemannian-good}.
\end{remark}

\section{Adjacency graphs}
\label{sec:adjacency-graphs}

In this section, we discuss adjacency graphs of cubical complexes and their subgraphs. We study cubical complexes, called realizations, whose adjacency graphs are given subgraphs. Of particular interest is the class of cut-graphs and their realizations. 
Using adjacency graphs, we introduce at the end of this section notions of tunnels and good complexes. This section may be viewed as a preliminary for the study of separating complexes in Section \ref{sec:Separating-complexes}. 

We begin by recalling that a complex $K$ is \emph{connected} if its space $|K|$ is connected. Similarly, a subcomplex $P\subset K$ is a \emph{component} of $K$ if $|P|$ is connected component of $|K|$

In what follows, we use a stronger notion of connectedness for cubical complexes based on adjacency.

\index{cubical complex!adjacency graph}
\begin{definition}
\label{def:adjacent}
\index{adjacency graph}
Two $n$-cubes $Q$ and $Q'$ in a cubical $n$-complex $K$ are said to be  \emph{adjacent} if they have a common face $Q\cap Q'$ in $ K^{[n-1]}$. 

The pair
\[
\Gamma(K) = \left( K^{[n]}, \left\{ \{Q,Q'\} \colon Q\ne Q',\ Q\cap Q'\in K^{[n-1]}\right\}\right)
\]
is the \emph{adjacency graph of $K$}. We call cubes in $K^{[n]}$ vertices and adjacent pairs $\{Q,Q'\}$ edges.

A cubical $n$-complex $K$ is said to be \emph{adjacently-connected} if its adjacency graph $\Gamma(K)$ is connected.
\end{definition}

Clearly, two $n$-cubes $Q$ and $Q'$ of $K$ may have non-empty intersection without being adjacent. Thus a cubical $n$-complex may have a topologically connected space without being adjacently-connected.

\begin{remark}
Let $K$ be a cubical $n$-complex in which every $(n-1)$-cube is the face of at most two $n$-cubes. Then each edge $ \{Q,Q'\}$ of $\Gamma(K)$ may be identified with the face $Q\cap Q'\in K^{[n-1]}$,  hence the set of edges of $\Gamma(K)$ with a subset of $K^{[n-1]}$. 
\end{remark}

\subsection{Realization of a subgraph}
\label{sec:Realization-subgraph}

Let $K$ be a cubical $n$-complex. Let $G = (V_G,E_G)$ be a subgraph of the adjacency graph $\Gamma(K)$, and let
\[
\Span_K(G) =  \Span_K\left( \{ Q\in K^{[n]} \colon Q\in V_G\}\right)
\]
be the \emph{subcomplex of $K$ spanned by $G$}. 

We define a cubical $n$-complex $\cR_K(G)$, called \emph{the realization of $G$ with respect to $K$}, by \emph{unidentifying} points in $\Span_K(G)$, which belong to two or more $n$-cubes but not lie in an edge of tree $G$.  See Figures \ref{fig:Realization-3} and \ref{fig:Realization-2}. 

For the definition, let 
\[
S_G = \bigsqcup_{Q\in V_G} Q \quad \text{and}\quad F_G = \bigsqcup_{Q\in V_G} \sfC(Q)
\]
be
the disjoint union of the $n$-cubes $Q$ in $G$ and the cubical $n$-complex having $S_G$ as its space, respectively.
Let $\sim_G$ be the equivalence relation in $S_G$ generated by  $x\sim_G x'$ between points $x\in Q$ and $x'\in Q'$ for which $x = x'$ in $|K|$ and $\{Q,Q'\}$ is an edge in $G$.  Denote by $[x]$  the equivalence class of $x\in S_G$, and by $[q] = \{ [x]\colon x\in q\}$ for $q\in F_G$.

\begin{remark}
For $Q\in \Span_K(G)^{[n]}$, the equivalence class $[Q]$ is an $n$-cube. For $q\in \Span_K(G)^{(n-1)}$, the equivalence class $[q]$ is either one cube or 
union of two cubes.
\end{remark}

\begin{definition}
\label{def:Realization}
\index{Realization of a subgraph $\cR_K(G)$}
Let $K$ be a cubical $n$-complex. The \emph{realization $\cR_K(G)$ of a subgraph $G$ of $\Gamma(K)$ (with respect to $K$)}  is the cubical complex
\[
\cR_K(G)= F_G/{\sim_G} = \{ [q] \colon q\in F_G\}.
\]
\end{definition}

\begin{figure}[htp]
\begin{overpic}[scale=1,unit=1mm]{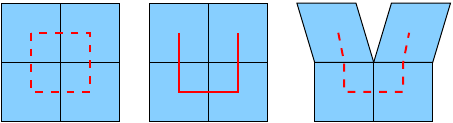} 
\end{overpic}
\caption{Left: complex $K$ and $\Gamma(K)$. Middle: a spanning tree $G \subset \Gamma(K)$ and $\Span_K(G)$. Right: realization $\cR_K(G)$ and $\Gamma(\cR_K(G))$.}
\label{fig:Realization-2}
\end{figure}

\begin{remark}
By the definition, $\Gamma(\cR_K(G))= G$. We want to emphasize the fact that adjacency graph alone does not determine the realization. 
There are cubically non-isomorphic complexes with isomorphic adjacency graphs.
The ambient complex $K$ is crucial in the definition.
\end{remark}

By transitivity of the relation $\sim_G$, the connected components of the realization $\cR_K(G)$ are in one-to-one correspondence with the connected components of $G$.

Let $K$ be a cubical $n$-complex and $G$ a subgraph of $\Gamma(K)$.
Let $\iota_G \colon S_G \to |\Span_K(G)|$ be the inclusion map of $n$-cubes $Q\in V_G$ into $|\Span_K(G)|$ and 
define, with the same symbol, $\iota_G \colon F_G \to \Span_K(G)$  the inclusion map of complexes. By construction, mapping $\iota_G \colon S_G \to |\Span_K(G)|$ factorizes through the quotient map $S_G \to |\cR_K(G)|$, $x\mapsto [x]$, and  mapping $\iota_G \colon F_G \to \Span_K(G)$ through the  quotient $F_G \to \cR_K(G)$, $q\mapsto [q]$. 

Thus we obtain a  well-defined \emph{$G$-quotient maps}, 
\[
\pi_G \colon |\cR_K(G)|\to |\Span_K(G)| \quad \text{and}\,\,\,\, \pi_G \colon \cR_K(G) \to \Span_K(G)
\]
for which the restriction $\pi_G|_Q \colon Q \to \Span_K(G)$ is an isomorphic embedding and the restriction $\pi_G|_{|Q|} \colon |Q| \to |\Span_K(G)|$ a homeomorphic embedding for  $Q\in \cR_K(G)^{[n]}$.
We call the image $\pi_G(Q)$ of a cube $Q\in \cR_K(G)$ the \emph{projection of $Q$}, and the preimage $\pi_G^{-1}(C)$  of a cube $C\in \Span_K(G)$ the \emph{lift of $C$}.

The $G$-quotient $\pi_G \colon \cR_K(G) \to \Span_K(G)$ induces an injective map
\[ \hat{\pi}_G \colon \Gamma(\cR_K(G)) \to \Gamma(\Span_K(G))\]
between adjacency graphs, given by  $Q\mapsto \pi_G(Q)$ on vertices and  $\{Q,Q'\}\mapsto \{ \pi_G(Q),\pi_G(Q')\}$ on  edges. By the definition of  $\cR_K(G)$, graph $\Gamma(\cR_K(G))$ is isomorphic to $G$
  and map $\hat{\pi}_G$ is an isomorphism onto $G \subset \Gamma(\Span_K(G))$. 
 
The inclusion $G\subset\Gamma(\Span_K(G))$ may be proper; see Figure \ref{fig:Realization-2}.
 
\begin{remark}\label{rmk:G-quotient-refinement}
The $G$-quotient 
$\pi_G \colon \cR_K(G) \to \Span_K(G)$ induces a bijection 
between $n$-cubes in $\Refine^k(\cR_K(G))$ and $n$-cubes in $\Refine^k(\Span_K(G))$. 
\end{remark}

In our construction, we consider mainly the realizations $\cR_K(G)$ of subgraphs $G\subset \Gamma(K)$ which are trees; see Section \ref{sec:Realization} on tunnels.

In such cases,  the spaces of the realizations 
are $n$-cells. We prove this fact in the following lemma; see also Figure \ref{fig:Realization-3} for an example.

\begin{lemma}
\label{lemma:Realization-cell}
Let $K$ be a cubical $n$-complex and  $G \subset \Gamma(K)$ be a tree. Then $|\cR_K(G)|$ is an $n$-cell.
\end{lemma}

\begin{proof}
Since $G$ is a tree, so is $\Gamma(\cR_K(G))=(V,E)$. Let $Q\in \cR_K(G)$ be an $n$-cube for which $|Q|$ is a leaf of $\Gamma(\cR_K(G))$, that is, there is a unique $n$-cube $Q'\in \cR_K(G)^{[n]}$ for which $\{Q, Q'\}\in G$. 
Let $R'_G =\Span(\cR_K(G)- \sfC(Q))$ be the subcomplex with $\sfC(Q)$ removed.

We claim that $Q$ does not meet any other leaf in $\cR_K(G)$. Otherwise, let $Q''\in \cR_K(G)$ be another leaf that has nonempty intersection with $Q$. By transitivity of the relation $\sim_G$, there exists a chain $Q=Q_0,\ldots, Q_m = Q''$ of adjacent $n$-cubes in $\cR_K(G)$ satisfying $Q\cap Q'' \subset Q_i$ for each $i=0,\ldots, m$. In particular, $Q \cap Q''$ is contained in an $(n-1)$-cube $Q_0\cap Q_1 \in \cR_K(G)$. Since $Q'$ is the only $n$-cube in $\cR_K(G)$ adjacent to $Q$, we have that $Q=Q_0$ and $Q'=Q_1$.

Thus there exists a PL homeomorphism $h_Q \colon |\cR_K(G)| \to  |R'_G|$ for which  $h_Q|_q=\id$ for all $n$-cubes $q\in \cR_K(G)$, $q \neq Q, Q'$. By removing leaves of $G$ iteratively by homeomorphisms, we may reduce $\cR_K(G)$ to an $n$-cube, which is an $n$-cell.
Hence $|\cR_K(G)|$ is an $n$-cell.
\end{proof}

\subsubsection{Cut-graphs}
\label{sec:cut-graphs}
\index{cubical complex!cut-graph}

A particular class of subgraphs of $\Gamma(K)$ is obtained by cutting $\Gamma(K)$ with a codimension one subcomplex $Z$ of $K$.

\begin{definition}
\label{def:complementary-graph}
\index{adjacency graph!cut-graphs} \index{$\Gammacut(K;Z)$}
Let $K$ be a cubical $n$-complex and $Z \subset K$ be a cubical $(n-1)$-subcomplex. The \emph{cut-graph $\Gammacut(K;Z)$ of the adjacency graph $\Gamma(K)$ relative to $Z$} is 
the subgraph 
\[
\Gammacut(K;Z) = \left( K^{[n]}, \left\{\{ Q\cap Q'\} \in \Gamma(K) \colon Q\cap Q'\not \in Z \right\} \right).
\]
We call the realization $\cR_K(\Gammacut(K;Z))$ 
the \emph{lift of $K$ relative to $Z$}, and the
$\Gammacut(K;Z)$-quotient map
 \[\pi_{(K;Z)} \colon \cR_K(\Gammacut(K;Z))  \to K\]
 \emph{the canonical projection from the lift onto $K$}.
\end{definition}

A component $G$ of $\Gammacut(K;Z)$ is called an \emph{inner component} if $\Span_K(G)$ does not meet $\partial K$, and  an \emph{outer component} if it is not an inner component.

\subsubsection{$\Sigma$-components of cut-graphs and their lifts}\label{sec:notation-lift}
In this article, we are mainly interested in the case when $Z$ is in the interior of $K$. 

Suppose that  $Z$ is in the interior of  $K$ and $\Sigma\subset \partial K$ is a boundary component.
We denote by $\Gammacut(K;Z;\Sigma)$  the unique connected component of $\Gammacut(K;Z)$ for which
\[
\Sigma \subset \Span_K(\Gammacut(K;Z;\Sigma)) \subset K
\]
and call the span
\[
\Comp_K(Z;\Sigma) = \Span_K(\Gammacut(K;Z;\Sigma)),
\]
the \emph{$\Sigma$-component of $K$}.\index{cubical complex!$\Sigma$-component} 

\begin{remark}
In general, $Z$ need not separate the components of $\partial K$. Thus it can happen that $\Gammacut(K;Z;\Sigma)=\Gammacut(K;Z;\Sigma')$, hence $\Comp_K(Z;\Sigma)= \Comp_K(Z;\Sigma')\supset \Sigma \cup \Sigma'$, for two distinct boundary components  $\Sigma$ and $\Sigma'$. 
Note also that, if the cut-graph
 $\Gammacut(K;Z)$ has an inner component, 
then the union $\bigcup_\Sigma \Comp_K(Z;\Sigma)$ over all boundary components $\Sigma$ is  strictly contained in $K$. These two observations lead us to the discussion of separating complexes in Section \ref{sec:Separating-complexes}.
\end{remark}

When $Z$ is in the interior of  $K$, we denote
\[\Real_K(Z; \Sigma) = \cR_K(\Gammacut(K;Z;\Sigma)).\]
Since $\pi_{(K;Z)}$ projects $\Real_K(Z;\Sigma)$ onto $\Comp_K(Z;\Sigma)$, we may also write
\[\Real_K(Z;\Sigma)= \pi^{-1}_{(K;Z)}(\Comp_K(Z;\Sigma)),\]
and call it the \emph{lift of $\Comp_K(Z;\Sigma)$ in $\cR_K(\Gammacut(K;Z))$}.

\begin{remark}\label{rmk:metric-Realization}
The $\Sigma$-component $\Comp_K(Z;\Sigma)$, being a subcomplex of $K$, inherits a polyhedral metric from $K$.  
Similarly, cubical $n$-complex $\Real_K(Z;\Sigma)$ has a metric in which each $n$-cube $Q$ is isomorphically isometric to its projection $\pi_{(K;Z)}(Q)$ in $\Comp_K(Z;\Sigma)$.
Under these metrics, the restriction of 
\[\pi_{(K;Z)}|_{\Real_K(Z;\Sigma)} \colon |\Real_K(Z;\Sigma)| \to |\Comp_K(Z;\Sigma)|\]
 is a length preserving map, that is, for each path $\gamma$ in $|\Real_K(Z;\Sigma)|$, 
we have
$\ell(\pi_{(K;Z)} \circ \gamma) = \ell(\gamma)$. 
\end{remark}

Since interiors of $|\Real_K(Z;\Sigma)|$ and $|\Comp_K(Z;\Sigma)|$ are locally Euclidean metric spaces with respect to the inherited metrics, the previous remark immediately yields that  map $\pi_{(K;Z)}$ is conformal in the interior of $|\Real_K(Z;\Sigma)|$. We record this fact as follows.

\begin{lemma}\label{lemma:metric-Realization}With respect to the metrics above, interiors $\interior (|\Real(Z;\Sigma_i)|)$ and $\interior (|\Comp(Z;\Sigma_i)|\setminus |Z|)$ are conformally equivalent.
\end{lemma}
 
We end this section with some more notations. Let  $P$ be an $n$-subcomplex of $K$, we may denote for simplicity 
\[\pi_{(K;Z)}^{-1}(P)= \Span_{\cR_K(\Gammacut(K;Z))}(\pi_{(K;Z)}^{-1}(P^{[n]})).\]

Clearly the $\Gammacut(K;Z)$-quotient map $ \pi_{(K;Z)}$ induces a natural isomorphism $\Refine^k(\cR_K(\Gammacut(K;Z)))  \to \Refine^k(K)$ between refinements. 
Thus, given an $n$-subcomplex $R$ of $\Refine^k(K)$, we may also denote 
\[\pi_{(K;Z)}^{-1}(R)= \Span_{\Refine^k(\cR_K(\Gammacut(K;Z)))}(\pi_{(K;Z)}^{-1}(R^{[n]})).\]

\subsection{Tunnels}
\label{sec:Realization}

In this section, we define a special class of cubical complexes called tunnels.

\begin{definition}
\label{def:tunnel}
\index{tunnel}
A cubical $n$-complex $K$ is a \emph{tunnel} if its adjacency graph $\Gamma(K)$ is a tree and $|K|$ is an $n$-cell. 

A cubical $n$-complex $K$ is \emph{tunnel-like with respect to an $(n-1)$-subcomplex $Z\subset K$} if $\Gammacut(K;Z)$ is a tree. 
\end{definition}

\begin{remark}
The term of tunnel-like complexes is justified by Lemma \ref{lemma:Realization-cell}. Indeed, if $K$ is tunnel-like with respect to a subcomplex $Z$, then the realization $\cR_K(\Gammacut(K;Z)$ is a tunnel.
\end{remark}

\begin{figure}[htp]
\begin{overpic}[scale=0.8,unit=1mm]{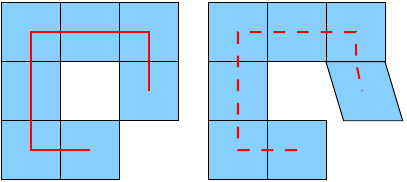} 
\end{overpic}
\caption{Left: not a tunnel. Right: a tunnel}
\label{fig:Realization-3}
\end{figure}

The importance of tunnels stems, in part, from the following contraction property.
This fact is used throughout the proof of Theorem \ref{intro-thm:Wada-Riemannian-manifold}.

\begin{proposition}[Tunnel-contracting]
\label{prop:tunnel-contracting}
Let $K \cup T$ be a cubical $n$-complex for which $T$  is a tunnel and the intersection $q_T=K\cap T$ is the common face of exactly two $n$-cubes.
Then there exist a constant $L=L(n,\# T^{[n]})\ge 1$ and a piecewise linear $L$-bilipschitz homeomorphism $\phi_T \colon |K\cup T| \to |K|$ for which $\phi_T$ is the identity on $|K-Q_T|$, where $Q_T$ is the unique $n$-cube in $K$ having $q_T$ as a face.
\end{proposition}

\begin{proof}
Let $T'=T\cup Q_T$. Then $\Gamma(T')$ is a tree. 

Let $Q\in T', Q\neq Q_T,$ be  a leaf in $\Gamma(T')$ and  $Q'$ be the unique $n$-cube in $T'$ sharing a face $q=Q\cap Q'$ with $Q$. Since $T'$ is isomorphic to the realization of $\Gamma(T')$, we have that $Q\cap (T'-Q) = Q\cap Q'$. Thus there exists a piecewise linear homeomorphism $\varphi_Q \colon |Q\cup Q'| \to |Q'|$, which is the identity on $|\partial Q'| \setminus |Q|$
and whose bilipschitz constant depends only on $n$. We extend $\varphi_Q$, by identity, to a homeomorphism
$\varphi_Q \colon |K\cup T| \to |(K\cup T) -Q|$. 
The mapping $\phi_T$ in the proposition is obtained as a composition of such homeomorphisms by contracting one leaf at a time. Thus the bilipschitz constant for $\phi_T$ depends only on the dimension $n$ and the size of the tree $\Gamma(T)$, which is $\#T^{[n]}$. 
\end{proof}

\subsection{Good cubical complexes}
\label{sec:good-complexes}

The cubical complex $K$ associated to a Riemannian manifold defined in Proposition \ref{prop:Riemannian-to-cubical} has certain properties which are needed in our construction. We collect these properties in the notion of good complexes.

\begin{definition}
\label{def:good-complex} \index{good complex}
A  cubical $n$-complex $K$ is a \emph{good complex} if 
\begin{enumerate}
\item $\Gamma(K)$ is connected,  \label{item:good-complex-adjacency-1}
\item each $(n-1)$-cube is a face of at most two $n$-cubes,  \label{item:good-complex-regularity-a}
\item star of each vertex and each $(n-2)$-cube is an adjacently-connected $n$-complex, and   \label{item:good-complex-regularity-b}
\item $\partial K$ is a cubical $(n-1)$-complex with at least one component, and each component $\Sigma$ of $\partial K$ is adjacently-connected.\label{item:good-complex-adjacency-2}
\end{enumerate}
A  cubical $n$-complex $K$ is said to have  a \emph{good interior} if Conditions \eqref{item:good-complex-adjacency-1},  \eqref{item:good-complex-regularity-a}, and  \eqref{item:good-complex-regularity-b}  are satisfied.
\end{definition}

We show next that by 
\eqref{item:good-complex-regularity-a}, the adjacency in  \eqref{item:good-complex-regularity-b} improves to cyclicity. For the statement, we say that a graph $G=(V,E)$ is \emph{cyclic} if it is isomorphic to either
\[
(\{1,\ldots, m\}, \{ \{1,2\}, \{2,3\},\ldots, \{m-1,m\}\})
\]
or 
\[
(\{1,\ldots, m\}, \{ \{1,2\}, \{2,3\},\ldots, \{m-1,m\}, \{m,1\}\}).
\]
In the first case, we say also that $G$ is \emph{linear}.

\begin{lemma}
\label{lemma:star-cyclic}
Let $K$ be a good cubical $n$-complex. Then the adjacency graph $\Gamma(\Star_K(\xi))$ of the star of an $(n-2)$-cube $\xi\in K^{[n-2]}$ is cyclic.
\end{lemma}

\begin{proof}
Let $\xi\in K^{[n-2]}$. By condition \eqref{item:good-complex-regularity-b} in Definition \ref{def:good-complex}, the adjacency graph $\Gamma(\Star_K(\xi))$ is connected. 
Let $Q\in \Star_K(\xi)$ be an $n$-cube. Then $\xi$ is intersection of exactly two faces of $Q$. Thus $Q$ is adjacent to at most two other $n$-cubes in $\Star_K(\xi)$ by condition \eqref{item:good-complex-regularity-a}. Hence $\Gamma(\Star_K(\xi))$ is connected and has valence at most two at each vertex. We conclude that $\Gamma(\Star_K(\xi))$ is cyclic. 
\end{proof}

The refinement of a good complex remains to be good; we omit the simple argument.
\begin{lemma}
\label{lemma:refinement-strongly-connected-good-complex}
If $K$ is a good cubical $n$-complex then  $\Refine(K)$ is a good cubical $n$-complex.
\end{lemma}

The cubical structure $T^\square$ in Proposition \ref{prop:PL-to-cubical} is a good complex.

\begin{proposition}
\label{prop:T-square-good}
Let $T$ be a simplicial $n$-complex for which $(|T|, d_T)$ is a Riemannian manifold with boundary. 
Then the cubical structure $T^\square$ associated to $T$ constructed in Proposition \ref{prop:PL-to-cubical} is a good complex.
\end{proposition}

\begin{proof}
Since $|T|$ is a connected manifold with boundary, the simplicial complex $T$ satisfies the simplicial counterpart of the conditions in Definition \ref{def:good-complex}. Since  $\Delta_n^\square$ is a good cubical complex, complex $T^\square$ is also good from the construction.
\end{proof}

The complex $T^\square$ remains good after a collar of the boundary is added. We  state this straightforward fact as a lemma.

\begin{lemma}
\label{lemma:cubical-starting-point}
Let $T$ be a simplicial complex whose space $|T|$ is a connected $n$-manifold with boundary. Let $\iota \colon \partial T^\square \to \partial T^\square \times [0,1]$ be the cubical inclusion $q \mapsto q\times \{0\}$, and let $K$ be the cubical $n$-complex
\[
K = T^\square \, {\bigcup}_\iota \left(\partial T^\square \times [0,1]\right)
\]
obtained by identifying points $q$ and $\iota(q)$ for $q\in \partial T^\square$. Then $K$ is a good cubical $n$-complex. Moreover, $(|T^\square|,d_{T^\square})$ and $(K,d_K)$ are bilipschitz equivalent for a constant depending only on $T$. 
\end{lemma}

\begin{remark}\label{rmk:Riemannian-good}
By Lemma \ref{lemma:cubical-starting-point}, we may assume that the cubical $n$-complex $K$ in Proposition \ref{prop:Riemannian-to-cubical} is  a good cubical $n$-complex for which the star $\Star_K(\partial K)$ of the boundary is isomorphic to $\partial K \times [0,1]$. 
Moreover complex $K$, equipped with a polyhedral metric $d_K$, is  quasisimilar to the original Riemannian manifold $M$.
\end{remark}


\section{Separating complexes}
\label{sec:Separating-complexes}
At the center of our construction is the notion of separating complexes. 
Roughly, a separating complex $Z$ in a good cubical $n$-complex $K$ with boundary is an $(n-1)$-dimensional subcomplex whose complement in $|K|$ is homeomorphic to $|\partial K| \times [0,1)$. Typically the space of a separating complex $Z$ is not a manifold. A formal definition stronger than the topological description above uses the notion of lifts. 
We refer to Section \ref{sec:adjacency-graphs} for definitions.

\begin{definition}[Separating complex]
\label{def:separating-complex}
Let $K$ be a good cubical $n$-complex for which  $\Star_K(\partial K)$ is isomorphic to $\partial K \times [0,1]$. An $(n-1)$-subcomplex $Z \subset K$ is a \emph{separating complex in $K$} if 
\begin{enumerate}
\item $Z\cap \Star_K(\partial K) = \emptyset$; \label{item:separating-no-boundary}
\item $\Gamma(Z)$ is connected;   \label{item:separating-strongly-connected}
\item for each boundary component $\Sigma \subset \partial K$, $\Comp_K(Z;\Sigma)\cap \partial K=\Sigma$, and the union of all $\Sigma$-components is $K$, i.e.~$\bigcup_{\Sigma} \Comp_K(Z;\Sigma)=K$; \label{item:separating-total}
\item for each boundary component $\Sigma \subset \partial K$,  the space of  lift $\Real_K(Z;\Sigma)$ is homeomorphic to $|\Sigma|\times [0,1]$.    \label{item:separating-mu-1}
 \end{enumerate}\index{separating complex} 
\end{definition}

\begin{remark}
\label{rmk:separating}
In general, spaces $|\Comp_K(Z;\Sigma)|$ and $|\Real_K(Z;\Sigma)|$ are not homeomorphic, even the interiors of $|\Comp_K(Z;\Sigma)|$ and $|\Real_K(Z;\Sigma)|$ need not be homeomorphic. 
Indeed, in Figure \ref{fig:Separating_complex}, the cut-graph $\Gammacut(K;Z)$ has four connected components and, for each $i=1,\ldots, 4$, $\Real_K(Z;\Sigma_i)$  has two boundary components. The subcomplex $\tau\subset Z$ in the figure, consisting of four $1$-cubes, is contained in $\Comp_K(Z;\Sigma_1)$ but not in $\partial (\Comp_K(Z;\Sigma_1))$. Thus the cubical complex $\Comp_K(Z;\Sigma_1)$ has three boundary components.

Condition \eqref{item:separating-mu-1}, however, implies a weaker property
\begin{align*}
|\Comp_K(Z;\Sigma_1)|\setminus |Z| 
&\approx \pi_{(K;Z)}^{-1}(|\Comp_K(Z;\Sigma_1)|\setminus |Z|) \\
&= |\Real_k(Z;\Sigma_1)| \setminus \pi_{(K,Z)}^{-1}(|Z|) 
\approx |\Sigma_1| \times [0,1).
\end{align*}
\end{remark}

\begin{figure}[h!]
\begin{overpic}[scale=.45,unit=1mm]{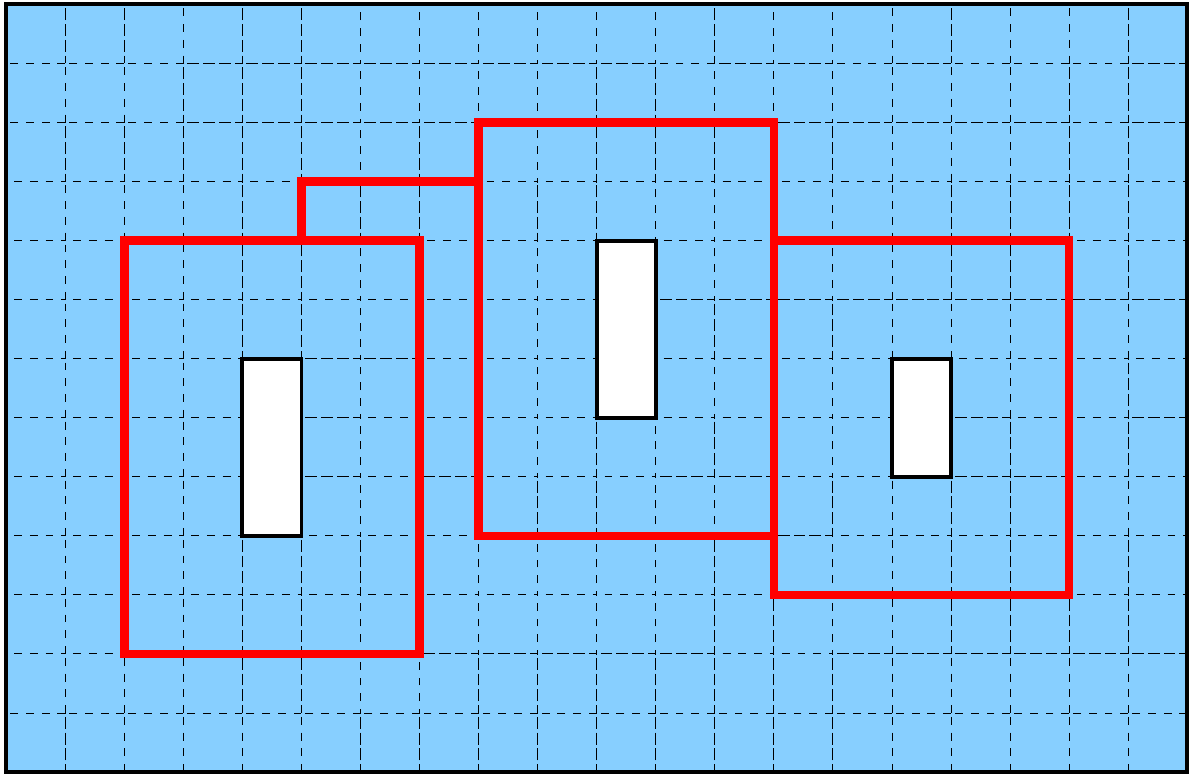}
\put(-4,50){$\Sigma_1$}
\put(14,27){$\Sigma_2$}
\put(41,37){$\Sigma_3$}
\put(63.5,28){$\Sigma_4$}
\put(25,46.5){$\tau$}
\end{overpic}
\caption{Separating complex $Z$ (in red) in a complex $K$ (in blue).}
\label{fig:Separating_complex}
\end{figure}
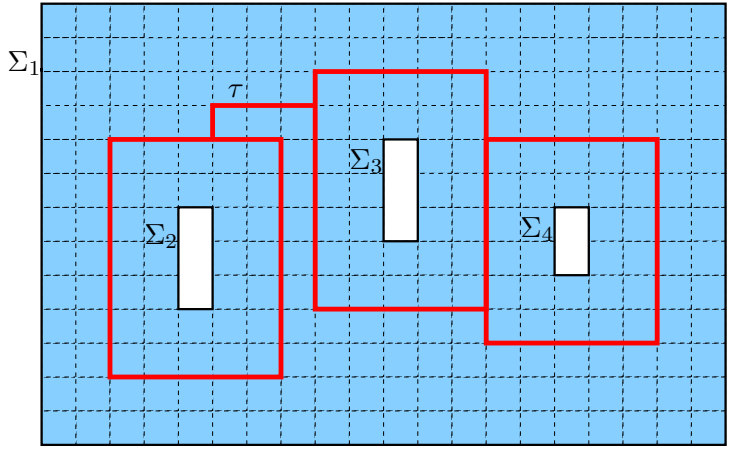

\begin{remark}
By \eqref{item:separating-no-boundary}, $Z$ does not meet the star $\Star_K(\partial K)$, hence   $\Gamma(\Star_K(\Sigma)) \subset \Gammacut(K;Z, \Sigma)$. Thus $\Star_K(\Sigma)$ lifts isomorphically into $\Real_K(Z;\Sigma)$, i.e., 
\[
\pi_{(K;Z)}|_{\pi_{(K;Z)}^{-1}(\Star_K(\Sigma))} \colon \pi_{(K;Z)}^{-1}(\Star_K(\Sigma)) \to \Star_K(\Sigma)
\]
is an isomorphism. 
From here on, we assume that the star $\Star_K(\Sigma)$ and its lift $\pi_{(K;Z)}^{-1}(\Star_K(\Sigma))$ are so identified.
\end{remark}

\begin{remark}\label{rmk:lift-good}
Let $Z$ be a separating complex in a good cubical $n$-complex $K$ and $\Sigma \subset \partial K$ be a boundary component. Then $\Real_K(Z;\Sigma)$ is also a good complex. We omit the straightforward verification.
\end{remark}

Separating complexes are stable under refinement, we record this fact as a lemma and omit a straightforward proof.

\begin{lemma}
\label{lemma:Separating-complex-trivial}
If $Z\subset K^{(n-1)}$ is a separating complex in a cubical $n$-complex $K$, then $\Refine(Z)$ is a separating complex in $\Refine(K)$.
\end{lemma}

\subsection{Existence of separating complexes}
\label{sec:separating-complex-existence}

The following theorem shows that, after refinement, every good complex $K$ admits a separating complex of particular form. Indeed, there exists a separating complex $Z$ in $\Refine(K)$ having the properties that $Z\subset \Refine(K - \Star_K(\partial K))$ and $\Refine(K-\Star_K(\partial K))$ is tunnel-line with respect to $Z$. We return to the structures of the subcomplexes separated by a separating complexe and their lifts in Section \ref{sec:evolution-seq} and, in anticipating that discussion, we formulate the statement in terms of the complementary components of $Z$.

\begin{theorem}
\label{theorem:separating-complex-existence}
Let $n\geq 2$, $K$ be a good cubical $n$-complex for which $\Star_K(\partial K)$ is isomorphic to $\partial K \times [0,1]$, and $\Sigma' \subset \partial K$ a boundary component.
Then there exists 
a separating complex $Z$ of $\Refine(K)$ having the following properties:
\begin{enumerate}
\item for each boundary component $\Sigma\subset \partial K$, $\Sigma\ne \Sigma'$, the $\Sigma$-component $\Comp_{\Refine(K)}(Z;\Sigma) = \Refine(\Star_K(\Sigma))$, and the lift $\Real_{\Refine(K)}(Z;\Sigma)$ is isomorphic to $\Refine(\Star_K(\Sigma))$.
\label{item:separating-1}
\item for the boundary component $\Sigma'$, 
 \[\Comp_{\Refine(K)}(Z;\Sigma') =  \Refine(\Star_K(\Sigma'))\, \bigcup \, (\Refine(K) - \Star_K(\partial K)),\]
and
\[ 
\Real_{\Refine(K)}(Z;\Sigma') =\Refine(\Star_K(\Sigma')) \cup \cR_{\Refine(K)}(\tau),
\]
where  $\tau$ is a spanning tree in $\Gamma(\Refine(K - \Star_K(\partial K)))$, $\cR_{\Refine(K)}(\tau)$ is the realization of $\tau$ and $\Refine(\Star_K(\Sigma')) \cap \cR_{\Refine(K)}(\tau)$ is an $(n-1)$-cube. \label{item:separating-2}
\end{enumerate}
\end{theorem}

\begin{proof}
Let $\Sigma_1,\ldots, \Sigma_m$ be the boundary components of $\partial K$ and  $\Sigma'=\Sigma_1$. From the assumption, stars $\Star_K(\Sigma_i)$ are necessarily mutually disjoint and each $\Gamma(\Star_K(\Sigma_i))$ is connected.

Let  $A'$ be the subcomplex of $K$ obtained by removing all stars $\Star_K(\Sigma_i)$, that is, 
\[
A' =  K - \left( \Star_K(\Sigma_1)\cup \cdots \cup \Star_K(\Sigma_m)\right),
\]
recalling the difference of two complexes defined in Section \ref{sec:From-Riem-To-Cubical}. Since  $\Gamma(K)$ is connected and the stars $\Star_K(\Sigma_i)$ are mutually disjoint, the complex $A'$ is adjacently-connected. 
Let $B' = \Star_K(\Sigma_1)$ and
\[
K' = B' \cup A'.
\]
Then $K'$ is adjacently-connected and $B'\cap A' =\partial(\Star_K(\Sigma_1))\setminus \Sigma_1$.    

Let $T\subset \Gamma(A')$ be a spanning tree and let 
\[
\Omega = \{c(q) \in \Refine(q) \colon q\in (A')^{[n-1]}\,\,\, \text{is an edge in}\,\, T\},
\]
where $c(q)$ is the center cube in $\Refine(q)$, be the collection of the center cubes of the edges in $T$. Let 
\[Y= \Span_{\Refine(K)} \left(\Refine((A')^{(n-1)}) \setminus \, \Omega \right).\]

Fix next a spanning tree $\tau$ in $\Gamma (\Refine(A'))$ which is contained in the cut-graph $\Gammacut (\Refine(A');Y)$ and contains all faces of the center cube $c(Q)$ of $\Refine(Q)$ for all $Q\in (A')^{[n]}$.
Fix also an $n$-cube $Q'\in (A')^{[n]}$ having a face $q'$ in $\Star_K(\Sigma_1)$ and denote by $\omega=c(q')$ the center cube in $\Refine(q')$. 

We claim that
\[
Z=\Span_{\Refine(K)} \left( (\Refine(A'))^{[n-1]} \setminus (\tau  \cup \{\omega\})\right)
\] 
is a separating complex in $\Refine(K)$. 

Since $Z\subset \Refine(A')$, the intersection $Z\cap \Star_{\Refine(K)}(\Refine(\partial K)) =\emptyset$. Condition \eqref{item:separating-no-boundary} in the definition of separating complex holds.

Since the $(n-1)$-complex $(A')^{(n-1)}$ is adjacently-connected, the refinement $\Refine((A')^{(n-1)})$ is adjacently-connected. After the collection $\Omega$ of $(n-1)$-dimensional center cubes is removed from  $\Refine((A')^{(n-1)})$, the complex $Y$ remains adjacently-connected. From the selection of $\tau$, any $(n-1)$-cube in $Z$ which does not belong to $Y$ is adjacent to an $(n-1)$-cube in $Y$. Hence $Z$ is adjacently-connected and Condition
\eqref{item:separating-strongly-connected} holds.

From the construction, the cut-graph 
\[\Gammacut(\Refine(K');Z)=\Gamma(\Refine(B')) \cup \tau \cup \{\omega\}\] 
is connected. Thus the connected components of $\Gamma(\Refine(K);Z)$ are in one-to-one correspondence with the components of $\partial K$. Hence Condition \eqref{item:separating-total}  holds.

Since $\Gamma(\Real_{\Refine(K)}(Z;\Sigma_1))$ is isomorphic to $\Gammacut(\Refine(K');Z)$,
the lift has an essential partition 
\[\Real_{\Refine(K)}(Z;\Sigma_1)= B \cup  A, \] 
 where $B=\pi^{-1}_{(K;Z)}(B')$ is isomorphic to, hence identified with,  $\Refine(\Star_K(\Sigma_1))$, $A=\cR_{\Refine(K)}(\tau)$ is a realization of a tree, hence a tunnel having an $n$-cell as its space, and 
$B \cap A = \{\omega\}$.
Finally, for $i=2,\ldots, m$,  the adjacency graph of  $\Real_{\Refine(K)}(Z;\Sigma_i)$ is isomorphic to $\Gamma(\Refine(\Star_K(\Sigma_i)))$. 
Thus Condition  \eqref{item:separating-mu-1} holds also. 

Hence $Z$ is a separating complex in $\Refine(K)$ having the properties stated in the theorem.
\end{proof}

The existence of a separating complex immediately yields a counterpart of Theorem \ref{intro-thm:Wada-Riemannian-manifold} for manifolds, whose boundary has 
at most one component.

\begin{corollary}
\label{cor:one_none_boundary}
Let $n\ge 3$, $m\in \{0,1\}$, and let $M$ be a compact connected Riemannian $n$-manifold with $m$ boundary components. Then there exist a constant $\sK=\sK(n,M)>1$ and a continuum $X\subset \interior M$, which has no interior points, for the following:
\begin{enumerate}
\item For $m=1$, $M\setminus X$ is $\sK$-quasiconformal to $\partial M \times [0,1)$ and a John domain in $M$. \label{item:m=1}
\item For $m=0$, $M\setminus X$ is $\sK$-quasiconformal to the Euclidean $n$-ball $B^n(0,1)$ and  
a John domain in $M$. \label{item:m=0}
\end{enumerate}
\end{corollary}

\begin{proof}
Let $K$ be a good complex, as in Remark \ref{rmk:Riemannian-good}, having $M$ as its space.

Suppose that $m=1$, and let $Z$ be a separating complex of $\Refine(K)$ chosen in the proof of Theorem \ref{theorem:separating-complex-existence}. Then, by Remark \ref{rmk:metric-Realization}, $|K|\setminus |Z|$ is conformal to its lift and, by Tunnel contracting proposition \ref{prop:tunnel-contracting}, the lift of $|K|\setminus |Z|$ is bilipschitz homeomorphic to $|\partial K|\times [0,1)$ for a constant depending only on $K$ and $n$. Let $X=|Z|$. Thus $M\setminus X$ is $\sK$-quasiconformal to $\partial M \times [0,1)$. The John property of the components of $|K|\setminus |Z|$ follows from the finiteness of the complex $K$.

For $m=0$, let $Q\in K$ be an $n$-cube and $Q'\in \Refine(Q)$ the $n$-cube, which does not meet the boundary of $Q$. Let $K' = \Refine(K)-Q'$. Then $|K'|$ an $n$-manifold with one boundary component and $K'$ is a good complex on $|K'|$. Let $Z$ be a separating complex of $\Refine(K')$ chosen as in Theorem \ref{theorem:separating-complex-existence}. Then $|K'|\setminus |Z|$ is quasiconformal  to $|\partial Q|\times [0,1)$. Thus $M\setminus |Z|$ is quasiconformal to an open $n$-cube $(-1,1)^n$ with a distortion constant depending only on $M$ and $n$. The John property again follows  from the finiteness of the complex $K$.
\end{proof}

\subsection{Statement of the Evolution Theorem}
\label{sec:statement}

In this section we state an Evolution Theorem for separating complexes (Theorem 
\ref{theorem:evolution-short}) which reduces the proof of the main theorem (Theorem \ref{intro-thm:Wada-Riemannian-manifold}) to an iterative construction of a sequence of separating complexes $(Z_k)$, with controlled geometry, in refinements of the good complex $K$.
We begin by associating two quantities to the good cubical complex $K$, namely, the local multiplicity of a cubical complex and the refinement scale.

\begin{definition}
\label{def:mu}
\index{local multiplicity of a complex}
\index{$\mu(K)$}
The \emph{local multiplicity of a cubical $n$-complex $K$} is
\[
\mu(K) = \max\left\{\#\{q\in K^{[n-1]}\colon e\in q^{[n-2]}\} \colon e\in K^{[n-2]}  \right\}.
\]
\end{definition}

\begin{definition}
\label{def:nu}
\index{refinement scale $\nu$}
\index{$\nu(K)$}
The \emph{refinement scale $\nu(K)\ge 1$ of a good cubical $n$-complex $K$ with $m$ boundary components} is the smallest integer $\nu\ge 1$ satisfying 
\[
3^\nu \ge 3^{10}m \mu(K)^2.
\]
\end{definition}

We now describe the specific properties that we are aiming at in the Evolution Theorem.

Suppose from here on that  $K$ is a good cubical $n$-complex for which $\Star_K(\partial K)$ is isomorphic to $\partial K \times [0,1]$, $Z$ is a separating complex in $K$, and $\nu=\nu(K)$ is the refinement scale.

The relative Wada property describes how the complementary components of the separating complexes are progressively intertwined.

\begin{definition}[Relative Wada property] 
\label{def:Wada}   
\index{Wada property!relative}
A separating compilex $Z'$ in $\Refine^{\nu}(K)$ has a \emph{relative Wada property with respect to a separating complex $Z$ of $K$} provided that,
for each $q \in Z^{[n-1]}$, the union $\Refine^\nu(Q\cup Q^*)$ of the $n$-cubes  $Q$ and $Q^*$ sharing the face $q$ 
contains at least one $n$-cube in each $ \Comp_{\Refine^{\nu}(K)}(Z'; \Sigma)$, where $\Sigma$ is a component of $\partial K$.
\end{definition} 

The second is a core-expanding property. For each component $\Sigma \subset \partial K$, the \emph{core of $\Comp_K(Z;\Sigma)$} is defined to be
\[
\Core_K(Z;\Sigma) = \Comp_K(Z; \Sigma) - \Star_K(Z) \subset K. 
\]
Since $Z\cap \Star_K(\partial K)=\emptyset$,  $\Core_K(Z;\Sigma)$ has an \emph{outer boundary component} $\Sigma$ and an \emph{inner boundary component} 
\[(\partial \Core_K(Z;\Sigma))_{\mathrm{inner}}=\partial \Core_K(Z;\Sigma) - \Sigma = \Core_K(Z;\Sigma) \cap \Star_K(Z).\]

\begin{definition}[Core-expanding property]
\label{def:core-expanding}
A separating complex $Z'$ in $\Refine^\nu(K)$ is \emph{core-expanding with respect to a separating complex $Z$ of $K$} if 
\[
|\Core_K(Z;\Sigma)| \subset |\Core_{\Refine^{\nu}(K)}(Z';\Sigma)|\setminus |(\partial \Core_K(Z';\Sigma))_{\mathrm{inner}}|
\]
for each component $\Sigma \subset \partial K$. A sequence $(Z_k)$ of separating complexes, $Z_k \subset \Refine^{\nu k}(K)$, is  said to be \emph{core-expanding} if $Z_{k+1}$ is core-expanding with respect to $Z_k$,\ for each $k$.\index{core!expanding}
\end{definition}

\begin{remark}\label{rmk:cores}
Since  $\Core_K(Z;\Sigma)$ lifts isomorphically into $\Real_K(Z;\Sigma)$, we identify $\Core_K(Z;\Sigma)$ with its lift and consider, from now on, $\Core_K(Z;\Sigma)$  as a subspace of $\Real_K(Z;\Sigma)$.
\end{remark}

The third property quantifies the trading of cubes between complementary components of a separating complex discussed in the introduction.

\begin{definition}[$\lambda$-perturbation]
\label{def:indented-perturbation}
\index{$\lambda$-tame}
A separating complex $Z'$ in a complex $K'=\Refine^{\nu}(K)$  is an \emph{$\lambda$-perturbation of a separating complex $Z$ in $K$} if  
there exist a number $\lambda \ge 1$ and essentially disjoint subcomplexes $T_\Sigma \subset \Refine^\nu(\Star_K(Z))$, where $\Sigma$ is a component of $\partial K$,
for which $T=\bigcup_{\Sigma} T_\Sigma$ has the following properties:
\begin{enumerate}
\item for each $\Sigma$, $\Comp_{K'}(Z';\Sigma) = \left( \Refine^\nu(\Comp_K(Z;\Sigma))-T \right) \bigcup T_\Sigma$, and \label{item:tame-perturbation-1}
\item for each $\Sigma$, $\pi_{(K';Z')}^{-1}(T_\Sigma)$ is a union of disjoint tunnels, each of which has at most $\lambda$ $n$-cubes.
 \label{item:tame-perturbation-2}
\end{enumerate} 
\end{definition}

\begin{remark}
As discussed in the introduction, a perturbation is achieved in  Sections \ref{sec:reservoir-canal-system} and \ref{sec:evolution-seq}
by 
first making a dent to each $\Sigma$-component and collecting the carved out dents into a complex $T$. We next repartition $T$ into  subcomplexes $T_\Sigma$ whose adjacency graphs are trees and then redistribute. Indentation, a formal notion of dent, is  discussed in Section \ref{sec:indentation}.
\end{remark}

For  geometrical discussion, it is helpful to recall the metrics on the $\Sigma_i$-components and the lifts from Remark \ref{rmk:metric-Realization}.

\begin{definition}[Quasiconformal stability]\index{separating complex!quasiconformally stable}
\label{def:qc-stable}
Let $K$ be a good cubical $n$-complex and let $\sfK \ge 1$. A sequence $(Z_k)$ of separating complexes in refinements $(\Refine^{k \nu}(K))$, respectively,  is \emph{$\sfK$-quasiconformally stable} if, for each $k\geq 1$ and each boundary component $\Sigma \subset \partial K$, there exists a $\sfK$-quasiconformal homeomorphism 
\[
 |\Real_K(Z; \Sigma)| \to |\Real_{\Refine^{\nu k}(K)}(Z_k; \Sigma)|, 
\]
which is  identity on $|\Core_K(Z; \Sigma)|$. 
\end{definition}

We are now ready to state an Evolution theorem for separating complexes.

\begin{restatable}[Evolution of separating complexes]{theorem}{restateEvolution}\index{Evolution theorem of separating complexes}
\label{theorem:evolution-short} 
Let $n\geq 3$ and $ m\geq 2$. Let $K$ be a good cubical $n$-complex having $m$ boundary components and a polyhedral metric $d_K$, and let $\nu=\nu(K)\ge 1$ be the refinement scale.  Suppose that $Z_0$ is a separating complex in $K$. Then there exist constants $\lambda=\lambda(K)\ge 1$ and $\sfK= \sfK(K) \geq 1$, and a sequence $(Z_k)$ of separating complexes, $Z_k \subset \Refine^{\nu k}(K)$, for which
\begin{enumerate}
\item  each $Z_k$ has the relative Wada property with respect to $Z_{k-1}$; \label{item:evolution-Wada}
\item $(Z_k)$ is core-expanding;  \label{item:core-expanding}
\item  each $Z_k$ is an $\lambda$-perturbation of $Z_{k-1}$; \label{item:evolution-tameness}
\item $(Z_k)$ is $\sfK$-quasiconformally stable. \label{item:evolution-quasiconformal}
\end{enumerate}
\end{restatable} 

The proof of this theorem takes up the next three sections.

\subsubsection*{Topological Lakes of Wada}

The first two conditions in Theorem \ref{theorem:evolution-short} yield the existence of topological Lakes of Wada.

\begin{proposition}\label{prop:topological-lakes-of-wada}
Let $K$ be a good cubical $n$-complex with boundary components $\Sigma_1,\ldots, \Sigma_m$, where $m\geq 3$. Suppose that $(Z_k)$ is a sequence of core-expanding separating complexes in $(\Refine^{\nu k}(K))$ which satisfies the relative Wada property. Then
\[
X = \bigcap_{k=1}^\infty |\Star_{\Refine^{k\nu}(K)}(Z_k)|
\]
is a Wada continuum having the property that each connected component of $ |K|\setminus X$ contains exactly one boundary component $\Sigma_i$.
\end{proposition}

\begin{proof}
Let
\[
M_i=\bigcup_{k=1}^\infty |\Core_{\Refine^{k \nu}(K)}(Z_k;\Sigma_i)| \quad \text{for}\,\,\,  i=1,\ldots, m.
\]

The core-expanding property yields the fact that 
each $M_i$ is a domain, i.e.~an open and connected set, in $M=|K|$ and that $\bigcup_{i=1}^m M_i$ is dense in $M$. Furthermore, the relative Wada property implies that $M_1,\ldots,M_m$ share the same topological boundary. We now give the details.

Since the cores
$\Core(Z_k;\Sigma_i)= \Comp(Z_k;\Sigma_i) - \Star_{\Refine^{k\nu}(K)}(Z_k)$  are expanding, it follows immediately that, for each $i= 1,\ldots,m$,  
the boundary $\partialtop M_i $ is contained in $X$.

We check next that $X\subset \partialtop M_i$ for each $i=1,\ldots, m$. 
Fix an index $i$ and let  $x\in X$. Then, for each $k\ge 1$, the point $x$ is contained in the union $Q_k \cup Q'_k$ of two adjacent $n$-cubes $Q_k, Q'_k \in \Star_{\Refine^{k\nu}(K)}(Z_k)$ having a common face in $Z_k^{[n-1]}$.  
The relative Wada property of
$Z_{k+1}$  with respect to $Z_k$ yields the existence of an $n$-cube $Q^*_{k+1}\in \Comp(Z_{k+1};\Sigma_i)^{[n]} \bigcap \Refine^\nu(Q_k \cup Q'_k )^{[n]}$. Every $n$-cube in $\Comp(Z_{k+1};\Sigma_i)$ must meet the core $|\Core(Z_{k+1};\Sigma_i)|$.
We fix a point $x_{k+1}$ in $Q^*_{k+1} \cap |\Core(Z_{k+1};\Sigma_i)|$. Thus $x$ is a limit point of a sequence $(x_{k+1})$ in $ M_i$. Since $x$ is not in the interior of $M_i$, it must be a boundary point of $M_i$.

We conclude that
\[\partialtop M_1 = \cdots = \partialtop M_m =X.\]
\end{proof}


\section{Indentations}\label{sec:Indentation}

The goal of this section is prove an Indentation-flattening Theorem (Theorem \ref{theorem:flattening-indentation}). 
Heuristically, a dented cubical complex, by dents of a particular type,  may be recovered to its pre-dented form by bilipschitz homeomorphisms with bilipschitz constants independent of the structure of the dents. This type of dents are called indentations.
Reader may prefer to read this section after seeing the construction of reservoirs and canals  in Section \ref{sec:sub-reservoir-canal-system}.

\subsection{Spectral cubes}
Before discussing indentations, we define spectral cubes and the spectrum of a subcomplex. 

\begin{definition}
\label{def:spectral-cube}
\index{spectrum of a subcomplex} \index{$\cS_*(P;K)$}
Let $K$ be a cubical $n$-complex, $P$ be an $n$-subcomplex of $ \Refine^k(K)$, and $0\leq j \leq k$. An $n$-cube $Q\in \Refine^j(K)$ is called a \emph{spectral cube of $P$ of rank $j$ with respect to $K$} if 
\begin{enumerate}
\item $|Q|\subset |P|$, and 
\item $Q$ does not belong to any refinement of an $n$-cube $Q'\in \Refine^i (K)$ satisfying $i<j$ and $|Q'|\subset P$. \label{item:spectral-cube-max}
\end{enumerate}
\end{definition}

\begin{remark}
Clearly, a refinement of a subcomplex has the same spectral cubes, that is, $Q$ is a spectral cube of $P \subset \Refine^k(K)$ of rank $j$ if and only if $Q$ is a spectral cube of $\Refine(P)\subset \Refine^{k+1}(K)$ of rank $j$.
\end{remark}

\begin{definition}
\label{def:spectrum}
\index{spectrum of a subcomplex} \index{$\cS_*(P;K)$}
Let $K$ be a cubical $n$-complex, and $P$ be an $n$-subcomplex of $\Refine^k(K)$. The \emph{spectrum $\cS_*(P;K)$ of $P$ with respect to $K$}  is the sequence 
\[
\cS_*(P;K) = (\cS_0(P;K), \cS_1(P;K),\ldots, \cS_k(P;K)),
\]
where $\cS_j(P;K)$ is the subcomplexes of $ \Refine^j(K)$ spanned by the spectral cubes of $P$ of rank $j$ with respect to $K$.
\end{definition}
With a slight abuse of notation, we denote the entire family of spectral cubes
of $P$ with respect to $K$ 
by
\[
\cS_*(P;K)^{[n]} = \bigcup_{j=0}^k \cS_j(P;K)^{[n]}.
\]

\begin{figure}[htp]
\begin{overpic}[scale=.5,unit=1mm]{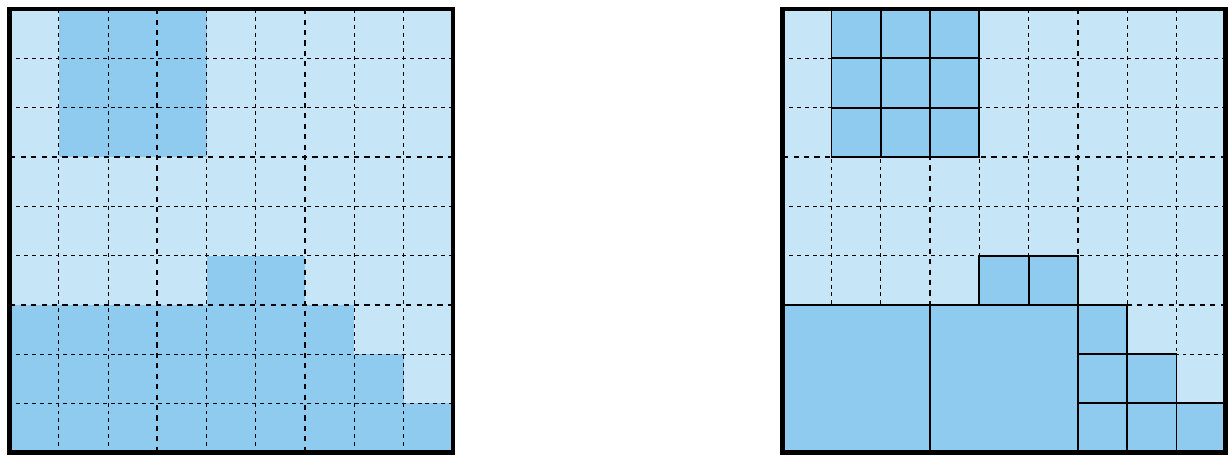} 
\put(16,7){{$P$}}
\put(10,33){{$P$}}
\end{overpic}
\label{fig:Spectrum}
\caption{Left: $K$ a cubical complex consisting of one $2$-cube, and $P$ a subcomplex of $\Refine^2(K)$. Right: spectral cubes  of $P$.}
\end{figure}

\begin{remark}
By 
maximality condition in
\eqref{item:spectral-cube-max}
in Definition \ref{def:spectral-cube}, a subcomplex $P \subset \Refine^k(K)$ has a unique spectrum $\cS_*(P;K)$ and that
\[
P = \bigcup_{j=0}^k \Refine^{k-j}(\cS_j(P;K)).
\]
Furthermore, $\cS_*(P;K)$ is an essential partition of $|P|$.
\end{remark}

\subsection{Cube and star indentations}
\label{sec:indentation}
We define
two classes of indentations, namely, cube-indentations and star-indentations, and  then combine them into the definition of indentation.

Loosely stated, a cube-indentation is an $n$-subcomplex in $\Refine^k(K)$ having all its spectral cubes lying on the boundary $|\partial K|$. 
The following formal definition is more restrictive and imposes some additional regularity conditions on the sizes and adjacency of spectral cubes.

\begin{definition}
\label{def:indentation}
\index{indentation}
Let $K$ be a good cubical $n$-complex, and let $q$ be an $(n-1)$-cube in $(\partial K)^{[n-1]}$. A cubical $n$-subcomplex 
$D\subset \Refine^k(K)$,
possibly empty, is a \emph{cube-indentation in $K$ (over  $q$)} if its spectrum $\cS_*(D;K)$ has the following properties.
\begin{enumerate}
\item $\cS_0(D;K)=\cS_1(D;K) = \emptyset$. \label{item:indentation-1}
\item Each spectral cube of $D$ has a face contained in the interior of $|q|$.  \label{item:indentation-2}
\item  Intersection $Q\cap Q'\cap Q''$ of three distinct spectral cubes $Q,Q',Q''$ is empty. \label{item:indentation-4}
\item The intersection $Q\cap Q'$ of spectral cubes $Q\in \cS_j(D;K)$ and $Q'\in \cS_{j'}(D;K)$, $j\le j'$, if nonempty, is a face of $Q'$; in the case $j<j'$,  $Q'$ meets only two faces of $Q$, one of which is $Q\cap |q|$. 
\label{item:indentation-3} 
\end{enumerate}
\end{definition}

We record two simple observations on cube-indentations.
\begin{lemma}
\label{lemma:s-c-i-adjacency}
Let 
$D\subset \Refine^k(K)$
be a cube-indentation in $K$ over $q\in (\partial K)^{[n-1]}$ and let $Q \in \cS_*(D;K)^{[n]}$ be a spectral cube in $D$. Then $Q$ meets at most two spectral cubes of the same or larger side length, out of which at most one has larger side length than $Q$. Furthermore, if $Q$ meets two spectral cubes $Q'$ and $Q''$ of the same side length 
or larger side length,
then $Q\cap Q'$ and $Q\cap Q''$ are opposite faces of $Q$.
\end{lemma}

Note that there is no restriction on the number of spectral cubes in $\cS_*(D;K)$ of smaller side length which cube $Q$ may meet. This is a crucial property in forthcoming constructions.

\begin{proof}[Proof of Lemma \ref{lemma:s-c-i-adjacency}]

For the first claim, suppose that a spectral cube $Q$ of $\cS_*(D;K)^{[n]}$ meets spectral cubes $Q'$ and $Q''$ in $\cS_*(D;K)^{[n]}$ of the same or larger side length. Since $Q\cap Q'\cap Q''$ is empty by \eqref{item:indentation-4}, we have that $Q\cap Q'$ and $Q\cap Q''$ are opposite faces of $Q$. Since this holds for all spectral cubes of $\cS_*(D;K)^{[n]}$ of the same or larger side length than $Q$, we conclude that $Q$ meets at most two spectral cubes of the same or larger side length. 

Retaining the previous notations, we may assume that the side length of $Q''$ is at most the side length of $Q'$. Since mutually disjoint spectral cubes have distance at least the side length of the smaller cube, we further have that $Q''$ has side length at most the side length of $Q$. Thus $Q$ meets at most one spectral cube having larger side length.
\end{proof}

\begin{lemma}\label{lemma:dent-cell}
Suppose that 
$D \subset \Refine^k(K)$
is a cube-indentation in $K$ over  $q$ for an $(n-1)$-cube $q\in \partial K$. Then
\begin{enumerate}
\item a connected $n$-subcomplex of  $\cS_j(D;K)$ consisting of $\ell$ distinct $n$-cubes is isomorphic to $[0,1]^{(n-1)}\times [0,\ell]$;\label{item:spectral-linear}
\item each connected component of $D$ is adjacently-connected, and it is also a cube-indentation over  $q$; and \label{item:spectral-connected-component}
\item if $D$  is connected, then $|D|$ is an $n$-cell and $|D \cap \partial K|$ and $|\partial D - \partial K|$ are $(n-1)$-cells. \label{item:dent-cells}
\end{enumerate}
\end{lemma}

\begin{proof}
Properties \eqref{item:spectral-linear} and \eqref{item:spectral-connected-component} follow immediately from the definition. Property \eqref{item:dent-cells} can be easily checked by induction on the number of spectral cubes in $\cS_*(D;K)$.
\end{proof}

We now define the second class of indentations -- star-indentations. Heuristically, cube-indentations over $\partial K$ do not bend. To bend, we connect indentations over  neighboring $(n-1)$-cubes with star-indentations. We give the formal definition in two parts.

\begin{definition}\label{def:partial-star}
Let $K$ be a good cubical $n$-complex and $\xi\in (\partial K)^{[n-2]}$ be an $(n-2)$-cube.  For $e \in (\Refine^j(\xi))^{[n-2]}$, we call the subcomplex $S=\Star_{\Refine^j(K)}(e)\subset \Refine^j(K)$ \emph{a star emitted from $\xi$}. 
\end{definition}

\begin{definition}\label{def:partial-star-indentation}\index{indentation!partial-star} 
Let $K$ be a good cubical $n$-complex and $\xi\in (\partial K)^{[n-2]}$ be an $(n-2)$-cube. A subcomplex $A\subset \Refine^k(K)$ is  a  \emph{star-indentation in $K$ (emitted from $\xi$)} if $A=\Refine^{k-j}(S)$ is the refinement of a star $S =\Star_{\Refine^j(K)}(e)$ emitted from $\xi$. In this case, we call  $S$ the star defining $A$ and denote it by $S_A$.
\end{definition}

\begin{remark}
In view of Lemma \ref{lemma:star-cyclic}, the adjacency graph of $\Star_{\Refine^j(K)}(e)$ is cyclic and  $S\cap \Refine^j(\partial K)$ is the union  of the  two $(n-1)$-cubes in $\partial S$ containing $\eta$. 
\end{remark}

An indentation in $K$ is  defined to be the union of a family of cube-indentations over  $(n-1)$-cubes in $\partial K$ and a family of star-indentations emitted from $(n-2)$-cubes in $\partial K$, with some particular rules on the intersections. 

\begin{definition}
\label{def:indentation-2}
Let $K$ be a good cubical $n$-complex.
A subcomplex $B$ of $\Refine^k(K), k\geq 2,$ is a \emph{indentation in $K$} if $B$ has an essentially disjoint partition 
\[
B=(\bigcup_{D\in \mathcal D(B)} \, D) \,\cup \,(\bigcup_{A \in \mathcal A(B)} \, A),
\]
where $\mathcal D(B)=\{D_q\colon q\in (\partial K)^{[n-1]}\}$ is a family of essentially disjoint cube-indentations and  $\mathcal A(B)$ is a family of mutually disjoint star-indentations in $K$ having following properties:
\begin{enumerate}
\item If $D_q, D_{q'} \in \mathcal D(B)$ meet for $q\neq q'$, then associated to each connected component $\eta$ of $|D_q \cap D_{q'}|$, there exists a star-indentation $A_\eta=\Refine^{k-j}(S_\eta)\in \mathcal A(B)$  emitted from $\xi=q\cap q'$ for which $|D_q \cap D_{q'} \cap A_\eta|=\eta$. \label{item:D}
\item If a star-indentation $A=\Refine^{k-j}(S_A)\in \cA(B)$ intersects a cube-indentation $D_q\in \cD(B)$, where $S_A$ is a star emitted from $\xi$, then $|A\cap D_q|$ is a common face of $S_A$ and a spectral cube in $D_q$ not intersecting $\xi$. \label{item:AD}
\end{enumerate}
\end{definition}

We make two remarks on this definition.
\begin{remark}
Suppose that $D_q\cap D_{q'} \ne \emptyset$ for $q\ne q'$. Then $q$ and $q'$ are faces of the same $n$-cube in $K$, which contains $|D_q|$ and $|D_{q'}|$; see left figure in Figure \ref{fig:indentation}.
By \eqref{item:D} and structure of the star, the each connected component of the intersection is a common edge $e$ (i.e.~$(n-2)$-cube) of the intersecting spectral cubes. In particular, the intersecting spectral cubes have the same side length.
Moreover, in this case, each star emitted from $\xi=q\cap q'$ consists of a single $n$-cube. Thus $A_\eta$ is a single cube having one face in $D_q$ and another in $D_{q'}$. 
\end{remark}

\begin{remark}
Let $A=\Refine^{k-j}(S_A)\in \mathcal A(B)$ be a star-indentation and $S_A$ be a star emitted from $\xi$. Since $\xi$ meets exactly two $(n-1)$-cubes of $\partial K$, $S_A$ meets  at most two cube-indentations in $\cD(B)$. Note that, the definition does not exclude the possibility that there is only one, or possibly none cube-indentation in $\cD(B)$ meeting $S_A$; see right figure in Figure \ref{fig:indentation}.
\end{remark}

\begin{figure}[htp]
\centering
\begin{overpic}[scale=.25,unit=1mm]{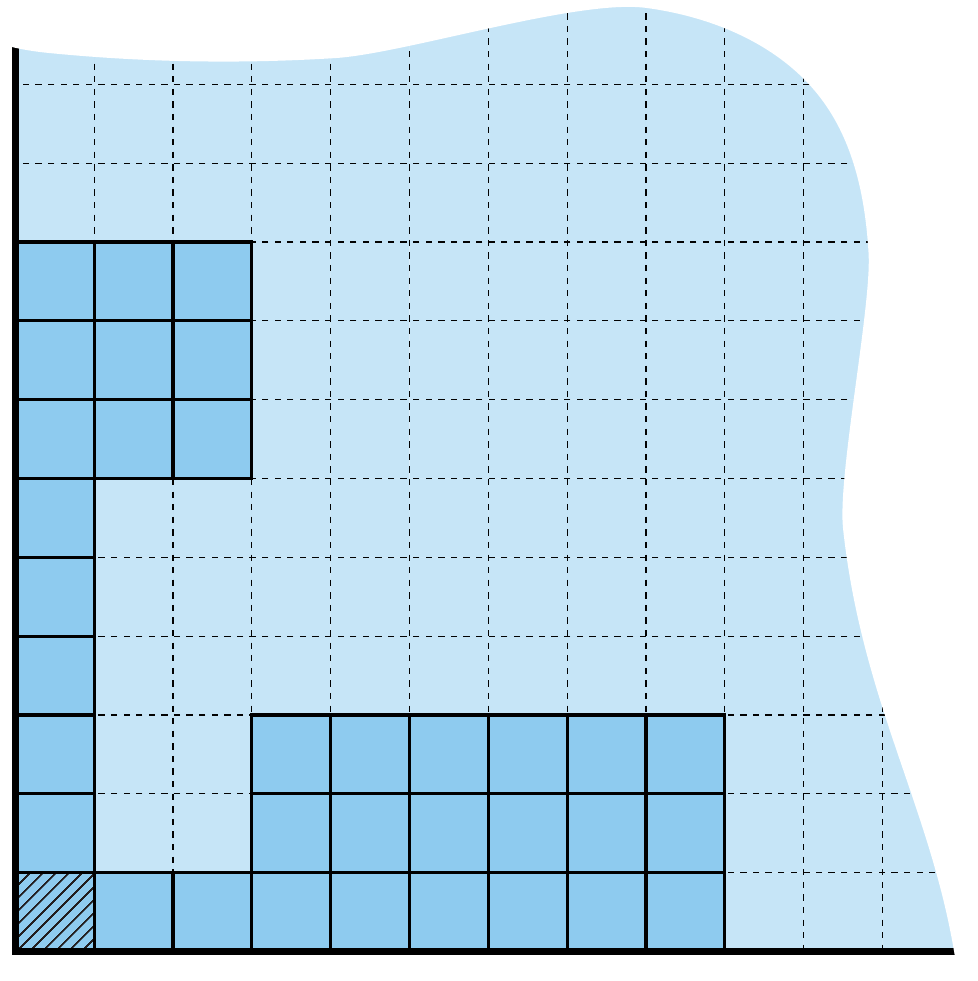} 
\put(4.5,4.5){$\eta$}
\put(-1.5,0){$e$}
\end{overpic}
\hfill
\begin{overpic}[scale=.23,unit=1mm]{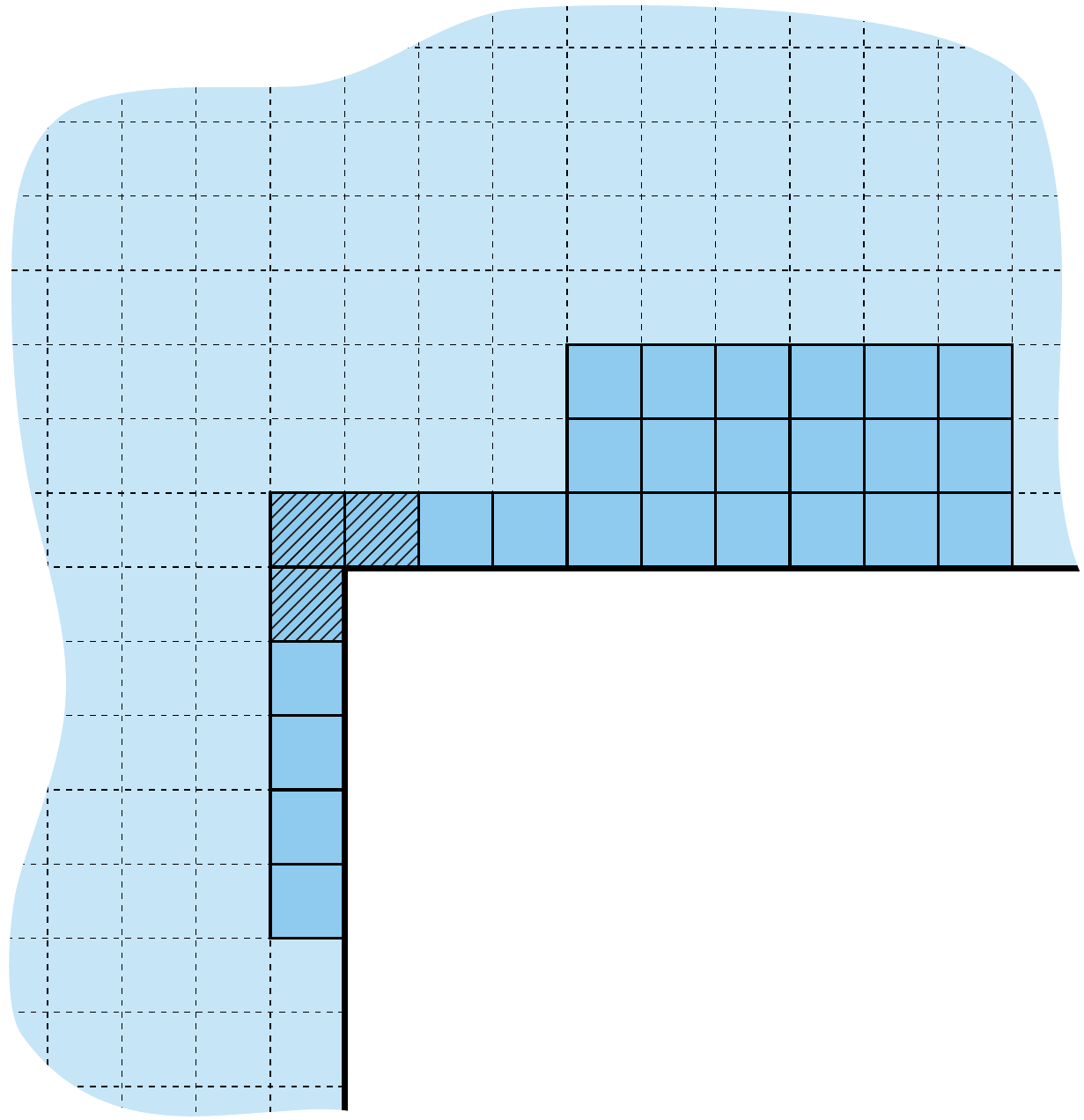} 
\put(16,22.5){$e$}
\end{overpic}
\caption{Indentations - spectral cubes in cube-indentations (medium blue) and stars in star-indentions (shaded blue).}
\label{fig:indentation}
\end{figure}

\subsection{Flattening indentations}
In this section, we discuss an iterative procedure for flattening indentations. The goal is to obtain a bilipschitz homeomorphism from the space of an indented complex to the space of the original complex, whose bilipschitz constant does not depend on the particular indentation. For this reason, we discuss first neighborhoods of indentations, called wedges, which allow us to track the change of bilipschitz constants in the flattening process.

\subsubsection{Wedges}

The definition of a wedge of subcomplex reads as follows.

\begin{definition}
\label{def:wedge}
\index{wedge} \index{$\Wedge_{|K|}(P)$}
Let $K$ be a good cubical $n$-complex.  Let $P \subset \Refine^k(K)$ be a subcomplex in which each spectral cube either has a face on $|\partial K|$, or belongs to a star $\Star_{\Refine^j(K)}(e)$ of an $(n-2)$-cube $e$ in the boundary $|\partial K|$.  We say that $x\in |K|$ is in the \emph{wedge of $P$} if there exists $p\in |P|$ for which 
\[
{\dist}_K(x,|P|) = d_K(x,p) \le \dist_K(p,|\partial K|)/4,
\]
where $d_K$ is the standard metric in $K$ defined Section \ref{sec:refinement-metric}. We call
\[
\Wedge_{|K|}(P)= \{ x\in |K| \colon x \text{ is in the wedge of } P\}
\]
the \emph{wedge of $P$ in $|K|$}.
\end{definition}

Note that $\Wedge_{|K|}(P)=\Wedge_{|K|}(\Refine(P))$ and observe that wedges are well-defined on an indentation as well as on a single $n$-cube $Q$ in an indentation.

We prove first a non-overlapping property of wedges associated to the cubes in a cube-indentation.

\begin{lemma}
\label{lemma:Wedge-intersections}
Let $k\geq 2$ and
$D \subset \Refine^k(K)$
be a cube-indentation in $K$ over an $(n-1)$-cube $q\in \partial K$. Then
\begin{enumerate}
\item if $Q$ and $Q'$ are cubes in $\cS_*(D)^{[n]}$ for which $Q\cap Q'=\emptyset$, then $\Wedge_{|K|}(Q)\cap \Wedge_{|K|}(Q')=\emptyset$; \label{item:Wedge-intersections-1}
\item if $Q,Q',Q''$ are three distinct cubes in $\cS_* (D)^{[n]}$  then $\Wedge_{|K|}(Q)\cap \Wedge_{|K|}(Q') \cap \Wedge_{|K|}(Q'')=\emptyset$. \label{item:Wedge-intersections-2}
\end{enumerate}
\end{lemma}
\begin{proof}
To prove \eqref{item:Wedge-intersections-1}, let $Q\in \cS_j(D)$ and $Q'\in \cS_{j'}(D)$ be two non-intersecting $n$-cubes. We may assume that
$j\leq j'$. Then the distance between $Q$ and $Q'$ is at least $3^{-j'}$.

Suppose towards contradiction that $\Wedge_{|K|}(Q) \cap \Wedge_{|K|}(Q')\ne \emptyset$ and let $x\in \Wedge_{|K|}(Q')\cap \Wedge_{|K|}(Q'')$. Let also $p\in Q$ and $p'\in Q'$ be points for which $\dist(x,Q) = d(x,p)\le \dist(p,|\partial K|)/4$ and $\dist(x,Q') = d(x,p')\le \dist(p',|\partial K|)/4$. 
Observe that
\begin{align*}
\dist(p,|\partial K|)&\leq d(p,x)+d(x,p')+ \dist(p',|\partial K|) \\
&\leq \dist(p,|\partial K|)/4+ 5\, \dist(p',|\partial K|)/4.
\end{align*}
Thus, $\dist(p,|\partial K|)\leq 5\, \dist(p',|\partial K|)/3$, and
\begin{align*}
\dist(Q,Q') &\le d(p,p') \le d(p,x) + d(x,p') \\ &\le \dist(p,|\partial K|)/4 + \dist(p',|\partial K|)/4 \\
&\le 2\dist(p',|\partial K|)/3 \le 2\cdot 3^{-j'-1}.
\end{align*}
This is a contradiction. Thus $\Wedge_{|K|}(Q')\cap \Wedge_{|K|}(Q'') = \emptyset$.

To prove \eqref{item:Wedge-intersections-2}, let $Q,Q',Q''$ be three distinct cubes in $\cS_*(D)^{[n]}$ and  assume that $Q''$ is less than or equal to the other two in size. 
Suppose  that $\Wedge_{|K|}(Q)\cap \Wedge_{|K|}(Q') \cap \Wedge_{|K|}(Q'')\neq \emptyset$. Then, by
 part \eqref{item:Wedge-intersections-1}
of this lemma, $Q\cap Q''\ne \emptyset$ and $Q'\cap Q''\ne \emptyset$ and, by Lemma \ref{lemma:s-c-i-adjacency}, 
the intersections
are opposite faces of $Q''$. 
This is a contradiction, since $Q\cap Q'\ne \emptyset$.
Hence $\Wedge_{|K|}(Q)\cap \Wedge_{|K|}(Q') \cap \Wedge_{|K|}(Q'')=\emptyset$.
\end{proof}

The proof of Lemma \ref{lemma:Wedge-intersections} may be modified to yield the following.

\begin{corollary}\label{cor:Wedge-intersections}
Let $B \subset \Refine^k(K)$ be an \emph{indentation in $K$}. As in Definition \ref{def:indentation-2}, we write $B=\left(\bigcup_{D\in \cD(B)} \, D\right) \cup \left(\bigcup_{A \in \mathcal A(B)} \, A\right)$, and let
\[
\cC(B)= \left(\bigcup_{D\in \cD(B)} \cS_*(D)^{[n]}\right) \cup \{S_A\colon A\in \mathcal A(B)\}
\]
be the collection of spectral cubes and stars defining $B$. Then
\begin{enumerate}
\item if $C$ and $C'$ are two elements in 
$\cC(B)$
for which
$|C|\cap |C'|=\emptyset$, then $\Wedge_{|K|}(C)\cap \Wedge_{|K|}(C')=\emptyset$; \label{item:Wedge-intersections-2-1}
\item if $Q,Q',Q''$ are three distinct spectral cubes in $\mathcal C(B)$, then the intersection of their wedges is empty. \label{item:Wedge-intersections-2-2}

\end{enumerate}
\end{corollary}
\begin{proof}
Property  \eqref{item:Wedge-intersections-2-1} has been proved in Lemma \ref{lemma:Wedge-intersections} in the case when both $C$ and $ C'$ are in the same cube-indentation. The cases, when $C$ and $C'$ are in different 
$D$ and $D'$ in $\cD(B)$ 
or when at least one of them is a star, can be proved analogously.

Property \eqref{item:Wedge-intersections-2-2} has been proved for the case when all three cubes are associated to the same 
$D\in \cD(B)$.
Other cases can also be proved analogously.
\end{proof}

\subsubsection{Indentation-flattening theorem}

We are now ready to state the main theorem for this section. Due to non-intersection properties of wedges around spectral cubes, indentations may be flattened by bilipschitz homeomorphisms
supported in wedges.
Recall that $\mu(K)$ in the statement is the local multiplicity of $K$; see Definition \ref{def:mu}.

\begin{theorem}[Indentation-flattening]
\label{theorem:flattening-indentation}
Let $K$ be a good cubical $n$-complex. Then there exists a constant $ L = L(n,\mu(K))\ge 1$ for the following. If $B\subset \Refine^{k}(K)$ is an indentation in $K$, then there exists a piecewise linear $ L$-bilipschitz homeomorphism
\[
\phi_B \colon |\Refine^k(K)-B| \to |\Refine^k(K)|,
\]
which is the identity on $|K|\setminus \Wedge_{|K|}(B)$. 
\end{theorem}

Proof of Theorem \ref{theorem:flattening-indentation} begins with the  flattening of a single cube.

\begin{proposition}
\label{prop:flattening-key-lemma}
Let $B \subset \Refine^k(K)$ be an indentation in $K$ and $Q$ be 
spectral cube in $\mathcal C(B)$ of the smallest side length.
Then 
spaces of complexes
$T=\partial Q\cap \partial (\Refine^k(K)-B)$ and $\partial Q-T$ are $(n-1)$-cells, whose common boundary is an $(n-2)$-sphere. Furthermore, there exists  a piecewise linear $L(n)$-bilipschitz map 
\[\psi \colon |\Refine^k(K)-B| \to |\Refine^k(K)-B|\cup |Q|,\]
which is the identity in the complement of $\Wedge_{|K|}(Q)$.
\end{proposition}

\begin{proof}
Let $q_0 = Q\cap |\partial K|$ and let $t\in \{0,1,2\}$ be the number of elements in $\mathcal C(B)$ adjacent to $Q$. 
For the first claim, 
we have three cases. 

If $t=0$, that is, $Q$ does not meet any other elements in $\mathcal C(B)$, then $\partial Q - T = q_0$ and $T = \partial Q - q_0$ are $(n-1)$-cells.

If $t=1$, that is,  $Q$ either shares a face $q_1$  with a spectral cube or with a star in  $\mathcal C(B)$, and with no other.  
Since $q_0$ and $q_1$ must be adjacent,  $\partial Q-T = q_0\cup q_1$ and $T=\partial Q - (q_0\cup q_1)$ are $(n-1)$-cells. 

Finally, if $t=2$ then $Q$ shares faces $q_1$ and $q_2$ with two distinct elements in $\mathcal C(B)$. Since $q_1$ and $q_2$ must be opposite faces of $Q$, both are adjacent to $q_0$. Hence $\partial Q - T = q_0\cup q_1\cup q_2$ and $T = \partial Q - (q_0\cup q_1 \cup q_2)$ are $(n-1)$-cells.

For the second claim, it suffices to observe that there exist a constant $L\ge 1$ and a piecewise linear $L$-bilipschitz map, $|T|\to |\partial Q-T|$, which is the identity on the boundary $|\partial T|$. This map extends as a piecewise linear bilipschitz map $\cl(\Wedge_{|K|}(Q)\setminus |B|) \to \cl(\Wedge_{|K|}(Q)\setminus |B-Q|)$, which is the identity on the boundary of $\Wedge_{|K|}(Q)$. We may now take $\psi$ to be the map, which is the extension of this map by identity.
\end{proof}

With Proposition \ref{prop:flattening-key-lemma} at our disposal, we can now 
inductively
flatten all spectral cubes of the smallest size. For the statement, let $\delta(B)$ be the largest index $j$ for which $\cS_j(D)\ne \emptyset$ for some $D\in \mathcal D(B)$, and let 
\[
B^\dagger = B- \bigcup_{D\in\mathcal D(B)} \Refine^{k-\delta(B)}(\cS_{\delta(B)}(D))
\]
be a subcomplex of $B$ 
\index{spectral reduction} \index{$B^\dagger$} 
from which the smallest spectral cubes of $B$ are excluded. Note that $B^\dagger$ is again an indentation. Observe that $\delta(B^\dagger) < \delta(B)$ and denote
\[
\mathcal C(B)_{\delta(B)}= \bigcup_{D\in \mathcal D(B)} \cS_{\delta(B)}(D)^{[n]}.
\]

\begin{lemma}
\label{lemma:first-flattening-theorem}
There exists a constant $L=L(n)\ge 1$ for the following. If $B \subset \Refine^k(K)$ is an indentation in $K$, then there exists a piecewise linear $L$-bilipschitz homeomorphism 
\[
\phi_{B^\dagger} \colon |\Refine^k(K)-B| \to |\Refine^k(K)-B^\dagger|,
\]
which is the identity on $|K|\setminus \Wedge_{|K|}(\mathcal C(B)_{\delta(B)})$. 
\end{lemma}

\begin{proof}
Fix an enumeration $Q_1,\ldots$, $ Q_s$ of the $n$-cubes in $\mathcal C(B)_{\delta(B)}$, to be used for the order in which the cubes are to be flattened. For each $i=1,\ldots, s$, set $E_i= Q_1\cup \cdots \cup Q_i$,
set also $E_0 = \emptyset$. Note that  each $B-E_i$ is an indentation in $K$ and that $B-E_s = B^\dagger$. 

By Proposition \ref{prop:flattening-key-lemma} that there exist $L=L(n)\ge 1$ and, for each $i=1,\ldots, s$, a piecewise linear $L'$-bilipschitz map 
\[
\psi_i \colon |(\Refine^k(K)-B)\cup E_{i-1}| \to |(\Refine^k(K)-B)\cup E_i|,
\]
which is identity in the complement of $\Wedge_{|K|}(Q_i)$. In view of Corollary \ref{cor:Wedge-intersections}, mapping $\psi_i$ can be chosen so that 
for each $Q\in \mathcal C(B)_{\delta(B)}$ intersecting $Q_i$, image $\psi_i (\Wedge_{|K|}(Q))$ does not meet $\Wedge_{|K|}(Q')$ for any  $Q'\in\mathcal C(B)_{\delta(B)}, Q'\neq Q_i, Q$.

Since each spectral cube of $D$ meets at most two other spectral cubes of the same size, 
we have,
 for each $x\in|K|$, that
\[\psi_i(\psi_{i-1}\ldots (\psi_1(x)))\ne \psi_{i-1}(\psi_{i-2} \ldots( \psi_1(x)))\]
for at most two indices $i\in \{2,\ldots, s\}$. 
Thus the composition
\[
\psi_B = \psi_s \circ \cdots \circ \psi_1 \colon |\Refine^k(K)-B| \to |\Refine^k(K)-B^\dagger|
\]
is an $(L')^2$-bilipschitz homeomorphism, which is an identity in the complement of $\Wedge_{|K|}(\mathcal C(B)_{\delta(B)})$.
\end{proof}

\begin{proof}[Proof of Theorem \ref{theorem:flattening-indentation}]
Since the subcomplex $B^\dagger$ of an indentation $B$ is again an indentation, we may iteratively apply Lemma \ref{lemma:first-flattening-theorem} to flatten all spectral cubes.
By Corollary \ref{cor:Wedge-intersections}, if $Q, Q', Q''$ are three distinct spectral cubes in $\bigcup_{D\in \mathcal D(B)} \cS_*(D)^{[n]}$ then 
\[
\Wedge_{|K|}(Q) \cap \Wedge_{|K|}(Q') \cap \Wedge_{|K|}(Q'') = \emptyset.
\]
An iterative application of  Lemma \ref{lemma:first-flattening-theorem} and its proof yields the existence of individual $L(n)$-bilipschitz homeomorphisms
\[
|\Refine^k(K)-B_0| \to |\Refine^k(K)-B_1| \to \cdots \to |\Refine^k(K) - B_t|,
\]
whose composition $\psi_{\mathcal D(B)}$ is $(L(n))^2$-bilipschitz and is the identity in the complement of $\Wedge_{|K|}(\bigcup_{D\in\mathcal D(B)}\, D)$. Here $B_0 = B$, $B_i = (B_{i-1})^\dagger$, and $t\in [1,k]$ is the number of indices $j$ for which $\cS_j(\bigcup_{D\in\mathcal D(B)}\, D)\ne \emptyset$.  Thus $B_t = \bigcup_{A\in\mathcal A(B)} A$, and all spectral cubes are flattened by the map
\[ \psi_{\mathcal D(B)}\colon |\Refine^k(K)-B|  \to |\Refine^k(K) - \bigcup_{A\in\mathcal A(B)} A|.\]

Since elements in $\mathcal A(B)$ are mutually disjoint, distinct elements  have disjoint wedges  by Corollary \ref{cor:Wedge-intersections}. Star indentations in $\mathcal A(B)$ 
may be flattened independently.
Let $A=\Refine^{k-j}(S_A)\in \mathcal A(B)$ and $S_A \subset \Refine^{j}(K)$. Then  $|\partial A \cap \Refine^k(\partial K)|=|S_A \cap \Refine^j(\partial K)|$ is the union of two adjacent $(n-1)$-cubes, hence an $(n-1)$-cell. 
Thus $\partial A -  \Refine^k(\partial K)|$ is also an $(n-1)$-cell.

Since the number of $n$-cubes in  $S_A,  A\in \mathcal A(B)$ is at most $\mu(K)$,  there exists a piecewise linear $L(n,\mu(K))$-bilipschitz homeomorphism 
\[
\psi_{\mathcal A(B)} \colon |\Refine^k(K)-\bigcup_{A\in \mathcal A(B)} \, A |\to |K|,
\]
which is the identity in the complement of  $\Wedge_{|K|}(\bigcup_{A\in \mathcal A(B)}\, A)$.

The claim follows by setting $ \phi_B = \psi_{\mathcal A(B)} \circ \psi_{\mathcal D(B)}$.
\end{proof}


\section{Channeling by reservoirs and canals}
\label{sec:reservoir-canal-system} 

In this section, we develop the technical part of the proof for the Evolution Theorem  \ref{theorem:evolution-short} -- channeling by reservoirs and canals. 

Suppose that - for concreteness - we are given a separating complex $Z$ in a cubical $n$-complex $K$ having $m$ boundary components, and a refinement scale $\nu$. Our  task is to perturb  $\Refine^\nu(Z)$ in $\Refine^\nu(K)$ into a new $(n-1)$-complex $\widetilde Z$ which has the relative Wada property with respect to $Z$ and has controlled geometry.

Perturbation is made in two steps. In the first step, we build a reservoir-canal system $\sfRC(Z)$ along $Z$ which has $m$ subsystems $\sfRC(Z)_1,\ldots,\sfRC(Z)_m,$ and transform $Z$ into a new $(n-1)$-subcomplex $\Tr(Z)\subset \Refine^\nu(Z)\cup \sfRC(Z)$, which satisfies a pre-Wada property with respect to $Z$.  The complex $\Tr(Z)$ is not yet a separating complex, 
since
the cut-graph $\Gammacut(\Refine^\nu(K); \Tr(Z))$ has $m$ extra connected components, each associated to a spanning tree of $\sfRC(Z)_i$. In the second step, we fix this issue by removing some $(n-1)$-cubes from $\Tr(Z)$ to create connections among some of the components of $\Gammacut(\Refine^\nu(K);\Tr(Z))$. We call this process \emph{channeling}, and the new complex $\widetilde Z=\Channel(Z)$ a channeling of $Z$.

Heuristically, the regions in $|K|$ separated by a separating complex $Z$ may be viewed as lakes filled with water of different colors. Reservoir-canal systems are merely intermediate structures and they are partitioned into artificial lakes after transformation and filled with water after channeling. In particular, the lakes of $\widetilde Z = \Channel(Z)$ have the same colors as lakes of $Z$.

\begin{remark*}
We emphasize that, although $Z$ is a subcomplex of $K$, perturbation of $Z$ is performed in the refinement $\Refine^\nu(K)$ of $K$. Therefore, it is understood that transformation $\Tr(Z)$, channeling $\Channel(Z)$, as well as the reservoir-canal systems $\sfRC(Y)$ are subcomplexes of $\Refine^\nu(K)$. 
\end{remark*}

This section is divided into two parts. In the first part (Sections \ref{sec:transformation}, \ref{sec:building-blocks}, and 
\ref{sec:reservoir-canal-transformation}),
we construct reservoir-canal systems and discuss the transformation. In the second part (Section \ref{sec:single-channeling}), we discuss the channeling. 

After the initial perturbation of $Z$, from the second iterative step onward, channeling  will be performed locally on subcomplexes of separating complexes. For this reason,  the constructions in this section are made on pairs $(U,Y)\subset (K,Z)$ in which the topological relations between $Y$ and $U$ are more complex than that between $Z$ and $K$. In particular, $Y$ need not be a separating complex in $U$, and $Y$ need not even be contained in the interior of $U$.

\subsection{Goal and Standing assumptions}
\label{sec:transformation}

We state  in Proposition \ref{prop:tr-existence-short} the goal of transformation and later, in 
Proposition \ref{prop:summary-Channeling},  summarize the combined process of transformation and channeling.
Throughout this section, we make the following assumptions on complex $K$.

\begin{standing}\label{standing:Z}$\,$ \index{Standing assumptions on $(K;Z)$}
\begin{itemize}
\item $K$ is a good cubical $n$-complex 
with $m (\geq 2)$ boundary components, $\Sigma_1,\ldots, \Sigma_m$, and admitting a separating complex.
\item $\mu=\mu(K)$ is the local multiplicity of $K$, and 
$\nu = \nu(K)\ge 1$ is the refinement scaleof $K$, that is, 
an integer for which  $3^\nu \ge 3^{10}\mu^2 m$.
\end{itemize}
\end{standing}

We now define complex pairs  $(U,Y)$  which are admissible for transformation.

\begin{definition}
\label{def:admissible}
Suppose that $k\geq 0$ and $Z_k$ is a separating complex in $\Refine^{k \nu}(K)$. An adjacently connected $(n-1)$-subcomplex $Y$ of $Z_k$ is said to be  \emph{admissible, in $(\Refine^{k \nu}(K), Z_k)$},  if there exists 
an adjacently connected $n$-subcomplex $U$ of $\Refine^{k \nu}(K)$ for which 
\begin{enumerate}
\item $Y \subset U$, 
\item each $(n-1)$-cube in $Y$ is a face of an $n$-cube in $U$, 
\item each $n$-cube in $U$ has a face in $Y$, 
\item each connected component of the cut-graph $\Gammacut(U;Y)$  is a subgraph of some $\Gamma(\Comp_{\Refine^{k \nu}(K)}(Z_k;\Sigma_i))$, and  \label{item:subgraph}
\item for each $i=1,\ldots, m$, subcomplex $U$ has at least one $n$-cube in $\Comp_{\Refine^{k \nu}(K)}(Z_k;\Sigma_i)$. \label{item:at-least}
\end{enumerate}
We call $(U,Y)$ an \emph{admissible pair.} 
\end{definition}

\begin{definition}\label{def:localization}
An essentially disjoint partition $\cL$ of $Z_k$ is said to be a \emph{localization of $Z_k$} if each element of $\cL$ is admissible in $(\Refine^{k \nu}(K), Z_k)$.
\end{definition}
Trivially,  $(K,Z)$ is an admissible pair in $(K, Z)$ and $\cL_0=\{Z\}$ is a localization of $Z$.

By conditions \eqref{item:subgraph} and \eqref{item:at-least} in Definition \ref{def:admissible}, $\Gammacut(U;Y)$ has at least $m$ connected components. We follow the convention below in labeling these components. 

\begin{convention}\label{convention:labeling}
Suppose that pair $(U, Y)$ is admissible  in $(\Refine^{k \nu}(K), Z_k)$. We label the connected components $G_1,\ldots, G_r$ of the cut-graph $\Gammacut(U;Y)$, where  $r=r(U,Y) \geq m$, so that, for $i\in \{1,\ldots, m\}$, $\Span_U (G_i)$ is contained in $\Comp_{\Refine^{k \nu}(K)}(Z_k;\Sigma_i)$.  
  \end{convention}

We now state the main result of the first part of this section.
\begin{proposition}\label{prop:tr-existence-short}
Assume the Standing assumptions \ref{standing:Z} and let $Z_k$ be a separating complex in $\Refine^{k \nu}(K)$. Suppose that  
$(U,Y)$ is an admissible pair in $(\Refine^{k \nu}(K), Z_k)$ and components $G_1, \ldots, G_r$  of $\Gammacut(U;Y)$ are labeled according to Convention \ref{convention:labeling}.
Then, there exists an  adjacently connected $(n-1)$-subcomplex 
\[ \Tr(Y)\subset \Refine^\nu(U)\]
 for which
\begin{enumerate}
\item the cut-graph $\Gammacut(\Refine^\nu(U);\Tr(Y))$ has $r + m$ connected components,  
\[G'_1, \ldots, G'_r, \tau_1, \ldots, \tau_m,\]
  where $G'_i$ is a subgraph of $\Gamma(\Refine^\nu(\Span_U(G_i)))$ for $i\in \{1,\ldots, r\}$,  and $\tau_i$ is a tree of size at most $\lambda=\lambda(\#(Y^{[n-1]}), \mu, \nu)$ for $i\in\{1,\ldots, m\}$,  and 
 \label{item:rplusm}
\item $ \Tr(Y)$ satisfies a \emph{pre-Wada property}:  for each $q\in Y^{[n-1]}$ and each $i\in \{1,\ldots, m\}$, there is an $n$-cube in $\Span_{\Refine^\nu(U)}(\tau_i)$ having a face in $q$. \label{item:pre-Wada}
\end{enumerate}
\end{proposition}

Complex $\Tr(Y)$ is called a \emph{transformation of $Y$}. The proof of the proposition is completed at the end of  Section \ref{sec:receded-complex}.

\subsubsection{Preference functions}
Before building the reservoir-canal systems for the transformation, we introduce the notion of preference functions which are used to guide the placement of reservoirs and canals.

\begin{definition}
\label{def:preference}
Let $U$ be an $n$-cubical complex with a good interior and $Y\subset U$ an $(n-1)$-subcomplex.
A function $\rho \colon Y^{[n-1]}\to U^{[n]}$ is called a \emph{preference function on $Y$} if, for each $q\in Y^{[n-1]}$,  $\rho(q)$ is an $n$-cube having $q$ as a face. \index{preference function}
\end{definition}

We do not assume the preference function $\rho$ to be injective, nor assume cubes $\rho(q)$ and $\rho(q')$ to be adjacent even for adjacent $q$ and $q'$.

For an admissible pair $(U,Y)$ in 
$(\Refine^{k \nu}(K), Z_k)$, a preference function $\rho_Y \colon Y^{[n-1]} \to U^{[n]}$ is said to be \emph{admissible} if $\rho_Y (Y^{[n-1]})$ contains at least one $n$-cube in each $\Span_U (G_i)$, hence in $\Comp_{\Refine^{k \nu}(K)}(Z_k;\Sigma_i)$, for $i\in \{1,\ldots, m\}$. 
Note that an admissible pair always admits an admissible preference function.

\subsection{Reservoir and canal systems}\label{sec:building-blocks}

We now construct reservoir-canal systems for the proof of Proposition \ref{prop:tr-existence-short}. An illustration of the construction is given in Figures 
\ref{fig:Reservoir_Canal_Single} and \ref{fig:Reservoir_Canal_System}.\footnote{Figures \ref{fig:Reservoir_Canal_Single} and \ref{fig:Reservoir_Canal_System} are drawn with the assistance of Julie Kaufman.}

  To do this, we retain all assumptions in the proposition, and fix an admissible preference function $\rho_Y\colon Y^{[n-1]}\to U^{[n]}$ and 
a partially ordered spanning tree $\cT_Y$ in $\Gamma(Y)$. 

Using data $(U,Y,\rho_Y, \cT_Y)$, we create a reservoir over each $q\in Y^{[n-1]}$ and then connect them with canals.
The placement of the reservoirs is determined by the preference function $\rho_Y$, and the flow of canals is determined by the spanning tree $\cT_Y$. For brevity, we write $\rho=\rho_Y$ and $\cT=\cT_Y$. 

We begin with the constructions of pre-reservoir-blocks, markers, canal sections, and connectors in Sections \ref{sec:pre-reservoir-block} $\sim$ \ref{sec:connector}. 

\subsubsection{Pre-reservoir-blocks}
\label{sec:pre-reservoir-block}

We  fix a family of model reservoirs $\mathbf{R}_1, \ldots, \mathbf{R}_m$ in the unit $n$-cube $[0,1]^n$. For each $i=1,\ldots, m$, let
\[
D_i= \left[\frac{4}{9}+\frac{i-1}{3^\nu}, \frac{5}{9}-\frac{i-1}{3^\nu}\right]^{n-1} \times \left[0,\frac{1}{9}-\frac{i-1}{3^\nu}\right] \subset [0,1]^n,
\]
and let $D_{m+1}= \emptyset$.   We take $\mathbf R_i\subset \Refine^\nu([0,1]^n)$ to be the subcomplex, whose space is $\cl(D_i\setminus D_{i+1})$, that is,  
\[
\mathbf R_i = \Span_K(\{Q\in \Refine^\nu(U)^{[n]}\colon Q\subset \cl(D_i\setminus D_{i-1})\}).
\]
Observe that each $|\mathbf R_i|$ is an $n$-cell, and $\Gamma(\mathbf R_i)$ is connected. We call the cubical complexes $\mathbf R_i$ \emph{pre-reservoir-blocks}.

To construct pre-reservoir-blocks on $Y$ with respect to a preference function $\rho$, let $q\in  Y^{[n-1]}$ and  $Q = \rho(q)$. After applying an isometry of $[0,1]^n$ if needed, we may assume that the map $\phi_Q \colon Q\to [0,1]^n$, 
associated to $Q$ in the cubical structure, has the property that $\phi_Q(q) = [0,1]^{n-1}\times \{0\}$. Since $\Refine^\nu(Q) = \phi_Q^*(\Refine^\nu([0,1]^n)$, subcomplexes 
\[
\sfR_{q;i} = \phi_Q^*(\mathbf R_i) \subset \Refine^\nu(Q), \qquad i=1,\ldots, m,
\]
are well-defined. We call $\sfR_{q,1},\ldots, \sfR_{q,m}$ \emph{pre-reservoir-blocks adjacent to $q$}.   The name stems from the fact that \index{reservoir-canal-system!pre-reservoir block}\index{$\sfR_{q;i}$}
\[
\sfR_{q;i} \cap \Refine^\nu(q) = \phi_Q^*(\sfR_i \cap \Refine^\nu([0,1]^{n-1}\times \{0\}))
\]
is an adjacently-connected cubical complex contained in $|q|$. We call the union
\[
\sfR_q =\bigcup_{i=1}^m \sfR_{q;i}
\]
a \emph{pre-reservoir-cube over $q$}, and observe that $\sfR_q$ is a refinement of a cube in $\Refine^2(Q)$.
\index{$\sfR_q$}

\subsubsection{Markers}
\label{sec:markers}\index{reservoir-canal system!markers}

Fix first a family of model markers in the unit cube $[0,1]^{n-2}$. Since $3^\nu \ge 3^{10}\mu^2 m$, there exists a family 
\[
\mathbf E=\{\zeta_1,\ldots, \zeta_m; \, \zeta_{m+1},\ldots,\zeta_{2m};\, \ldots; \,\zeta_{(\mu^2-1)m+1}\ldots, \zeta_{\mu^2 m}\}
\]
of mutually disjoint $(n-2)$-cubes contained in $\Refine^{\nu-2}\left([\frac{4}{9},\frac{5}{9}]^{n-2}\times \{0\}\times \{0\}\right)$. 
The elements in $\mathbf E$ are called the \emph{model markers}.

Note that $[\frac{4}{9},\frac{5}{9}]^{n-2}\in \Refine^2([0,1]^{n-2})$, hence 
\[
\Refine^{\nu-2}\left(\left[\frac{4}{9},\frac{5}{9}\right]^{n-2}\times \{0\}\times \{0\}\right) \subset \Refine^\nu([0,1]^{n-2}\times \{0\}\times \{0\})).
\]

We use 
cubical isomorphisms $\phi_Q \colon Q\to [0,1]^n$ in the cubical structure of $U$ to define the markers on $Y^{[n-2]}$. To each $(n-2)$-cube $e  \in Y^{[n-2]}$, we fix an $n$-cube $Q\in U^{[n]}$ which contains $e$. Let $\phi_Q \colon Q\to [0,1]^n$ be the map associated to $Q$ in the cubical $n$-complex $U$.
After applying an isometry of $[0,1]^n$ if needed we may assume that $\phi_Q(e) = [0,1]^{n-2}\times \{0\}\times \{0\}$. We call 
\[
\zeta_{e;i}=  \phi_Q^*(\zeta_i), \qquad i=1,\ldots, \mu^2 m,
\]
\emph{markers on $e$}, and set 
\[
\sfE(e)= \{\zeta_{e,1},\ldots, \zeta_{e, m}; \, \zeta_{e, m+1},\ldots,\zeta_{e, 2m};\, \ldots; \,\zeta_{e,(\mu^2-1)m+1}\ldots, \zeta_{e,\mu^2 m}\}.
\]
Note that markers in $\sfE(e)$ are contained in $\Refine^{\nu-2}(c^2(e)) \subset \Refine^\nu(e)$, where $c^2(e) \in \Refine^2(e)$ is the center cube of the center cube $c(e)$ of $e$. We emphasize that, despite the notation, the markers of $U$ are not elements of $U$ but elements of $\Refine^\nu(U)$.

Since the model markers are mutually disjoint $(n-2)$-cubes in the refinement $\Refine^\nu([0,1]^{n-2}\times \{0\}\times \{0\})$,  stars of markers in $\Refine^\nu(U)$ are mutually disjoint; recall the notion of star from Definition \ref{def:star}.

\begin{lemma}
Let $\zeta$ and $\zeta'$ be markers in $U$, $\zeta\ne \zeta'$. Then 
\[
\Star_{\Refine^\nu(U)}(\zeta) \cap \Star_{\Refine^\nu(U)}(\zeta') = \emptyset.
\]
\end{lemma}

For each $e \in U^{[n-2]}$, we distribute the markers on $\sfE(e)$ to the pairs $\{q,q'\}$ of $(n-1)$-cubes in $\Star_U(e)^{[n-1]}$ for which $q\cap q'=e$. To do this, we fix for each $e\in U^{[n-2]}$, an injective marking function  
\[
\kappa_e \colon \{ \{q,q'\} \colon q, q' \in \Star_U(e)^{[n-1]}, q\cap q'=e \} \to \{1,\ldots, \mu^2\},
\]
and denote the $m$ markers associated to the pair $\{q,q'\}$ by 
\[
\zeta_{e,\{q,q'\};i} = \zeta_{e,(\kappa_e(\{q,q'\})-1)m+i}\in \sfE(e), \qquad  i=1,\ldots, m.
\]

The reason for the elaborate placement of markers will be explained in Remark \ref{rmk:markers}.

\subsubsection{Canal sections}
\label{sec:canal-section}

We define canal sections connecting pre-reservoir-blocks to markers guided by a spanning tree $\cT$ of $\Gamma(Y)$; see Figure \ref{fig:Reservoir_Canal_Single}

Let $q\in Y^{[n-1]}$. For each edge $\{q,q'\}$ in $\cT$ and  $i\in \{1,\ldots, m\}$, take $\sfC_{q,\{q,q'\};i}$ to be the unique adjacently connected $n$-subcomplex of $\Refine^\nu(\rho(q))$ consisting of minimal number of $n$-cubes  for which \index{reservoir-canal system!canal section} \index{$\sfC_{q,\{q,q'\};i}$}
\begin{enumerate}
\item $\sfC_{q,\{q,q'\};i} \cap \sfR_{q;i}$ is an $(n-1)$-cube, and 
\item $\sfC_{q,\{q,q'\};i}\cap \Star_{\Refine^\nu(U)}(\zeta_{q\cap q', \{q,q'\};i})$ is an $(n-1)$-cube. \label{item:rcs-2}
\end{enumerate}

\begin{remark}
The minimality of $\sfC_{q,\{q,q'\};i}$ has two immediate implications. First, the space $|\sfC_{q,\{q,q'\};i}|$ is isometric to $[0,s]^{[n-1]}\times [0,s\, \#(\sfC_{q,\{q,q'\};i}^{[n]})]$, where $s$ is the side length of $\zeta_{q\cap q', \{q,q'\};i}$. In particular, complexes $\sfC_{q,\{q,q'\};i}$ and $\sfC_{q',\{q,q'\};j}$ do not meet unless $q=q'$ and $i=j$.

Second, each $n$-cube in $\sfC_{q,\{q,q'\};i}$ has a face in $q$. Indeed, by \eqref{item:rcs-2}, $\sfC_{q,\{q,q'\};i}$ meets $\Star_{\Refine^\nu(U)}(\zeta_{q\cap q', \{q,q'\};i})$ in a face and, by minimality, this face meets $q$. 
\end{remark}

We call $\sfC_{q,\{q,q'\}, i}$ \emph{the canal section connecting the pre-reservoir-block $\sfR_{q;i}$ to the star of marker $\zeta_{q\cap q', \{q,q'\};i}$}. 
We denote 
\[
\sfC_{q;i} = \bigcup_{\{q,q'\}\in \cT} \sfC_{q,\{q,q'\};i}.
\]
Clearly, $\sfC_{q;i} \subset \Refine^\nu(\rho(q)) \subset \Star_{\Refine^\nu(U)}(\Refine^\nu(q))$.

The following two properties follow immediately from the construction, in particular, the nested nature of the pre-reservoir-blocks.

\begin{corollary}
\label{cor:canal-reservoir-intersection}
Let $q\in Y^{[n-1]}$, $\{q,q'\} \in \cT$, and $i,i'\in \{1,\ldots,m\}$. Then $\sfC_{q,\{q,q'\};i}$ intersects $\sfR_{q,i'}$ (in an $n$-cube) if and only if  $1\leq i'< i\leq m$.
\end{corollary}

\begin{corollary}
For each $q\in Y^{[n-1]}$ and $i\in \{1,\ldots, m\}$, the cubical complex $\sfR_{q;i} \cup \sfC_{q;i}$ is adjacently-connected.
\end{corollary}

\subsubsection{Connectors}\label{sec:connector}

We now define for each edge $\{q,q'\}$ in $\cT$ and each $i\in \{1,\ldots, m\}$, a subcomplex $\sfJ_{\{q,q'\};i}$ which connects canal sections $\sfC_{q,\{q,q'\};i}$ and $ \sfC_{q',\{q,q'\};i}$. The definition below uses the fact that complex $U$ has a good interior.

\begin{definition}
\label{def:connector} \index{reservoir-canal system!connector} \index{$\sfJ_{\{q,q'\};i}$}
Let $\{q,q'\}\in \cT$ be an edge and  $i\in \{1,\ldots, m\}$ an index. We define \emph{a connector $\sfJ_{\{q,q'\};i}\subset \Refine^\nu(K)$  connecting $\sfC_{q,\{q,q'\};i}$ and $ \sfC_{q',\{q,q'\};i}$  over marker $\zeta_{q\cap q', \{q,q'\};i}$} in two cases. 
\begin{itemize}
\item[Case 1.] If $\rho(q) = \rho(q')$, we take $\sfJ_{\{q,q'\};i}$ to be the unique $n$-cube in $\Refine^\nu(\rho(q)) \cap \Star_{\Refine^\nu(U)}(\zeta_{q\cap q', \{q,q'\};i})$.
\item [Case 2.]  
If $\rho(q)\neq \rho(q')$, we take $\sfJ_{\{q,q'\};i}$ to be a minimal adjacently-connected $n$-subcomplex of  $\Star_{\Refine^{\nu}(U)}(\zeta_{q\cap q', \{q,q'\};i})$ for which 
\[
\Star_{\Refine^\nu(U)}(\zeta_{q\cap q', \{q,q'\};i})\cap \left( \Refine^\nu(\rho(q)) \cup \Refine^\nu(\rho(q'))\right) \subset \sfJ_{\{q,q'\};i}.
\]
\end{itemize}
\end{definition}

\begin{remark}\label{rmk:J}In Case 2, the connector $\sfJ_{\{q,q'\};i}$ may have $n$-cubes outside
preferred cubes, that is, complexes $\Refine^\nu(\Span_U(\rho(Y^{[n-1]})))$.

Since $\rho(q)\ne \rho(q')$,  $\Star_{\Refine^\nu(U)}(\zeta_{q\cap q', \{q,q'\};i})\cap \left( \Refine^\nu(\rho(q)) \cup \Refine^\nu(\rho(q'))\right)$ has two $n$-cubes. Since  star $\Star_{\Refine^{\nu}(U)}(\zeta_{q\cap q', \{q,q'\};i})$ is cyclic, 
there are at most two possible choices for a connector $\sfJ_{\{q,q'\};i}$.. 
The minimality in the definition ensures that any proper subcomplex of a connector is not a connector.
\end{remark}

We combine now canal sections with connectors and say that two
canal sections are joined together by a connector into a canal stretch; see Figure \ref{fig:Reservoir_Canal_Single}.

\begin{definition}\index{$\sfC_{\{q,q'\};i}$} \index{reservoir-canal system!canal stretch}
For each edge $\{q,q'\}$ in $\cT$ and $i\in \{1,\ldots, m\}$, the complex 
\[
\sfC_{\{q,q'\};i} = \sfC_{q,\{q,q'\};i} \cup \sfJ_{\{q,q'\};i} \cup \sfC_{q',\{q,q'\};i}
\]
is called a \emph{canal stretch of index $i$ over the edge $\{q,q'\}$}.
\end{definition}

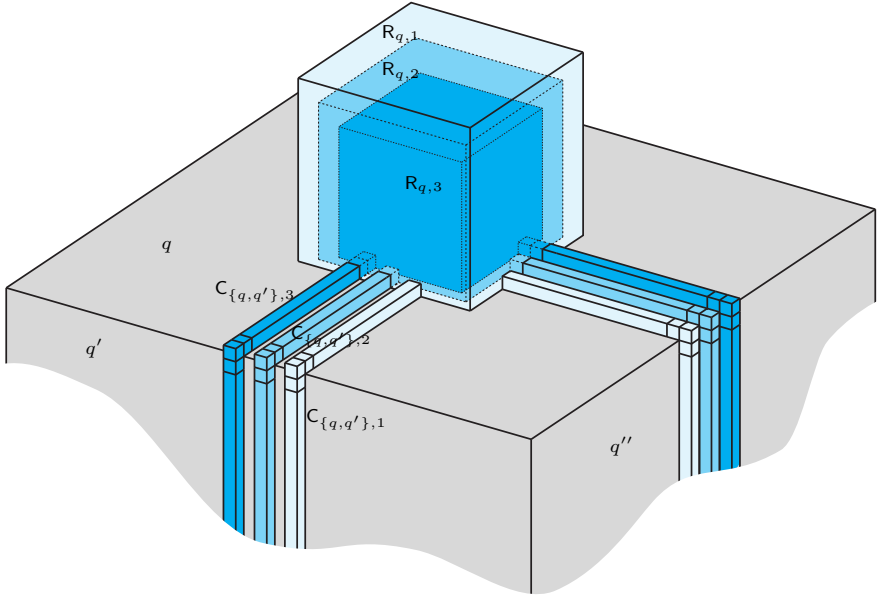
\begin{figure}[h!]
\begin{overpic}[scale=.56,unit=1mm]{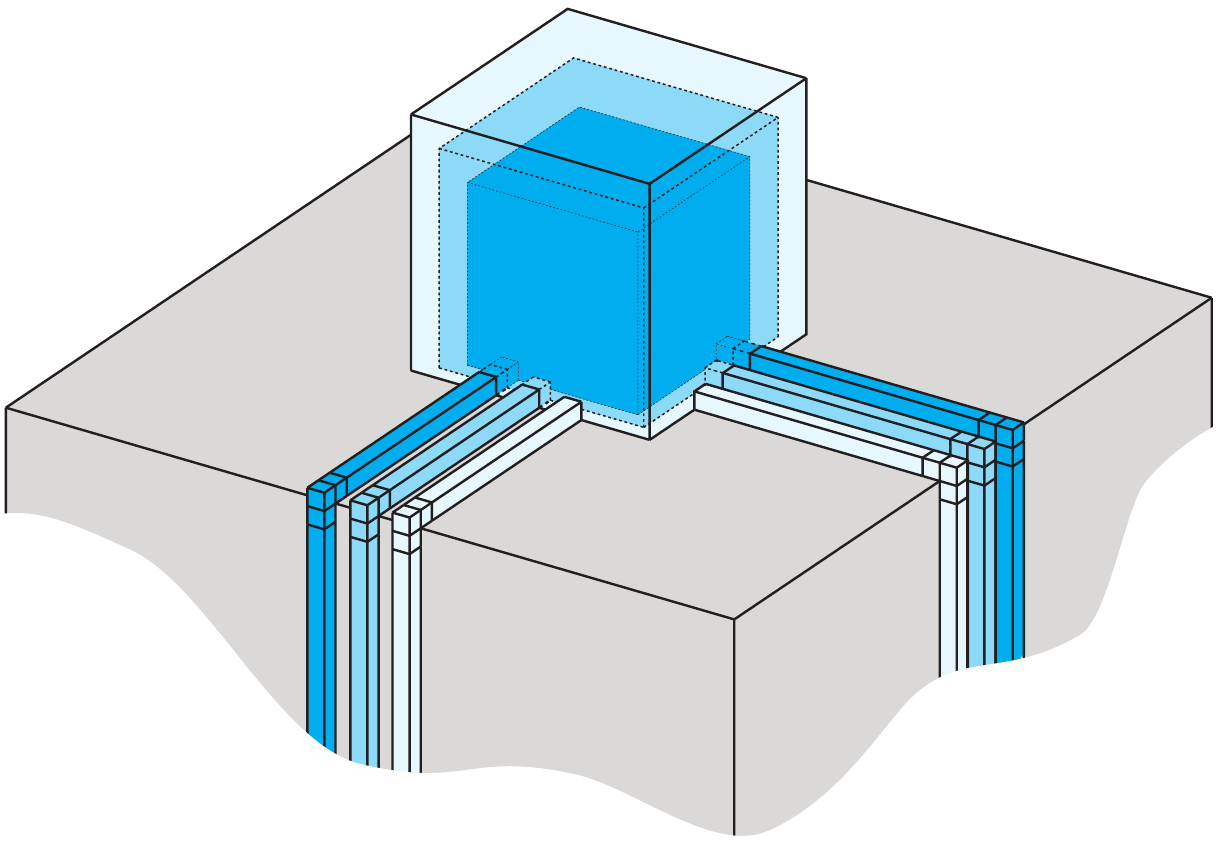}
\put(50,75){\tiny$\sfR_{q,1}$}
\put(50,70){\tiny$\sfR_{q,2}$}
\put(21,47){\tiny$q$}
\put(11,33){\tiny$q'$}
\put(80,20){\tiny$q''$}
\put(53,55){\tiny$\sfR_{q,3}$}
\put(38,35){\tiny$\sfC_{\{q,q'\},2}$}
\put(28,41){\tiny$\sfC_{\{q,q'\},3}$}
\put(40,24){\tiny$\sfC_{\{q,q'\},1}$}
\end{overpic}
\caption{Pre-reservoir-blocks and canal stretches over an $(n-1)$-cube $q\in  Y$ (not in scale).
}
\label{fig:Reservoir_Canal_Single}
\end{figure}

We omit the straightforward proof of the following lemma.

\begin{lemma}
\label{lemma:adjacency-canal-stretch}
For each edge $\{q,q'\}\in \cT$ and $i\in \{1,\ldots, m\}$, the canal stretch $\sfC_{\{q,q'\};i}$ and its boundary $\partial \sfC_{\{q,q'\};i}$ are adjacently-connected.
\end{lemma}

\begin{remark}\label{rmk:markers}
The placement of markers guarantees that canal stretches $\sfC_{\{q,q'\};i}$ and $\sfC_{\{q'',q'''\};j}$ for which $q\cap q'= q''\cap q'''=e$, in other words passing the same edge $e$,  do not meet and that the adjacency graph of each $\sfC_{\{q,q'\};i}$ is linear.
See Figure \ref{fig:Reservoir_Canal_System}.
\end{remark}

\subsubsection{Construction of the system $\sfRC_{\Refine^\nu(U)}(Y)$}
\label{sec:sub-reservoir-canal-system}
For  each fixed $i$, we have connected every pair of pre-reservoir-blocks $\sfR_{q;i}$ and $\sfR_{q';i}$, for $\{q,q'\}\in \cT$, into a canal stretch $\sfC_{\{q,q'\};i}$.  
 Since each canal stretch $\sfC_{\{q,q'\};i}$ enters every pre-reservoir-block $\sfR_{q,i'}$ of a lower index $i'<i$ (Corollary \ref{cor:canal-reservoir-intersection}), we first remove the $n$-cubes in the intersection  from each pre-reservoir-block.

\begin{figure}[h!]
\begin{overpic}[scale=.57,unit=1mm]{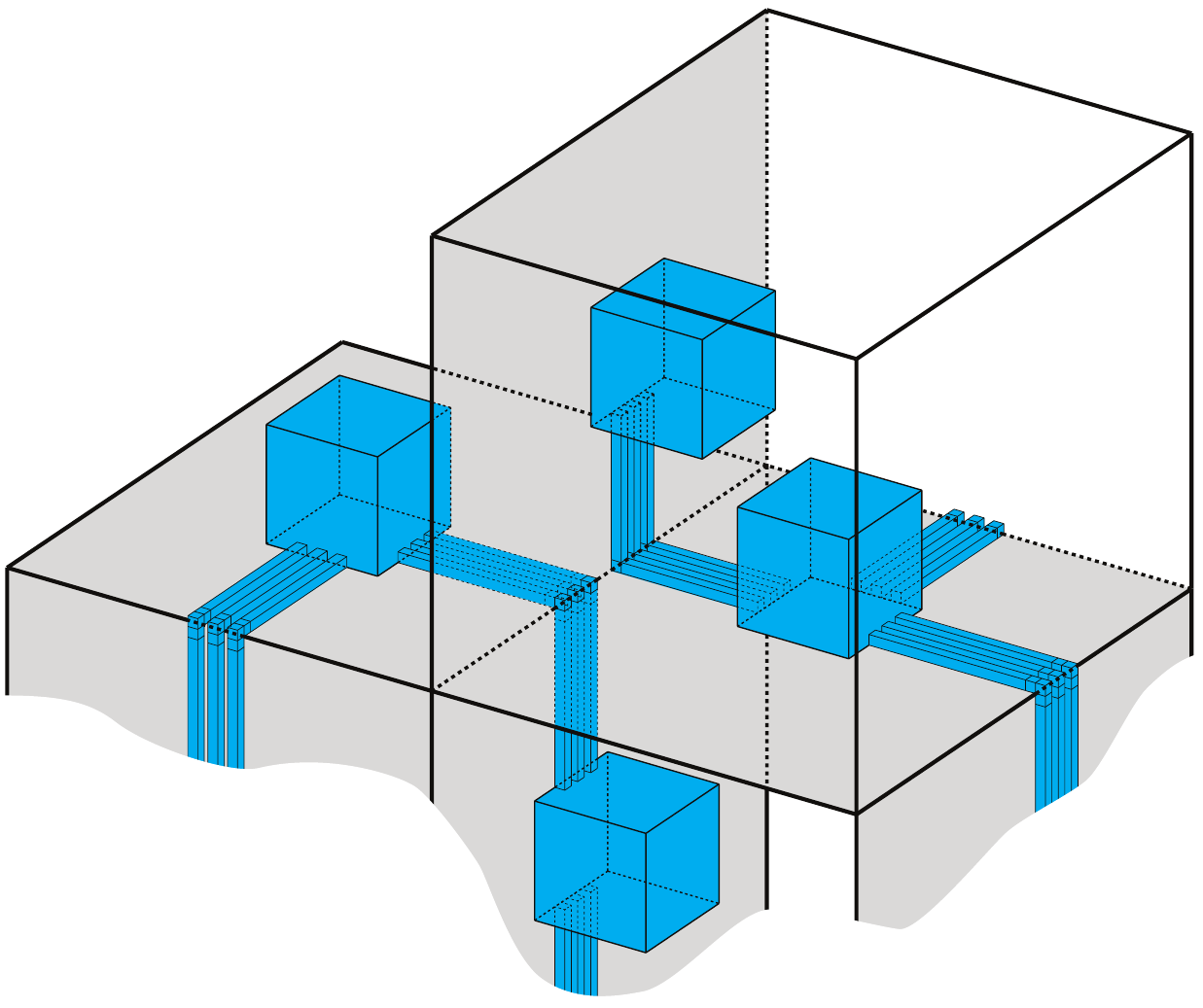}
\put(102,40){$q_1$}
\put(47,24){$q_5$}
\put(94,20){ $q_3$}
\put(10,35){ $q_6$}
\put(17,48){ $q_4$}
\put(67,85){ $q_2$}
\put(50,35){$e$}
\end{overpic}
\caption{Schematic figure of the space of a reservoir-canal system $\sfRC_{\Refine^\nu(U)}(Y)$ - a local picture (not in scale). Cubes in $Y$ are colored in grey.}
\label{fig:Reservoir_Canal_System}
\end{figure}
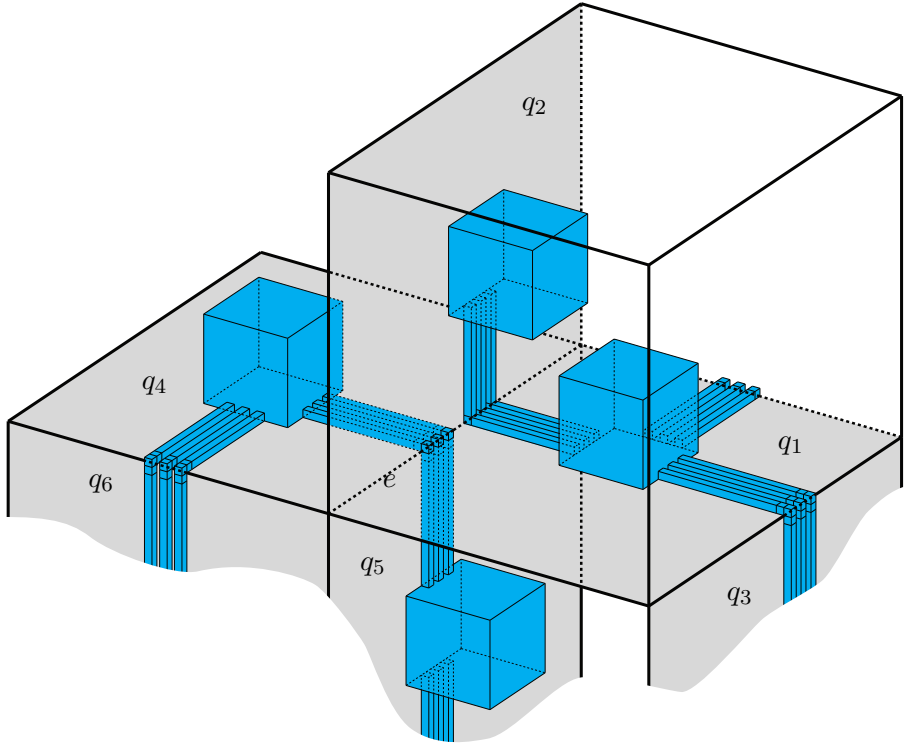

For each $i\in \{1,\ldots, m\}$ and each $q\in Y^{[n-1]}$, we call 
\[
\widehat \sfR_{q;i} = \sfR_{q;i} - \bigcup_{\{q,q'\}\in \cT} \left( \bigcup_{i'>i}  \sfC_{q,\{q,q'\},i'}\right)
\]
the \emph{reservoir of index $i$ adjacent to $q$}.   \index{reservoir-canal system!reservoir} \index{$\widehat \sfR_{q;i}$}

\begin{lemma}
\label{lemma:adjacency-reservoirs}
Complexes $\widehat \sfR_{q;i}$ and $\partial \widehat \sfR_{q;i}$ are adjacently-connected.
\end{lemma}
\begin{proof}
Since each canal stretch $\sfC_{\{q,q'\},i'}$ meets $\sfR_{q;i}$ in an $n$-cube if and only if $i' > i$, and these $n$-cubes are mutually disjoint, we conclude that both adjacency graphs $\Gamma(\widehat \sfR_{q;i})$ and $\Gamma(\partial \widehat \sfR_{q;i})$ are connected. 
\end{proof}

\begin{definition}
\label{def:reservoir-canal-system}
\index{reservoir-canal system}   
\index{$\sfRC_{\Refine^\nu(U)}(Y)_i $}
We call the subcomplex 
\[
\sfRC_{\Refine^\nu(U)}(Y;\cT, \rho)_i = \left(\bigcup_{q\in Y^{[n-1]}} \widehat \sfR_{q;i} \right) \, \cup \left( \bigcup_{\{q,q'\}\in \cT} \sfC_{\{q,q'\};i} \right) \subset \Refine^\nu(U).
\]
the \emph{reservoir-canal system  of index $i$ over $Y$ (associated to the spanning tree $\cT$ and the preference function $\rho$}),  and call  the complex  
\[
\sfRC_{\Refine^\nu(U)}(Y;\cT,\rho) =\bigcup_{i=1}^m \sfRC_{\Refine^\nu(U)}(Y; \cT,\rho)_i \subset \Refine^\nu(U)
\]
 the \emph{reservoir-canal system over $Y$}. 
For brevity, we  write $ \sfRC_{\Refine^\nu(U)}(Y)_i$  and   $\sfRC_{\Refine^\nu(U)}(Y)$ for  $ \sfRC_{\Refine^\nu(U)}(Y;\cT, \rho)_i$ and $\sfRC_{\Refine^\nu(U)}(Y;\cT,\rho)$, respectively.
\end{definition}

\subsubsection{Properties of reservoir-canal systems} 

We record now properties of reservoir-canal systems and their boundaries, which are needed in the subsequent constructions.

We note first that aside from the family $\sfJ_{\Refine^\nu(U)}(Y)$ of connectors, the system $\sfRC_{\Refine^\nu(U)}(Y)$ is contained in $\Refine^\nu(\rho(Y^{[n-1]})$. More precisely, we have the following corollary.

\begin{corollary}\label{cor:RC-core-relation}
From the construction, we have that 
\[ \sfRC_{\Refine^\nu(U)}(Y) \subset     \Refine^\nu(\rho(Y^{[n-1]}))\,\cup \sfJ_{\Refine^\nu(U)}(Y) \subset \Refine^\nu(\Star_U(Y)),\]
and 
\[\Star_{\Refine^\nu(U)}(\sfRC_{\Refine^\nu(U)}(Y))\subset \Refine^\nu(\Star_U(Y)).\]
\end{corollary}

The adjacent connectedness follows from Lemmas \ref{lemma:adjacency-canal-stretch} and \ref{lemma:adjacency-reservoirs}.

\begin{corollary}
\label{cor:adjacency-RC}
The complex $\sfRC_{\Refine^\nu(U)}(Y)_i$  is adjacently-connected for each $i\in\{1,\ldots, m\}$,  and has an adjacently-connected boundary. Moreover, the system  $\sfRC_{\Refine^\nu(U)}(Y)$ is adjacently-connected and  $\bigcup_{i=1}^m \partial \sfRC_{\Refine^\nu(U)}(Y)_i$ is also adjacently-connected. 
\end{corollary}

The space of the reservoir-canal system of index $i$ is $n$-cell. 

\begin{corollary}
\label{cor:RC-is-a-cell}
For each $i\in\{1,\ldots, m\}$, the space of $\sfRC_{\Refine^\nu(U)}(Y)_i$ is an $n$-cell. 
\end{corollary}
\begin{proof}
Let $i\in \{1,\ldots, m\}$. For each $q\in Y^{[n-1]}$, the space of $\widehat \sfR_{q;i}$ is an $n$-cell. Moreover, for each $\{q,q'\}\in \cT$, the space of $\sfC_{\{q,q'\};i}$ is an $n$-cell and $\widehat \sfR_{q;i}\cap \sfC_{\{q,q'\};i}$ is an $(n-1)$-cube, hence the space of $\widehat \sfR_{q;i}\cup \sfC_{\{q,q'\};i}$ is an $n$-cell. The claim follows by induction along the tree $\cT$. 
\end{proof}

Since the complexes $\sfRC_{\Refine^\nu(U)}(Y)_i, i=1,\ldots,m,$ with distinct indices meet only in reservoirs, we have the following. 
\begin{corollary}
\label{cor:reservoir-canal-system-intersection}
For $i\ne i'$, the intersection $\sfRC_{\Refine^\nu(U)}(Y)_i \cap \sfRC_{\Refine^\nu(U)}(Y)_{i'}$ is an $(n-1)$-complex.
\end{corollary}

Finally,  the system $\sfRC_{\Refine^\nu(U)}(Y)$ is far from the boundary $\partial Y$ of $Y$ in the graph distance.

\begin{corollary}
\label{cor:border-of-Y} Suppose that $\partial Y$ is nonempty.
The graph distance, with respect to $\Gamma({\Refine^\nu(U)})$, between an $n$-cube in $\sfRC_{\Refine^\nu(U)}(Y)$ and any $n$-cube in ${\Refine^\nu(U)}$ that meets  $\Refine^\nu(\partial Y)$ is  at least  $3^{\nu-2}$. 
\end{corollary}
\begin{proof}
Since  $(n-2)$-cubes on the boundary of $Y$ are not edges of the tree $\cT$, there are no markers in $\partial Y$. Therefore, there are no canal sections in $\sfRC_{\Refine^\nu(U)}(Y)$ connecting pre-reservoir blocks to $\Star_{\Refine^\nu(U)}(\partial Y)$. The claim now follows from the fact that every pre-reservoir block is located over the center cube of $\Refine^2(q)$ for some  $q\in Y^{[n-1]}$. 
\end{proof}

\subsubsection{A partially ordered partition  of $\sfRC_{\Refine^\nu(U)}(Y)$ determined by tree $\cT$}
\label{sec:RC-partition-UY}
To prepare for a localized perturbation in Section \ref{sec:localization}, we choose and fix a partition 
$\cR \cC_{\Refine^\nu(U)}(Y)$ 
of the reservoir-canal system $\sfRC_{\Refine^\nu(U)}(Y)$ as follows.

We divide each canal stretch into two parts by assigning each connector entirely to one of the two canal sections using the partial order $<_Y$ in $\cT$. For edge $\{q, q'\}\in \cT$ with $q<_Y q'$, we take
\[
\widetilde \sfC_{q,\{q,q'\};i} =\sfC_{q,\{q,q'\};i} \cup J_{\{q,q'\};i} \quad \text{and}\,\,\, \widetilde \sfC_{q',\{q,q'\};i}= \sfC_{q',\{q,q'\};i}.
\]
Let 
\[
\sfRC_{\Refine^\nu(U)}(Y)_{q;i}= \widehat \sfR_{q;i} \cup \bigcup_{\{q,q'\}\in \cT} \widetilde \sfC_{q,\{q,q'\};i},
\]
and 
\[
\sfRC_{\Refine^\nu(U)}(Y)_q = \bigcup_{i=1}^m \sfRC_{\Refine^\nu(U)}(Y)_{q;i}.
\]
Then the collection 
\[ \cR \cC_{\Refine^\nu(U)}(Y)=\{ \sfRC_{\Refine^\nu(U)}(Y)_q \colon q\in Y^{[n-1]}\}\]
 is an essential partition of $\sfRC_{\Refine^\nu(U)}(Y)$. We call $\cR \cC_{\Refine^\nu(U)}(Y)$ a 
\emph{partially ordered partition of $\sfRC_{\Refine^\nu(U)}(Y)$ (uniquely) determined by tree $\cT$}.

By Lemmas \ref{lemma:adjacency-reservoirs} and \ref{lemma:adjacency-canal-stretch}, each $\sfRC_{\Refine^\nu(U)}(Y)_q$ is adjacently-connected and has an adjacently-connected boundary.

\begin{remark}\label{rmk:ambient-essentially-disjoint}
The ambient complex $\Refine^\nu(U)$ is locally Euclidean at each $(n-2)$-cube in  $\sfRC_{\Refine^\nu(U)}(Y)$ which is not a marker in  connectors. The requirement 'each connector belonging entirely to one of the elements in $ \cR \cC_{\Refine^\nu(U)}(Y)$' for the partition ensures each marker belongs to only one element in the partition. This condition simplifies the discussion of local geometry in the future.
\end{remark}

\subsection{Transformation}
\label{sec:reservoir-canal-transformation}
\index{separating complex!transformation}\index{reservoir-canal system!transformation}
Using $\sfRC_{\Refine^\nu(U)}(Y)$, we transform $Y$ into a complex $\Tr(Y)$ in $\Refine^\nu(K)$ which has $m$ additional complementary components than that of $Y$. 

For each $i\in \{1,\ldots, m\}$, we fix a spanning tree $\tau_i \subset \Gamma(\sfRC_{\Refine^\nu(U)}(Y)_i )$ and denote
\[
\tau_Y = \tau_1\cup \cdots \cup \tau_m.
\]
We call the family 
of  $(n-1)$-cubes $Q\cap Q'$ representing the edges in $\tau_Y$,
\[
P_{\tau_Y} = \{ Q\cap Q' \in \sfRC_{\Refine^\nu(U)}(Y)^{[n-1]} \colon \{Q,Q'\}\in \tau_Y \},
\]  \emph{$\tau_Y$-passages.}

\begin{definition}
\label{def:reservoir-canal-transformation}
The \emph{transformation of $Y$} is  the $(n-1)$-subcomplex  \index{$\Tr_{\Refine^\nu(U)}(Y)$} 
\[
\Tr(Y) =\Tr_{\Refine^\nu(U)}(Y)=\left( \Refine^\nu(Y) \cup  \sfRC_{\Refine^\nu(U)}(Y)^{(n-1)}\right) -  P_{\tau_Y}.
\]
\end{definition}

\begin{remark}
Although not emphasized,  the definition of $\Tr(Y)$ depends on the assumptions in Proposition \ref{prop:tr-existence-short}, as well as as the spanning tree $\cT$, the selection of the markers, and the choice of the trees $\tau_i$ in the proof. We omit the recording of these dependencies, as they have no roles in what follows.
\end{remark}

\begin{remark}
\label{rmk:tau_i}
The $\tau_Y$-passages are encoded into the reservoir-canal transformation $\Tr_{\Refine^\nu(U)}(Y)$. Indeed, for each $i=1,\ldots, m$, 
we have that 
\[
\tau_i = \Gammacut \left(\sfRC_{\Refine^\nu(U)}(Y)_i; \, \Tr_{\Refine^\nu(U)}(Y)\cap \sfRC_{\Refine^\nu(U)}(Y)_i \right).
\]
\end{remark}

The following lemma shows that the transformation $\Tr(Y)$ of $Y$ inherits the adjacently-connectedness of $Y$. This is one of the key properties in terms of inductive construction of separating complexes.

\begin{lemma}
\label{lemma:Y_sfRC-connected}
The transformation $\Tr_{\Refine^\nu(U)}(Y)$ is adjacently-connected.
\end{lemma}
\begin{proof}
Since $Y$ is adjacently-connected, the refinement $\Refine^\nu(Y)$ is adjacently-connected. 
We check next that $\Gamma (\sfRC_{\Refine^\nu(U)}(Y)^{(n-1)}  -  P_{\tau_Y})$ is connected. 

For each $i\in \{1,\ldots, m\}$, let $P_i$ be the set of passages associated to the tree $\tau_i$, thus 
\begin{equation}
\label{eq:sfRC-union}
\sfRC_{\Refine^\nu(U)}(Y)^{(n-1)}  -  P_{\tau_Y}  =  \bigcup_{i=1}^m \, (\sfRC_{\Refine^\nu(U)}(Y)_i ^{(n-1)}  -   P_i).
\end{equation}
Since $\sfRC_{\Refine^\nu(U)}(Y)_i$ is an adjacently-connected $n$-complex, the $(n-1)$-subcomplex $\sfRC_{\Refine^\nu(U)}(Y)_i ^{(n-1)}$ is also adjacently-connected. Since $(n-1)$-cubes in $P_i$ are \emph{inside} $\sfRC_{\Refine^\nu(U)}(Y)_i$, 
we conclude that the adjacency graph $\Gamma(\sfRC_{\Refine^\nu(U)}(Y)_i ^{(n-1)}  -  P_i)$ remains connected after their removal.

Since complexes ${\sfRC_{\Refine^\nu(U)}(Y)_i}^{(n-1)}  -   P_i$ and ${\sfRC_{\Refine^\nu(U)}(Y)_j}^{(n-1)}  -   P_j$ having consecutive indices $i$ and $j$, meet in $(n-1)$-cubes, their union is adjacently connected. Thus, by \eqref{eq:sfRC-union}, $\Gamma( \sfRC_{\Refine^\nu(U)}(Y)^{(n-1)}  -  P_{\tau_Y})$ is connected.

The $n$-cubes in the canal stretches in $\sfRC_{\Refine^\nu(U)}(Y)$ have faces in $\Refine^\nu(Y)$ and these faces do not belong to $P_{\tau_Y}$. Therefore, the adjacency graphs $\Gamma(\sfRC_{\Refine^\nu(U)}(Y)^{(n-1)}  -  P_{\tau_Y})$ and $\Gamma(\Refine^\nu(Y))$ belong to the same connected component of $\Gamma(\Tr_{\Refine^\nu(U)}(Y))$. Thus $\Gamma(\Tr_{\Refine^\nu(U)}(Y))$ is connected.
\end{proof}

From Corollary \ref{cor:border-of-Y} and the construction of $\Tr(Y)$, the boundary of $Y$ remains in the boundary of $\Tr(Y)$ after the transformation.

\begin{corollary}\label{cor:boundary-of-Y_sfRC}
The $(n-2)$-complex $\Refine^\nu(\partial Y)$ remains in the boundary $\partial(\Tr_{\Refine^\nu(U)}(Y))$ of the transformation.
\end{corollary}

\subsubsection{Receded subcomplexes} 
\label{sec:receded-complex}
Recall that $G_1,\ldots, G_r$ are the connected components of the cut-graph $\Gammacut(U;Y)$ labeled according to Convention \ref{convention:labeling}, where $r=r(U,Y)\geq m$. We denote
\[
U_j = \Span_U(G_j), \qquad j=1,\ldots, r.
\]

We call  the subcomplex of $\Refine^\nu(U_j)$ remained, after the removal of the reservoir-canal system $\sfRC_{\Refine^\nu(U)}(Y)$, a receded subcomplex.

\begin{definition}\label{def:receded}For $j\in \{1, \ldots, r\}$, 
we call 
\[
\Urec_j= \Refine^\nu(U_j) - \sfRC_{\Refine^\nu(U)}(Y),
\]
a \emph{receded subcomplex of $\Refine^\nu(U_j)$ of index $j$}.
\end{definition}

Complex $\Urec_j$ is adjacently-connected.
We first prove a local version of this property. 

\begin{lemma}
\label{lemma:receded-connected}
For each $j\in \{1,\ldots, r\}$ and each $n$-cube $Q$ in $U_j$, the adjacency graph $\Gamma(\Refine^\nu(Q)\cap \Urec_j)$ is connected. 
\end{lemma}

\begin{proof}
Let $Q$ be an $n$-cube in $U_j$. 
Suppose first that $Q\cap Y=\emptyset$. Then  $\Refine^\nu(Q) \subset \Urec_j$. Hence $\Gamma(\Refine^\nu(Q)\cap \Urec_j) = \Gamma(\Refine^\nu(Q))$ is connected. 

Suppose next that $Q\cap Y \ne \emptyset$. We consider two cases. 
Assume first that $Q$ is not in the image of the preference function $\rho$. 
Then the  intersection $\Refine^\nu(Q)\cap \sfRC_{\Refine^\nu(U)}(Y)$ either (i) has no $n$-cubes, or (ii)  has only $n$-cubes in the connectors in $\sfRC_{\Refine^\nu(U)}(Y)$. In case (i), $\Gamma(\Refine^\nu(Q)\cap \Urec_j) = \Gamma(\Refine^\nu(Q))$ is connected. 
In case (ii), each $n$-cube in the intersection meets   $\Refine^\nu(\partial Q)$ in an $(n-1)$-cell. Thus $\Gamma(\Refine^\nu(Q)\cap \Urec_j)$ is connected.

Assume second that $Q$ is the preference cube of one or more $(n-1)$-cubes in $Y$, under $\rho$. Then the space of $R= \Refine^\nu(Q)\cap \sfRC_{\Refine^\nu(U)}(Y)$ has an essential partition into $n$-cells, each of which is either the space of an $n$-cube in $\Refine^2(Q)$ with a face on $|\partial Q|$, or the space of an $n$-cube in $\Refine^\nu(Q)$ having one or two faces on $|\partial   Q|$. Since 
\[
\Refine^\nu(Q) \cap \Urec_j= \Refine^\nu(Q) - R,
\]
each $n$-cube in $\Refine^\nu(Q) \cap \Urec_j$ may be connected,
 by a chain of adjacent $n$-cubes, 
in $\Refine^\nu(Q) \cap \Urec_j$ to the center $n$-cube  of $\Refine^\nu(Q)$.  Thus $\Gamma(\Refine^\nu(Q) \cap \Urec_j)$ is connected. 
\end{proof}

\begin{corollary}\label{cor:receded-connected}
Each complex $\Urec_j$ is adjacently-connected.
\end{corollary}

\begin{proof}
Let $\{Q,Q'\} \in \Gamma(U_j)$ and $q=Q\cap Q'$. Since both graphs $\Gamma(\Refine^\nu(Q)\cap \Urec_j)$ and  $\Gamma(\Refine^\nu(Q')\cap \Urec_j)$ are connected and they may be connected to each other through some $(n-1)$-cubes in $\Refine^\nu(q)$, graphs $\Gamma(\Refine^\nu(Q)\cap \Urec_j)$ and $\Gamma(\Refine^\nu(Q')\cap \Urec_j)$ belong to the same connected component of  $\Gamma(\Urec_j)$. Since $\Gamma(U_j)$ is connected, we conclude that   $\Gamma(\Urec_j)$ is connected.
\end{proof}

By Remark \ref{rmk:tau_i} and Corollary \ref{cor:receded-connected}, cut-graph $\Gammacut(\Refine^\nu(U);\Tr_{\Refine^\nu(U)}(Y))$ has $r+m$ connected components.

\begin{corollary}
\label{cor:reservoir-canal-components}
Let $\Tr_{\Refine^\nu(U)}(Y)$ be a transformation of $Y$ in $\Refine^\nu(U)$. Then the connected components of  $\Gammacut(\Refine^\nu(U);\Tr_{\Refine^\nu(U)}(Y))$ are 
 \[
\Gamma(\Urec_1),\ldots,  \Gamma(\Urec_r), \,\tau_1,\ldots,\tau_m.
\]
Moreover, each $\tau_i$ is a tree of length at most $\lambda=\lambda(\# Y^{[n-1]}, \mu, \nu)$. 
\end{corollary}

\begin{proof}[Proof of Proposition \ref{prop:tr-existence-short}]Property \eqref{item:rplusm} in the Proposition follows from
Corollary \ref{cor:reservoir-canal-components}. 
Property \eqref{item:pre-Wada} holds because, for each $i\in \{1,\ldots, m\}$, there is a $\sfRC(Y)_{q;i}$ over each $q\in Y^{[n-1]}$  by the construction.
\end{proof}

\subsubsection{Indentation in realizations}
\label{sec:lift-RC} Recall that $G_1,\ldots, G_r$ are the connected components of the cut-graph  $\Gammacut(U;Y)$, $U_j=\Span_U(G_j)$, $\cR_U(G_j)$ is the realization of $G_j$,
 $\cR_U(\Gammacut(U;Y))$ is the realization of  $\Gammacut(U;Y)$, and 
\[ 
\pi_{(U;Y)} \colon \cR_U(\Gammacut(U;Y)) \to U
\]
is the $\Gammacut(U;Y)$-quotient map. 

We prove next that the lift of $\sfRC_{\Refine^\nu(U)}(Y)$ in $\cR_U(G_j)$,
\[D_U(G_j)= \pi_{(U;Y)}^{-1}\left(\sfRC_{\Refine^\nu(U)}(Y) \cap \Refine^\nu(U_j)\right),\]
is an indentation in $\cR_U(G_j)$, for each $ j= 1,\ldots, r$.

\begin{proposition}
\label{prop:reservoir-canals-as-indentations}
For each $j\in \{1,\ldots,r\}$, the complex $D_U(G_j)$ is an indentation in $\cR_U(G_j)$. 
\end{proposition}

First we show that the lift of each canal stretch $\sfC_{\{q,q'\};i}$ is an indentation in $\cR_U(G_j)$. Note that  indices $i$ and $j$ have no relation.

\begin{lemma}
\label{lemma:canal-stretch-as-indentation}
For $j\in \{1,\ldots, r\}$, the lift $\pi_{(U;Y)}^{-1}\left(\sfC_{\{q,q'\};i}\right)\cap \Refine^\nu(\cR_U(G_j))$
of any canal stretch $\sfC_{\{q,q'\};i}$, if nonempty, is an indentation in  $\cR_U(G_j)$. 
\end{lemma}

We make an observation before giving the proof.
\begin{remark}
Although a canal stretch $\sfC_{\{q,q'\};i}$ is adjacently-connected, its lift in $\Refine^\nu(\cR(\Gammacut(U;Y))$ need not be connected. 
The splitting of the lift 
occurs when $Y$ enters the interior of the connector $|\sfJ_{\{q,q'\};i}|$; see Figure \ref{fig:lift-of-RC} for an illustration.
\end{remark}

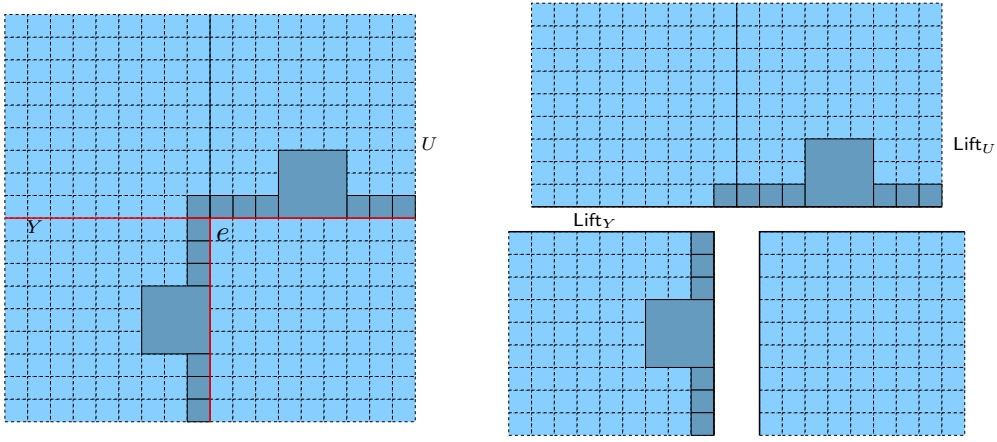
\begin{figure}[htp]
\centering
\begin{overpic}[scale=.30,unit=1mm]{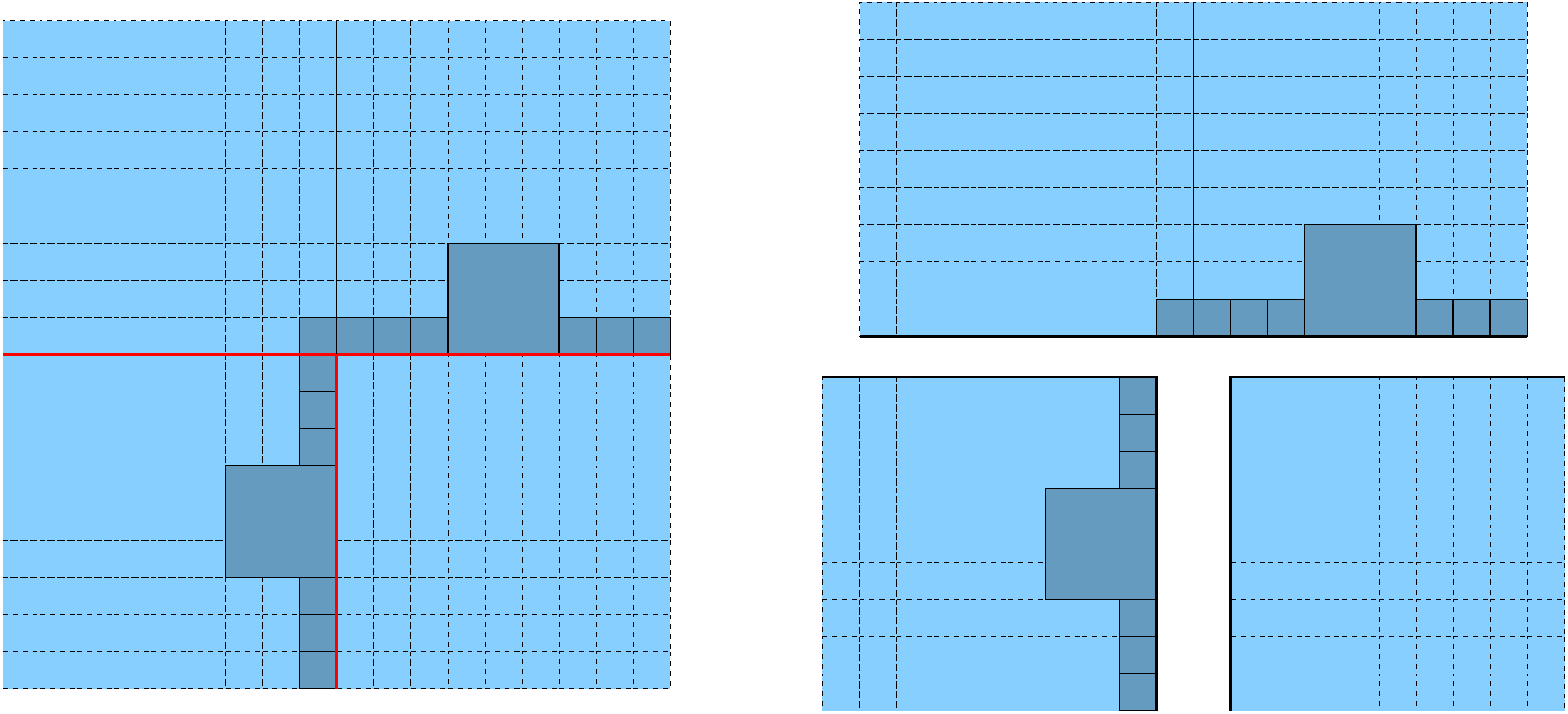} 
\put(3,27){\tiny $Y$}
\put(55,38){\tiny $U$} 
\put(125,38){\tiny $\Real_U$} 
\put(75,28){\tiny $\Real_{Y}$}
\put(28, 26){$e$}
\end{overpic}

\caption{A schematic picture of a reservoir-canal system (dark blue) and its lift.}
\label{fig:lift-of-RC}
\end{figure}

\begin{proof}[Proof of Lemma \ref{lemma:canal-stretch-as-indentation}]
Let $\sfC_{\{q,q'\};i}$ be a canal stretch having a nonempty lift $\pi_{(U;Y)}^{-1}(\sfC_{\{q,q'\};i}) \cap \Refine^\nu( \cR_U(G_j))$.
Let $P$  be a subcomplex of $\sfC_{\{q,q'\};i}$ whose lift  $\pi_{(U;Y)}^{-1}(P)$ is a connected component of $\pi_{(U;Y)}^{-1}(\sfC_{\{q,q'\};i})\cap \Refine^\nu( \cR_U(G_j))$. Then $\Gamma(P)$ is a connected component of  $\Gammacut(\sfC_{\{q,q'\};i};Y)$ and is isomorphic to $\Gamma\left(\pi_{(U;Y)}^{-1}(P\cap \Refine^\nu(U_j))\right)$. 

Complex $P$ then has
one of the following three forms:
 \begin{enumerate}
\item  the entire $\sfC_{\{q,q'\};i}$, or  \label{item:P_adjacent_both}
\item the union of  a canal section in $\sfC_{\{q,q'\};i}$ and an adjacently-connected subcomplex $R$ of the connector, or \label{item:P_adjacent_one}
\item  an adjacently-connected subcomplex $R$ of the connector and it is not adjacent to either canal section.  \label{item:P_isolated}
\end{enumerate}

Each $n$-cube $Q\in P^{[n]}$ belonging to a canal section has a face in the the same $(n-1)$-cube in $\partial U_j$. Thus the lift of a canal section  is a spectral-cube-indentation in $\cR_U(G_j)$. 
The partial star of $R$ in Cases \eqref{item:P_adjacent_one} or  \eqref{item:P_isolated}, has two faces on $Y$, hence  $\pi_{(U;Y)}^{-1}(R)$ is a star in $\Refine^\nu(\cR_U(G_j))$. 
Therefore, $\pi_{(U;Y)}^{-1}\left(P\cap \Refine^\nu(U_j)\right)$ is  an indentation  in $\cR_U(G_j)$. 
Since a union of mutually disjoint indentations is an indentation, the claim 
follows. 
\end{proof}

\begin{proof}[Proof of Proposition \ref{prop:reservoir-canals-as-indentations}]
Every pre-reservoir-cube $Q$ associated to the system $\sfRC_{\Refine^\nu(U)}(Y)$ is a  refinement of  an $n$-cube $C_Q$ in $\Refine^{2}(K)$,  hence $C_Q$ is the spectral cube of $Q$.
Thus,  if a pre-reservoir-cube $Q$ belongs to $U_j$, then  $\pi_{(U;Y)}^{-1}(C_Q)$  is a spectral cube of $D_U(G_j)$ in $\Refine^\nu(\cR_U(G_j))$ and has a face on $\partial U_j$.

From the placement of reservoirs and canals in the construction, a pre-reservoir-cube and a canal stretch can meet only  in a face of the canal stretch. 
In particular, the intersection of  three distinct spectral cubes in $D_U(G_j)$ is empty. Hence $D_U(G_j)$ is an indentation in $\cR_U(G_j)$.
\end{proof}


\subsection{Channeling}
\label{sec:single-channeling}
Reservoir-canal system provides a 
method to transform
$Y\subset U$ into a new complex  $\Tr(Y) \subset \Refine^\nu(U)$. 
While the original cut-graph $\Gammacut(U;Y)$ has $r$ components, where $r=r(U,Y)\geq m$,  
  the cut-graph $\Gammacut(\Refine^\nu(U);\Tr(Y))$ has $m$ additional components $\tau_1,\ldots,\tau_m$ created from the reservoir-canal systems as seen in Corollary \ref{cor:reservoir-canal-components}.
Channeling, which we now discuss, connects  $\tau_1,\ldots,\tau_m$  to $\Gamma(\Urec_1),\ldots, \Gamma(\Urec_m)$, by removing a gate between  
$\sfRC(Y)_i$ and $\Urec_i$ for each $i\in \{1,\ldots, m\}$.

\subsubsection{Gates on $\sfRC(Y)_i$}

Let $(U,Y,\rho_Y)$ be admissible in 
$(\Refine^{k \nu}(K), Z_k)$ and assume 
the labeling of cut-graphs as in Convention \ref{convention:labeling}.

Since $\rho_Y$ is admissible, we may fix for  each $i\in \{1,\ldots, m\}$, an $n$-cube  $C_i\in  \sfRC(Y)_i\cap \Refine^\nu(U_i)$ which is \emph{well-located in the system $\sfRC(Y)$}, that is,  $C_i $ belongs to a canal section in $\sfRC(Y)_i$ but  not adjacent to any pre-reservoir-cube and is closer to a pre-reservoir-cube than a connector in graph distance.
 
Let $\omega_i$ be the unique face of $C_i$  which does not meet $Y$, in other words,  $\omega_i$ is in the interior of $U_i$.  We connect  $\Gamma(\Urec_i)$ and $ \tau_i$ by $\omega_i$. Since $\omega_i \in \sfRC(Y)_i \cap \Urec_i$, we refer to $\omega_i$ as a \emph{gate between $\sfRC(Y)_i$ and $\Urec_i$}.

\begin{definition}
\label{def:Channeling}
The subcomplex  \index{separating complex!channeling} \index{$\Channel(Y) $}
\[
\Channel(Y) = \Tr(Y) -\bigcup_{i=1}^m \{\omega_i \} \subset  \Refine^\nu(U)
\]
is called a \emph{channeling of $Y$ (based on $\rho_Y$ and $\cT_Y$}). 
\end{definition}\index{channeling}

\begin{remark}\label{rmk:sheet-neighborhood}
For future discussion, we fix now, for each $i\in \{1,\ldots, m\}$,
a \emph{sheet neighborhood $\Lambda(\omega_i)$}  of $\omega_i$ on the boundary $\partial(\Urec_i)$ of the receded complex $\Urec_i$, by setting
\[\Lambda(\omega_i) = |\Star_{\Refine^\nu(U)}(\omega_i)\cap \partial (\Urec_i)|.\]
Note that  $\Lambda(\omega_1),\ldots, \Lambda(\omega_m)$ are mutually disjoint $(n-1)$-cells and that $\Lambda(\omega_i)$  contains the corresponding $\omega_i$ in its interior.
\end{remark}

We record now the basic properties of the channeling.
\begin{lemma}
\label{lemma:Channeling-basic-properties}
Let $\Channel(Y)$ be a channeling of $Y$. Then $\Channel(Y)$ is adjacently connected, $\Channel(Y)$ has the pre-Wada property with respect to $Y$ as stated in Proposition \ref{prop:tr-existence-short}, and $\Channel(Y)$ is an $\lambda$-perturbation of $Y$ for a constant $\lambda=\lambda(\#(Y^{[n-1]}),\mu,\nu)$. Moreover, the connected components of the graph $\Gammacut(\Refine^\nu(U);\Channel(Y))$ are
\[
G_{\Channel(Y);i}= 
\begin{cases} \Gamma(\Urec_i) \cup \tau_i \cup \{ \omega_i\} & \quad \text{for}\,\,\,  i=1,\ldots, m, \\
 \Gamma(\Urec_i) & \quad \text{for}\,\,\,  i=m+1,\ldots, r.
\end{cases}
\]
\end{lemma}

\begin{proof}
For the first claim it suffices to observe that, since $\Tr(Y)$ is adjacently-connected by Lemma \ref{lemma:Y_sfRC-connected} and the removal of mutually disjoint $(n-1)$-cubes $\omega_1,\ldots, \omega_m$ does not change the connectedness, we have that  $\Channel(Y)$ is adjacently connected. The second part has been proved in Proposition \ref{prop:tr-existence-short}.

For the third claim, we observe first that the structure of the connected components $G_{\Channel(Y);i}$ of $\Gammacut(\Refine^\nu(U); \Channel(Y))$ follows immediately from Corollary \ref{cor:reservoir-canal-components} and the choice of gates $\omega_i$.
Let $U'_i=\Span_{\Refine^\nu(U)}(G_{\Channel(Y);i})$ for $i=1,\ldots, r$ and let
\[\pi_{(\Refine^\nu(U);\Channel(Y))}\colon \cR_{\Refine^\nu(U)}(\Gammacut(\Refine^\nu(U); \Channel(Y)))\to \Refine^\nu(U)   \]
be the $\Gammacut(\Refine^\nu(U); \Channel(Y))$-quotient map.

Since $U'_i=(\Refine^\nu(U_i)-\sfRC(Y))\,\cup \,\sfRC(Y)_i$ and the adjacency graph of the lift $\pi_{(\Refine^\nu(U);\Channel(Y))}^{-1}(\sfRC(Y)_i)$ is the tree $\tau_i$, the channeling $\Channel(Y)$ is an $\lambda$-perturbation of $Y$ for a constant $\lambda=\lambda(\#(Y^{[n-1]}),\mu,\nu)$.
\end{proof}

\subsubsection{Summary on channeling}\label{sec:summary}
We summarize the channeling process and collect properties of  $\Channel(Y)$ in the next  proposition; we refer to Sections \ref{sec:reservoir-canal-transformation} and \ref{sec:single-channeling} for notations. All properties listed here have been proved above.
\begin{proposition}
\label{prop:summary-Channeling}
Retain all assumptions and the labeling in Proposition \ref{prop:tr-existence-short}. 
Guided by a preference function $\rho_Y\colon Y^{[n-1]}\to U^{[n]}$ and a spanning tree $\cT_Y$ in $\Gamma(Y)$, we may construct in $\Refine^\nu(U)$ the following:
\begin{enumerate}
\item an adjacently connected reservoir-canal system $\sfRC(Y)=\bigcup_{i=1}^m \sfRC(Y)_i$ consisting of essentially disjoint  
$n$-subcomplexes $\sfRC(Y)_1,\ldots, \sfRC(Y)_m$, each of which is adjacently connected and has a spanning tree, $\tau_i$, of size at most $\lambda=\lambda(\#(Y^{[n-1]}), \mu, \nu)$, which satisfies 
\begin{enumerate}
\item $\sfRC(Y)\subset \Refine^\nu(\Star_U(Y))$, and
\item $\sfRC(Y)\cap \Star_{\Refine^\nu(U)}(\Refine^\nu(\partial Y)) = \emptyset;$ \label{item:RC-existence-7}
\end{enumerate}
\item  an $(n-1)$-subcomplex $ \Tr(Y) \subset \Refine^\nu(Y) \cup  \sfRC(Y)^{(n-1)}$ for which the connected components of 
 $\Gammacut(\Refine^\nu(U);\Tr(Y))$ are 
 \[\Gamma(\Urec_1),\ldots,\Gamma(\Urec_r),\, \tau_1, \ldots, \tau_m,\]
where $r=r(U,Y)\geq m$, $\Urec_i= \Refine^\nu(U_i) - \sfRC(Y)$ for $i\in\{1,\ldots,r\}$, and $\tau_i$ is a spanning tree in $\Gamma(\sfRC(Y)_i)$ for $i\in\{1,\ldots,m\}$;
\item  for each $i\in \{1,\ldots, m\}$, an $(n-1)$-cube $\omega_i\in \partial \sfRC(Y)_i\cap \partial \Urec_i$
contained in the interior of $U_i$ for which their sheet neighborhoods $\Lambda(\omega_i)$ are mutually disjoint;
\item an adjacently connected $(n-1)$-subcomplex
\[\Channel(Y)= \Tr(Y) -\bigcup_{i=1}^m \{\omega_i\} \subset  \Refine^\nu(U),\] 
for which \label{item:summary-Ch(Y)}
 \begin{enumerate}
\item  the connected components of $\Gammacut(\Refine^\nu(U); \Channel(Y))$ are
\[G_{\Channel(Y);i}= 
\begin{cases} \Gamma(\Urec_i) \cup \tau_i \cup \{ \omega_i\} & \quad \text{for}\,\,\,  i=1,\ldots, m, \\
 \Gamma(\Urec_i)  & \quad \text{for}\,\,\,  i=m+1,\ldots, r,
\end{cases}\]
\item  $\Channel(Y)$ satisfies the pre-Wada property with respect to $Y$ in Proposition \ref{prop:tr-existence-short}, and \label{item:channel-wada}
\item  $\Channel(Y)$ is a $\lambda$-perturbation of $Y$. \label{item:perturbation}
\end{enumerate}
\end{enumerate}
\end{proposition}

\begin{proof}
For (1), see Corollaries \ref{cor:adjacency-RC} and \ref{cor:reservoir-canal-system-intersection}. The length estimate for the trees $\tau_i$ is recorded in Corollary \ref{cor:reservoir-canal-components}. For (1a), see Corollary \ref{cor:RC-core-relation} and for (1b) Corollary \ref{cor:boundary-of-Y_sfRC}. For (2), see Corollary \ref{cor:reservoir-canal-components}, and, for (3), Remark \ref{rmk:sheet-neighborhood}. Item (4) follows from Lemma \ref{lemma:Channeling-basic-properties}.
\end{proof}

Whereas the previous proposition describes the structure of the complementary components of $\Channel(Y)$ relative to the complementary components of $Y$, the following corollary gives a similar structure for the realization $\cR_{\Refine^\nu(U)}(\Gammacut(\Refine^\nu(U);\Channel(Y)))$ with respect to $\cR_U(\Gammacut(U;Y))$.

\begin{corollary}\label{cor:summary-Channeling}
 In the setting of Proposition \ref{prop:summary-Channeling}, 
let 
\[\pi_{(U;Y)}\colon \cR_U(\Gammacut(U;Y))\to U\]
be the canonical quotient map which maps $\cR_U(G_i)$ onto $U_i$. 
Then the lift 
\[D_U(G_i) =\pi_{(U;Y)}^{-1}(\sfRC(Y) \cap \Refine^\nu(U_i))\]  is an indentation in  $\cR_U(G_i)$, and 
$\cR_{\Refine^\nu(U)}(\Gammacut(\Refine^\nu(U);\Channel(Y)))$
has an essential partition, 
\[\cR_{\Refine^\nu(U)}(G_{\Channel(Y);i})= 
\begin{cases} \cR_{\Refine^\nu(U)}(\Gamma(\Urec_i)) \cup \mathsf{t}_i & \quad \text{for}\,\,\,  i=1,\ldots, m, \\
\cR_{\Refine^\nu(U)}( \Gamma(\Urec_i))  & \quad \text{for}\,\,\,  i=m+1,\ldots, r,
\end{cases}
\]
where $\mathsf{t}_i$ is a tunnel whose adjacency graph is  tree $\tau_i$, and the intersection  $\mathsf{t}_i  \cap \cR_{\Refine^\nu(U)}(\Gamma(\Urec_i)) =\pi_{(U;Y)}^{-1}(\omega_i)$.
\end{corollary}

\begin{proof}
The lift $D_U(G_i)$ is an indentation by Proposition \ref{prop:reservoir-canals-as-indentations} and the essential partition of the realization $\cR_{\Refine^\nu(U)}(\Gammacut(\Refine^\nu(U);\Channel(Y)))$ follows from item (4a) of the previous proposition and the role of gates $\omega_i$.
\end{proof}



\section{Evolution of separating complexes}
\label{sec:evolution-seq}

We return now to the context of a separating complex $Z$ in a cubical $n$-complex $K$. The task in this section is to construct a sequence $ (Z_k)$ of separating complexes for the Evolution Theorem.

\restateEvolution*

We assume the Standing assumptions \ref{standing:Z} throughout this section, and fix
\begin{itemize}
\item a preference function $\rho_Z \colon Z^{[n-1]}\to K^{[n]}$ for which $\rho_Z(Z^{[n-1]})$ contains at least an $n$-cube in each $\Comp_K(Z;\Sigma_i), i=1,\ldots,m$, and 
\item a partially ordered spanning tree $\cT_Z$ in $\Gamma(Z)$. 
\end{itemize}
Fix also an integer 
\[\lambda_\loc=\lambda_\loc(n,\mu, \nu)\]
 which is an upper bound of the number of $n$-cubes and  $(n-1)$-cubes in $\Refine^\nu(\Star_K(Q))$  for all $Q\in K^{[n]}$.

We define channeling $Z_1=\Channel(Z)$ in Section \ref{sec:Z-one}, construct  $Z_2=\Channel(Z_1)$ by a localized channeling of $Z_1$ in Section \ref{sec:localization}, and then iterate the localized procedure to obtain the entire sequence $(Z_k)$ in Section \ref{sec:evolution}. 
Geometry between cores is discussed in Section \ref{sec:core-geometry}, and quasiconformal stability of  $(Z_k)$ is proved in Section \ref{sec:quasiconformality}. 

The reason for using localized channeling after the first step is to ensure that tunnels attached at subsequent steps have a hierarchical relationship  and are uniformly bounded in size. These properties are needed for establishing the quasiconformal stability.

\subsubsection*{Notation}

Throughout this section, we denote 
\[
\pi_k=\pi_{(\Refine^{k\nu}(K);Z_k)}\colon  \cR_{\Refine^{k\nu(K)}}(\Gammacut(\Refine^{k\nu}(K);Z_k)) \to \Refine^{k\nu}(K) 
\]
for the canonical quotient map associated to the realization, whenever $Z_k$ is a separating complex in $\Refine^{k\nu}(K)$. Recall that, we use the notation $\Comp(Z_k;\Sigma_i)$ for the connected component of $\Refine^{\nu k}(K)$ separated by $Z_k$ and containing $\Sigma_i$, and that $\Real(Z_k;\Sigma_i)=\pi_k^{-1}(\Comp(Z_k;\Sigma_i))$. 
Recall also that the subcomplex $\Core(Z_k;\Sigma_i)$ lifts isomorphically in $\pi_k$, and we have identified $\pi_k^{-1}\Core(Z_k;\Sigma_i)$ with $\Core(Z_k;\Sigma_i)$. For this reason, we call $\Core(Z_k;\Sigma_i)$ also a core $\Real(Z_k;\Sigma_i)$.

\subsection{Construction of $Z_1=\Channel(Z)$ by channeling}
\label{sec:Z-one}

Let $U=K$ and $Y= Z$ in Section \ref{sec:single-channeling}. Proposition \ref{prop:summary-Channeling},  applied to  \[(K, \,Z,\, \rho_Z,\, \cT_Z), \]
 yields a channeling
\[
Z_1 = \Channel(Z).
\]

Since $Z_1$ is obtained by channeling, the cut-graph
$\Gammacut(\Refine^\nu(K);Z_1)$ has exactly $m$ components $G_{Z_1;1}, \ldots, G_{Z_1;m}$, one corresponding to each $\Sigma_i$, and 
\[\Comp(Z_1;\Sigma_i)= \Span_{\Refine^\nu(K)}(G_{Z_1;i}).\]

To show  that $Z_1$ is a separating complex, we recall briefly properties of some of the complexes in the construction.
By Proposition \ref{prop:tr-existence-short} and property \eqref{item:RC-existence-7} in Proposition \ref{prop:summary-Channeling}, the reservoir-canal system $\RC(Z)$ satisfies 
\[\Star_{\Refine^{\nu}(K)}(Z_1)\subset  \Star_{\Refine^{\nu}(K)}(\sfRC(Z) \cup \Refine^\nu(Z)) \subset \Refine^\nu(\Star_K(Z)),\]
which yields immediately the core-expanding property.

\begin{corollary}\label{cor:core-Z-Z_1}
For each $i\in \{1,\ldots, m\}$, 
\[
|\Core(Z;\Sigma_i)|\subset |\Core(Z_1;\Sigma_i)| \setminus |(\partial \Core(Z_1;\Sigma_i))_{\mathrm{inner}}|.
\]
\end{corollary}

Note from Proposition \ref{prop:reservoir-canals-as-indentations} that 
\[
\D(Z;\Sigma_i) =\pi_1^{-1}(\sfRC(Z)\cap \Refine^\nu(\Comp(Z;\Sigma_i)))
\] 
is an indentation in $\Real(Z_1;\Sigma_i)$. Corollary \ref{cor:summary-Channeling} may now be restated as follows (cf. Theorem \ref{theorem:separating-complex-existence}).

\begin{corollary}
\label{cor:channel-Realization-structure}
For $i=1,\ldots, m$,
\[
\Real(Z_1;\Sigma_i)=\Realdented(Z;\Sigma_i) \cup \tunnel(Z;\Sigma_i),
\]
where
\begin{enumerate}
\item $ \Realdented(Z;\Sigma_i)=\Refine^\nu(\Real(Z;\Sigma_i)) - \D(Z;\Sigma_i)$,
\item $\tunnel(Z;\Sigma_i)$ is a tunnel whose adjacency graph is a tree $\tau_{Z;i}$ of size at most $\lambda(n, \mu, \nu, \#Z^{[n-1]})$, and
\item $\Realdented(Z;\Sigma_i) \cap \tunnel(Z;\Sigma_i)= \pi_1^{-1}(\omega_{Z;i})$ is the lift of a gate.
\end{enumerate}
\end{corollary}

Indentation-flattening Theorem \ref{theorem:flattening-indentation} and the Tunnel-contracting Proposition \ref{prop:tunnel-contracting} now yield the following bilipschitz expansion property.

\begin{corollary}
\label{cor:Z_1-bilipschitz-expanding}
There exist a constant $L=L(n,K)\ge 1$, depending only on the dimension $n$ and complex $K$, and surjective  $L(n,K)$-bilipschitz homeomorphisms 
\[ 
|\Real(Z;\Sigma_i)| \xrightarrow{f} |\Realdented(Z;\Sigma_i)| \xrightarrow{g} |\Real(Z_1;\Sigma_i)|,
\]
whose composition is the identity on the core of $\Real(Z;\Sigma_i)$.
\end{corollary}

We summarize the properties of $Z_1$ as follows.
\begin{corollary}
\label{cor:channel-separating-complex}
The channeling $Z_1=\Channel(Z)$ of $Z$ is a separating complex in $\Refine^\nu(K)$, and it has the relative Wada property  with respect to $Z$, is core-expanding with respect to $Z$, and is an $\lambda(n,K,Z)$-perturbation  of $Z$.
\end{corollary}

\begin{proof}

We verify first that $Z_1 = \Channel(Z)$ is a separating complex of $\Refine^\nu(K)$. 

Since $Z\cap \partial K=\emptyset$, we have that $Z_1(q) \cap \partial \Refine^\nu(K)=\emptyset$ for each $q\in Z^{[n-1]}$. Thus $Z_1 \cap \partial \Refine^\nu(K)=\emptyset$. This verifies the first property in the definition of the separating complex. By \eqref{item:summary-Ch(Y)} in Proposition \ref{prop:summary-Channeling}, $\Gamma(Z_2)$ is connected. This verifies the second property. 

We observe now that the construction immediately yields the properties $\Comp_{\Refine^\nu(K)}(Z_1;\Sigma_i) \cap \partial \Refine^\nu(K) = \Sigma_i$ and $\bigcup_{i=1}^m \Comp_{\Refine^\nu(K)}(Z_1;\Sigma_i) = \Refine^\nu(K)$. Thus the third condition holds. 

By Corollary \ref{cor:Z_1-bilipschitz-expanding}, $|\Real(Z_1;\Sigma_i)| \approx |\Real(Z;\Sigma_i)| \approx |\Sigma_i|\times [0,1]$. Thus the last condition holds and $Z_1$ is a separating complex of $\Refine^\nu(K)$.

Core-expanding property has been shown in Corollary \ref{cor:core-Z-Z_1} and the other properties are listed in (4) of Proposition \ref{prop:summary-Channeling} and in Lemma \ref{lemma:Channeling-basic-properties}.
\end{proof}

\subsubsection{Relative John property of $Z_1$ with respect to $Z$}
As a preparation for the proof of John property of the  Lakes of Wada, we prove now a relative John type property for the cores associated to $Z$ and $Z_1$. 
\begin{lemma}
\label{lemma:Z_1-core-expanding-John}
Each $n$-cube $Q$ in 
\[
P = \Core(Z_1;\Sigma_i) - \Refine^\nu(\Core(Z;\Sigma_i))
\]
can be connected to an $n$-cube $Q'$ in $P$ having a face in the inner boundary component $\Refine^\nu(\partial \Core(Z;\Sigma_i) - \partial K)$ of $\Core(Z;\Sigma_i)$, 
by a chain of adjacent $n$-cubes of length at most $3^\nu \theta(K)$, where 
\[
\theta(K) = \max\left\{\#\{ Q\in \Star_K(\{p\})^{[n]}\} \colon p\in K^{[0]}  \right\}.
\] 

\end{lemma}

\begin{proof}
Since
\[
\Core(Z;\Sigma_i) = \Comp(Z;\Sigma_i) - \Star_K(Z)
\]
and 
\[
\Core(Z_1; \Sigma_i) = \Comp(Z_1;\Sigma_i) - \Star_{\Refine^\nu(K)}(Z_1),
\]
we have that
\[
P= \Comp(Z_1;\Sigma_i) \cap \left( \Refine^\nu(\Star_K(Z))- \Star_{\Refine^\nu(K)}(Z_1) \right).
\]
By the construction of transformation $\Tr(Z)$, we have that each $n$-cube in $\Comp(Z_1;\Sigma_i)$, which is not contained in $\Refine^\nu(\Comp(Z;\Sigma_i)$, has a face in $Z_1$, and is hence not in $P$. Thus $P \subset \Refine^\nu(\Comp(Z;\Sigma_i))$ and we have that
\[
P = \Refine^\nu( \Comp_K(Z;\Sigma_i) \cap \Star_K(Z)) - \Star_{\Refine^\nu(K)}(Z_1).
\]

Let now $Q$ be an $n$-cube in $P$ and let $C$ be an $n$-cube in $\Comp_K(Z;\Sigma_i)\cap \Star_K(Z))$, whose refinement contains $Q$. Since $K$ is a good complex and $Z$ is in the interior of $K$, by the definition of core, $C$ meets $\Core_K(Z;\Sigma_i)$ (at least) in a vertex $v$. Since $K$ is a good complex, the star $\Star_K(v)$ of $v$ is adjacently connected. Since $\Star_K(v)$ contains an $n$-cube, say $C''$, in $\Core_K(Z;\Sigma_i)$, we may fix a chain $C = C_1,\ldots, C_r=C'$ adjacent cubes contained in $P$, where $r<\theta(K)$, from $C$ to a cube $C'$ sharing a face with $C''$.
It is now easy to find a chain of adjacent cubes through the refinements $\Refine^\nu(C_1),\ldots, \Refine^\nu(C_r)$ from $Q$ to a cube $Q'$, which has length at most $3^\nu \theta(K)$, where the multiplier $3^\nu$ is an upper bound for the number of $n$-cubes needed in $\Refine^\nu(C_i)$ to connect a given $n$-cube in $\Refine^\nu(C_i)$ to any of the faces of $C_i$ by a chain of adjacent $n$-cubes in $\Refine^\nu(C_i)$.
\end{proof}

\subsubsection{A preference function $\rho_1$ on $Z_1^{[n-1]} $ nearly nested in $\rho_0$}
\label{sec:channeling-preference-function}
Recall that $\rho_0=\rho_Z\colon Z^{[n-1]} \to K^{[n]}$ is the preference function used in channeling $Z$, and 
\[
Z_1=\Channel(Z) \subset  \Refine^\nu(Z)  \cup  \sfRC(Z)^{(n-1)}. 
\]
For the construction of $Z_2$, we fix a preference function 
\[\rho_1\colon Z_1^{[n-1]} \to  \Refine^\nu(K)^{[n]}\]
which satisfies the (nearly nested) conditions:
 \begin{enumerate}
\item $\rho_1(q^*)\in {\sfRC(Z)}^{[n]}$\, if $q^* \in {\sfRC(Z)}^{[n-1]}$, and 
\item $\rho_1(q^*) \in \Refine^\nu(\rho_0(q))$\, if  $q^* \not \in {\sfRC(Z)}^{[n-1]}$ but $q^*\in  \Refine^\nu(q)$ for some $q\in Z^{[n-1]}$. 
\end{enumerate}

\begin{remark}
The name nearly nested stems from the following observation.

Aside from the connectors $\sfJ(Z)$ in $\sfRC(Z)$, the system $\sfRC(Z)$ is completely contained in $\rho_0(Z^{[n-1]})$. Since $\rho_1$ satisfies the conditions above, we have that 
\[
\rho_1((Z_1)^{[n-1]})\subset  \Refine^\nu(\rho_0(Z^{[n-1]}))\,\cup \sfRC(Z) \subset  \Refine^\nu(\rho_0(Z^{[n-1]}))\,\cup \sfJ(Z).
\]
\end{remark}

\subsection{Construction of $Z_2 = \Channel_{\cL_1}(Z_1)$ by localized channeling}
\label{sec:localization}
\index{channeling!localized}
We have already had 
\[ Z_0=Z, \, \cL_0=\{Z_0\},\,  Z_1=\Channel(Z),\]
where the construction  $\Channel(Z)$ is based on $\cL_0$, $\rho_0=\rho_Z$ and  $\cT_0=\cT_Z$.

\subsubsection{Outline for constructing $Z_2$}\label{sec:outlines}

To define $Z_2$, we perturb $Z_1$ locally. We define first a partially ordered partition 
\[\cL_1=\{ Z_1(q) \colon q\in Z^{[n-1]}\}\]
 of  $Z_1$ guided by tree $\cT_0$. After choosing an appropriate ambient complex $U(q)$ for each $Z_1(q)$, we show that all pairs $(U(q), Z_1(q))$ are admissible, hence $\cL_1$ is a localization, and that the restriction of the function $\rho_1$, chosen in Section \ref{sec:channeling-preference-function}, to $Z_1(q)^{[n-1]}$ has image in $U(q)^{[n]}$ for each $q$. Thus 
\[\rho_q = \rho_1|_{Z_1(q)} \colon Z_1(q)^{[n-1]} \to U(q)^{[n]}\]
is a well-defined preference function. Let   
 \[\mathcal{P}_1=\{\rho_q\colon q\in Z^{[n-1]}\}.\]
Lastly, we fix a family 
\[\fT_1=\{\cT_q\colon q\in Z^{[n-1]}\}\]
of spanning trees $\cT_q$ on $Z_1(q)$ each of which is partially order and locally Euclidean in $\Refine^{\nu}(K)$. 

For the definition of locally Euclidean trees, we adopt Definition \ref{def:locally-Euclidean}.
\begin{definition}
A spanning tree $\cT\subset \Gamma(R)$ of an adjacently connected $(n-1)$-subcomplex $R$ in a cubical $n$-complex $P$ is \emph{locally Euclidean in $P$} provided that $P$ is locally Euclidean at each edge of $\cT$.
\end{definition}

Having $\cL_1, \mathcal{P}_1,\fT_1$ in hand, we apply Proposition \ref{prop:summary-Channeling} to  
\[(U(q),Z_1(q), \rho_q, \cT_q)\]
to obtain local channelings $\Channel(Z_1(q))$, and then take their union
\[
Z_2 = \Channel_{\cL_1}(Z_1)=\bigcup_{q\in Z^{n-1]}} \Channel(Z_1(q)).
\]

\subsubsection{Localization of $Z_1$}
\label{sec:Z-localization}\index{localization!separating complex}
We begin by subdividing  complex $Z_1$ into two subcomplexes
\[
I(Z_1) = Z_1\cap \sfRC(Z)
\quad \text{and} \quad
II(Z_1) = \Refine^\nu(Z) - \sfRC(Z)^{(n-1)}.
\]
Using the partially ordered partition 
$\cR \cC_{\Refine^\nu(K)}(Z)= \{\sfRC(Z)_q\colon q\in Z^{[n-1]}\}$ 
 of $\sfRC(Z)$ determined by  $\cT_0$ in Section  \ref{sec:sub-reservoir-canal-system},
we subdivide complexes $I(Z_1)$ and $II(Z_1)$ further into subcomplexes 
\[
I_q(Z_1) =  Z_1\cap \sfRC(Z)_q,
\quad \text{and} \quad
II_q(Z_1) = \Refine^\nu(q) - \sfRC(Z)^{(n-1)},
\]
for $q\in Z^{[n-1]}$. Clearly
\[
I(Z_1) = \bigcup_{q\in Z^{[n-1]}} I_q(Z_1) 
\quad \text{and} \quad
II(Z_1) = \bigcup_{q\in Z^{[n-1]}} II_q(Z_1).
\]
It is straightforward to check that 
\[Z_1(q) = I_q(Z_1) \cup II_q(Z_1)\]
 is adjacently connected.

Since any $(n-1)$-cube common to two distinct $\sfRC(Z)_q$ and $\sfRC(Z)_{q'}$ belongs to the passages that have been removed before  the channeling,  such $(n-1)$-cubes are not in $Z_1$. 
Therefore the family $\{ I_q(Z_1)^{[n-1]} \colon q\in Z^{[n-1]}\}$ is an essential partition of $I(Z_1)^{[n-1]}$. Clearly the family $\{ II_q(Z_1)^{[n-1]} \colon q\in Z^{[n-1]}\}$ is also an essential partition of $II(Z_1)^{[n-1]}$.

We now fix an ambient $n$-subcomplex $U(q)\subset \Refine^\nu(\rho(q))$ for $Z_1(q)$. 
For each $q\in Z^{[n-1]}$, let 
\[
N(q) =\Span_{\Refine^\nu(K)}\left(\{Q' \in \Refine^{\nu}(\rho(q))^{[n]} \colon Q' \cap q \neq \emptyset \} \right) - \sfRC(Z) 
\]
and take   \index{localization!ambient space}
\[
U(q)= \sfRC(Z)_q  \cup N(q) \subset \Refine^{2\nu}(K).
\]
 \index{separating complex!localized ambient space}

As a preparatory step towards localized channeling over $Z_1$, we summarize the data needed for the process.

\begin{lemma}
\label{lemma:admissible-step-1} 
For each $q\in Z^{[n-1]}$, $(U(q),Z_1(q))$ is admissible in $(\Refine^\nu(K), Z_1)$, and
\[
\rho_q = \rho_1|_{Z_1(q)} \colon Z_1(q)^{[n-1]} \to U(q)^{[n]}
\]
is an admissible preference function. 
Moreover,  each $\Gamma(Z_1(q))$ admits a  spanning tree which is locally Euclidean in $\Refine^\nu(K)$ and of size at most $\lambda_\loc$.
\end{lemma}

\begin{proof}
Since $\sfRC(Z)_q= \bigcup_{i=1}^m \sfRC(Z)_{q;i}$ and each $\sfRC(Z)_{q;i}$ belongs to a distinct $\Sigma_i$-component $\Comp(Z_1;\Sigma_i)$, the pair $(U(q),Z_1(q))$ is admissible in $(\Refine^\nu(K), Z_1)$. Since $\rho_1$ is nearly nested in $\rho_0$, each preference function $\rho_q$ is admissible.

It follows from Remark \ref{rmk:markers} and Remark \ref{rmk:ambient-essentially-disjoint} that markers in connectors are far from one another and that each marker  is contained in only one $Z_1(q)$. Because markers are the only $(n-2)$-cubes in $Z_1(q)$ at which $\Refine^\nu(K)$ might not be locally Euclidean, there exist spanning trees in $\Gamma(Z_1(q))$ which do not pass markers. Such trees are locally Euclidean.
\end{proof}

\subsubsection{Family of localizations $\Channel(Z_1(q))$}

We begin with the observation that by
Definition \ref{def:localization}, the essential partition $\cL_1$ is  a localization of $Z_1$. In view of Lemma \ref{lemma:admissible-step-1}, we may fix for each $q\in Z^{[n-1]}$, a partially ordered locally Euclidean spanning tree $\cT_q$  of $\Gamma(Z_1(q))$ of size at most $\lambda_\loc$. Let
$\fT_1=\{\cT_q\colon  q\in Z^{[n-1]}\}$. 
 
As outlined in Section \ref{sec:outlines}, we may now apply Proposition \ref{prop:summary-Channeling} to each $(U(q),Z_1(q), \rho_q, \cT_q)$ to obtain $\Channel(Z_1(q))$.

Recall from Definition \ref{def:Channeling}  that channeling 
\[
\Channel(Z_1(q)) = \Tr(Z_1(q)) -  \bigcup_{i=1}^m \omega_{Z_1(q);i}
\]
is obtained by removing gates $ \omega_{Z_1(q);i}$ on the transformation $\Tr(Z_1(q))$ and that  the transformation has the form
\[
\Tr(Z_1(q)) = \left( \Refine^\nu({Z_1(q)}) \cup \sfRC(Z_1(q))^{(n-1)} \right) -  P_{Z_1(q)} \, \subset\, \Refine^{2\nu}(K), 
\]
where $\sfRC(Z_1(q))= \bigcup_{i=1}^m \sfRC(Z_1(q))_i  \subset \Refine^\nu(U(q)) \subset \Refine^{2\nu}(K)$, $P_{Z_1(q)}$ is the family of $(n-1)$-cubes representing the edges in 
\[\tau_{Z_1(q)}= \tau_{Z_1(q),1}\cup \cdots \cup\tau_{Z_1(q),m},\]
and each $\tau_{Z_1(q);i}$ is a  spanning tree in $\Gamma(\sfRC(Z_1(q))_i)$ of size at most $\lambda_\loc$.

Regarding the choice of reservoir-canal systems $\sfRC(Z_1(q))$, we observe first that each tree $\cT_q\in \fT_1$ is locally Euclidean. Thus,
connectors in the reservoir-canal system $\sfRC(Z_1(q))$ over $Z_1(q)$ may be chosen so that 
the connectors are contained in the union of $\sfRC(Z)$ and the preference cubes $\rho_q(Z_1(q))$. We formulate this  observation as follows and leave its verification to the interested reader.

\begin{lemma}
\label{lemma:rho-Channel-regularity}
Let $q\in Z^{[n-1]}$.
The reservoir-canal system  $\sfRC(Z_1(q))$ may be chosen so that 
for each edge $\{q^*, q^{**}\} \in \cT_q$ and each $i  \in \{1,\ldots, m\}$,  the connector $\sfJ_{\{q^*,q^{**}\};i}$ in the canal stretch $\sfC_{\{q^*,q^{**}\};i} \subset \sfRC(Z_1(q))$  satisfies
\[
\sfJ_{\{q^*,q^{**}\};i} \subset \Refine^\nu(\sfRC(Z))\, \cup \, \Refine^\nu(\rho_q(Z_1(q)^{[n-1]})).
\]
\noindent In particular, if $q^*\in \Refine^\nu(\sfRC(Z))$ and  $q^{**}\in \Refine^\nu(Z_1(q))-  \Refine^\nu(\sfRC(Z))$, then 
$\sfJ_{\{q^*,q^{**}\};i}  -\Refine^\nu(\sfRC(Z))$ consists of exactly one $n$-cube, which has a face  in $q$.
\end{lemma}

Before defining complex $Z_2$, we observe that complexes $\Channel(Z_1(q))$ are essentially disjoint.

\begin{corollary}\label{cor:RC-far}
Let $q\neq q'\in Z^{[n-1]}$. Then local systems $\sfRC(Z_1(q))$ and $\sfRC(Z_1(q'))$ are disjoint and 
their graph distance in $\Gamma(\Refine^{2\nu}(K))$ is at least $3^{\nu-2}$. 
In particular, $|\partial Z_1(q)| \subset |\partial \Channel (Z_1(q))|$ for $q\in Z^{[n-1]}$,  and $|\Channel(Z_1(q)) \cap \Channel(Z_1(q'))| =| Z_1(q)\cap Z_1(q')|$ for $q\neq q'$. 
\end{corollary}
\begin{proof}

By Corollary \ref{cor:border-of-Y}, complexes $\sfRC(Z_1(q))$, for $q\in Z^{[n-1]},$ have the claimed graph distance. Hence $\Channel(Z_1(q))$ and $\Channel(Z_1(q'))$, if $q\neq q'$, can meet only at their common boundary.
\end{proof}

\subsubsection{Definition and properties of  $Z_2=\Channel(Z_1)$}

We define
\[
Z_2 = \Channel(Z_1)=\bigcup_{q\in Z^{[n-1]}} \Channel(Z_1(q)).
\]

\begin{remark}
The complex $Z_2$ is determined by the choices associated to the channeling of the elements in the localization $\cL_1$, but the localization $\cL_1$ depends on $(K,Z,\rho_Z,\cT_Z)$ as seen in Section \ref{sec:Z-localization}; this dependence is also present in the union taken over the elements of $Z^{[n-1]}$. Hence, the complex $Z_2$ stems from  $Z_1$ as well as from $Z_0=Z$.
\end{remark}

We show that $Z_2$ is a separating complex in $\Refine^{2\nu}(K)$, which 
has the relative Wada property with respect to $Z_1$,  is a $\lambda_\loc$-perturbation of $Z_1$, and is core-expanding.

We begin with four immediate properties.
\begin{corollary}
\label{cor:Z_2-adj-conn}
Complex $Z_2$ is an adjacently connected.
\end{corollary}
\begin{proof}
Since each $\Channel(Z_1(q))$ is adjacently connected and boundary of each $\Channel(Z_1(q))$ agrees with $\Refine^\nu(q)$, we conclude that complex $Z_2$ is adjacently connected.
\end{proof}

\begin{corollary}\label{cor:localization-Wada-Z}
Complex $\Channel(Z_1)\subset \Refine^{2 \nu}(K)$ has the relative Wada property with respect to $Z_1$.
\end{corollary}
\begin{proof}
The relative Wada property follows from Proposition \ref{prop:summary-Channeling} and the construction of local channelings $\Channel(Z_1(q))$ for $q\in Z_1^{[n-1]}$.
\end{proof}

\begin{corollary}\label{cor:core-Z_1-Z_2}
For each $i\in \{1,\ldots, m\}$,
\[
|\Core(Z_1;\Sigma_i)|\subset |\Core(Z_2;\Sigma_i)|\setminus  |(\partial \Core(Z_2;\Sigma_i))_{\mathrm{inner}}|.
\]
\end{corollary}
\begin{proof}
By Lemma \ref{lemma:rho-Channel-regularity}, $\sfRC(Z_1(q)) \subset \Refine^\nu(\sfRC(Z))\, \cup \, \Refine^\nu(\rho_q({Z_1(q)}))$ and $\Star_{\Refine^{2\nu}(K)}(Z_2)\subset  \Star_{\Refine^{2\nu}(K)}(\sfRC(Z_1) \cup \Refine^\nu(Z_1)) \subset \Refine^\nu(\Star_{\Refine^\nu(K)}(Z_1)$. 
The core-expansion property follows now as in the case of Corollary \ref{cor:core-Z-Z_1}.
\end{proof}

\begin{corollary}\label{cor:cut-graph-2}
The cut-graph  $\Gammacut(\Refine^{2\nu}(K);Z_2)$ has exactly $m$ connected components,
\[
G_{Z_2;i}=\Gamma(\Rec(Z_1;\Sigma_i))\cup \bigcup_{q\in Z^{[n-1]}} (\tau_{Z_1(q);i} \cup  \omega_{Z_1(q);i}),\quad  i=1,\ldots, m.
\]
\end{corollary}
\begin{proof}
For each $i\in\{1,\ldots, m\}$, the receded $\Sigma_i$-component resulted from the removal of the system $\sfRC(Z_1)=\bigcup_{q\in Z^{[n-1]}}\sfRC(Z_1(q))$ from $\Comp(Z_1;\Sigma_i)$ is
\[\Rec(Z_1;\Sigma_i)= \Refine^\nu(\Comp(Z_1;\Sigma_i))  - \sfRC(Z_1).\]

Since pairs $(U(q); Z_1(q))$ are admissible, all local systems $\sfRC(Z_1(q))_i, q\in Z^{[n-1]},$ with  index $i$ are channeled to the receded complex $\Rec(Z_1;\Sigma_i)$. The claim follows.
\end{proof}

To show that $Z_2$ is an perturbation of $Z_1$, we begin with an observation.
By construction, the $\Sigma_i$-component $\Comp(Z_2;\Sigma_i)=\Span_{\Refine^{2\nu}(K)}(G_{Z_2;i})$ has an essential partition 
\[
\Comp(Z_2;\Sigma_i) = \Rec(Z_1;\Sigma_i) \cup \sfRC(Z_1)_i.
\]
By  
Proposition \ref{prop:reservoir-canals-as-indentations},
\[\D(Z_1;\Sigma_i)=\bigcup_{q\in  Z^{[n-1]}} \pi_2^{-1}\left(\sfRC(Z_1)\cap \Refine^\nu(\Comp(Z_1;\Sigma_i)) \right)\]
 is a union of mutually disjoint indentations hence an indentation in lift $\Real(Z_1;\Sigma_i)=\pi_1^{-1}(\Comp(Z_1;\Sigma_i))$.
Thus Corollary \ref{cor:cut-graph-2} gives rise to a description of the structural relation between $\Real(Z_2;\Sigma_i)=\pi_2^{-1}(\Comp(Z_2;\Sigma_i))$ and $\Real(Z_1;\Sigma_i)$.

\begin{lemma} 
\label{lemma:localized-Realization-structure-2}
For each $i=1,\ldots, m$, 
\[
\Real(Z_2;\Sigma_i) = \Realdented(Z_1;\Sigma_i) \cup \Tunnel(Z_2;\Sigma_i),
\]
where 
\begin{enumerate}
\item $\Realdented(Z_1;\Sigma_i)=\Refine^\nu(\Real(Z_1;\Sigma_i))-\D(Z_1;\Sigma_i)$, 
\item  $\Tunnel(Z_2;\Sigma_i)=  \bigcup_{q\in Z^{[n-1]}}\tunnel(q;\Sigma_i)$ is a family of tunnels whose adjacency graphs are  isomorphic to trees $\tau_{Z_1(q);i}$, of sizes at most  $ \lambda_\loc$, and
\item $\tunnel(q;\Sigma_i)\cap  \Realdented(Z_1;\Sigma_i)  =\pi_2^{-1}( \omega_{Z_1(q);i})$ is the lift a gate.
\end{enumerate}
\end{lemma}

Similarly, as for $Z_1$, Indentation-flattening Theorem \ref{theorem:flattening-indentation} and the Tunnel-contracting Proposition \ref{prop:tunnel-contracting} now yield the following bilipschitz relations. Note that, due to the controlled lengths of the tunnels $\tunnel(q;\Sigma_i)$ and the choice of the nearly nested preference function $\rho_1$, the constants in the statement depend only on the dimension $n$ and the original complex $K$.

\begin{corollary}
\label{cor:Z_2-bilipschitz-expanding}
There exist a constant $L=L(n,K)\ge 1$, depending only on the dimension $n$ and complex $K$, and surjective  $L(n,K)$-bilipschitz homeomorphisms 
\[ 
|\Real(Z_1;\Sigma_i)| \xrightarrow{f} |\Realdented(Z_1;\Sigma_i)| \xrightarrow{g} |\Real(Z_2;\Sigma_i)|,
\]
whose composition is the identity on the core of $\Real(Z_1;\Sigma_i)$.
\end{corollary}

We are now ready to summarize the properties of $Z_2$.

\begin{corollary}
\label{cor:localization-separating-complex-2}
Complex $Z_2$ is a separating complex in $\Refine^{2\nu}(K)$, which 
has the relative Wada property with respect to $Z_1$,  is a $\lambda_\loc$-perturbation of $Z_1$, and is
core-expanding.
\end{corollary}

\begin{proof}
We verify first that $Z_2$ is a separating complex. Since 
\[
Z_1\cap \Star_{\Refine^\nu(K)}(\partial \Refine^\nu(K)) = \emptyset,
\]
 we have that $\Channel(Z_1(q)) \cap \Star_{\Refine^{2\nu}(K)}(\partial \Refine^{2\nu}(K))=\emptyset$ for each $q\in Z_1^{[n-1]}$. Thus $Z_2 \cap \Star_{\Refine^{2\nu}(K)}(\partial \Refine^{2\nu}(K))=\emptyset$ and the first condition holds. By Corollary \ref{cor:Z_2-adj-conn}, $\Gamma(Z_2)$ is connected and the second condition holds.

For the third condition, recall that we identify boundary $\partial K$ and its components $\Sigma_i$ with their refinements $\Refine^{k\nu}(\partial K)$ and $\Refine^{k\nu}(\Sigma_i)$, respectively. Due to the construction, the complex $Z_2$ inherits from $Z_1$ the properties $\Sigma_i \subset \Comp_{\Refine^{2\nu}(K)}(Z_2; \Sigma_i) \cap \Refine^{2\nu}(K) = \Sigma_i$ and $\bigcup_{i=1}^m \Comp_{\Refine^{2\nu}(K)}(Z_2;\Sigma_2) = \Refine^{2\nu}(K)$. Thus the third property holds.

Finally, by Corollary \ref{cor:Z_2-bilipschitz-expanding},
$|\Real(Z_2;\Sigma_i)|$ is homeomorphic to  $|\Real(Z_1;\Sigma_i)|$ hence also to $|\Sigma_i|\times [0,1]$. We conclude that $Z_2$ is a separating complex.

By Corollaries \ref{cor:localization-Wada-Z} and \ref{cor:core-Z_1-Z_2}, $Z_2$ has Wada property with respect to $Z_1$ and is core-expanding with respect to $Z_1$. Finally, by  Lemma \ref{lemma:localized-Realization-structure-2} again, $\Channel(Z_1)$ is an $\lambda_\loc$-perturbation of $Z_1$.
\end{proof}

\subsubsection{Cumulative effect of channelings $\Channel(Z)$ and $\Channel_{\cL_1}(\Channel(Z))$}\label{sec:quote}

Before proceed to the construction of $Z_k$ for $k \ge 3$, we examine the cumulative effect of successive channelings on $\Comp(Z;\Sigma)$.
This cumulative effect can be summarized, in the case of $\Real(Z_2;\Sigma_i)$, as follows:

\begin{quote}
Dents in $\Comp(Z_1;\Sigma_i)$ are made by carving out $\sfRC(Z_1(q))$ from $\Comp(Z_1;\Sigma_i)$ near each $q\in Z^{[n-1]}$, including those $q\in Z^{[n-1]}$ on $\partial \Comp(Z;\Sigma_i)$ which have not been altered during the construction of $Z_1$.
This yields 
that tunnels in $\Real(Z_2;\Sigma_i)$ created in the construction of $Z_2$ are
attached not only to the boundary of tunnels created in the construction of $Z_1$ but also to the part of the boundary of $Z$ not affected by the construction of $Z_1$. This leads to a construction that, in the inductive step,
the newly created tunnels at the $k$-th steps will be attached to its ancestors of all previous generations. 
\end{quote}

This seemingly straightforward observation
leads to 
notational and bookkeeping issues.
We deal with these by considering  combined reservoir-canal systems and cumulative indentations.

For this discussion, we begin with an observation on relationship of spectral cubes in complexes $Z_0=Z$ and $Z_1$.

\begin{lemma}
\label{lemma:RC_cup-1}
Each spectral cube of the complex $\sfRC(Z_1)$ is either contained in a spectral cube of  $\Refine^\nu(\sfRC(Z_0))$ or has a face in $Z_0$. 
\end{lemma}
\begin{proof}
Since
\[ \Refine^\nu(Z_0) -  \sfRC(Z_0)\subset Z_1 \subset \sfRC(Z_0) \cup (\Refine^\nu(Z_0) -  \sfRC(Z_0)),\] 
we have
\[\sfRC(Z_1) \subset \Refine^\nu(\sfRC(Z_0))\, \cup \, \Refine^\nu(\rho_1((\Refine^\nu(Z_0) -  \sfRC(Z_0))^{[n-1]}),\]
where $\rho_1\colon Z_1^{[n-1]}\to \Refine^\nu(K)^{[n]}$ is the preference function, nearly nested in $\rho_0$, used for the construction of $Z_2$.
From these inclusion relations and Lemma \ref{lemma:rho-Channel-regularity}, the  claim follows.
\end{proof}

We now combine the reservoir-canal systems in the first two steps into one complex: 
\[
\sfRC_{\cup}(Z_0,Z_1) =   \Refine^{\nu}(\sfRC(Z_0)) \cup \sfRC(Z_1) \subset \Refine^{2\nu}(K).
\]

\begin{lemma}
\label{lemma:Dent_cup-1}
The complex
\[
\Dent_\cup(Z_0, Z_1;\Sigma_i) = \pi_0^{-1}\left(\sfRC_{\cup}(Z_0, Z_1) \bigcap\Refine^\nu(\Comp(Z_0; \Sigma_i))  \right)
\]
is a well-defined subcomplex of $\Refine^{2\nu}(\Real(Z_0;\Sigma_i))$ and is an indentation in the complex $\Real(Z_0;\Sigma_i)$.
\end{lemma}

\begin{proof}
The first claim follows from Remark \ref{rmk:G-quotient-refinement}. 
For the second claim, we observe first that the complex
 $\Dent(Z_0;\Sigma_i)=\pi_0^{-1}(\sfRC(Z_0)\cap \Comp(Z_0;\Sigma_i))$ is an indentation in $\Real(Z_0;\Sigma_i)$ by Proposition \ref{prop:reservoir-canals-as-indentations}.

System $\sfRC(Z_1)=\bigcup_{q\in Z_0^{[n-1]}} \sfRC(Z_1(q))$ is a union of mutually disjoint local systems. Since $\sfRC(Z_1(q))-\sfRC(Z_0)$ is contained in the preference cube $\rho_1(q)\in K$, the local geometry is Euclidean.
By Lemma \ref{lemma:RC_cup-1}, every spectral cube of 
 $\sfRC(Z_1(q))$, if not contained in $\sfRC(Z_0)$, has a face in $Z_0$. 
Indeed, by Lemma \ref{lemma:rho-Channel-regularity}, intersection $\sfRC(Z_1(q))-\sfRC(Z_0)$ and $\sfRC(Z_0)$ consists of mutually disjoint $(n-1)$-cubes. Moreover, by construction, intersection of three distinct spectral cubes in $\sfRC_\cup(Z_0,  Z_1)$ is empty.

Lifts $\pi_0^{-1}(\sfRC(Z_1(q))-\sfRC(Z_0))$ and $\sfRC(Z_1(q))-\sfRC(Z_0)$ are isomorphic. Thus
$\Dent_\cup(Z_0,  Z_1;\Sigma_i)$
is a union of essentially disjoint indentations $\Dent(Z_0;\Sigma_i)$ and
$\pi_0^{-1}((\sfRC(Z_1(q))-\sfRC(Z_0))\cap  \Comp(Z_0;\Sigma_i))$
for $q\in Z_0^{[n-1]}$, whose pairwise intersection satisfies the requirements in Definition \ref{def:indentation} and intersection of any three distinct spectral cubes  is empty. 
We conclude that $\Dent_\cup(Z_0,Z_1;\Sigma_i)$ is an indentation in $\Real(Z_0;\Sigma_i)$. 
\end{proof}

The following proposition describes the iterative relation between lifts $\Real(Z_2;\Sigma_i)$ and $ \Real(Z_0;\Sigma_i)$. It can be seen as a composition of Corollary \ref{cor:channel-Realization-structure} and Lemma \ref{lemma:localized-Realization-structure-2}.
\begin{proposition}
\label{prop:combined-localized-Realization-structure-2}
For each $i\in\{1,\ldots,m\}$,  $\Real(Z_2;\Sigma_i)$ has an essential partition
\[
\Real(Z_2;\Sigma_i) = \Real^{\dagger_2}(Z_0;\Sigma_i)  
\cup \Tunnel^{\dagger_2}(Z_1;\Sigma_i)   
 \cup \Tunnel(Z_2;\Sigma_i),
\]
where
\begin{enumerate}
\item $\Real^{\dagger_2}(Z_0;\Sigma_i)  = \Refine^{2\nu}(\Real(Z_0;\Sigma_i))-\Dent(Z_0, Z_{1};\Sigma_i)$;
\item $\Tunnel^{\dagger_2}(Z_1;\Sigma_i) = \Refine^{\nu}(\Tunnel(Z_1;\Sigma_i)) - \Dent(Z_1;\Sigma_i)$
is a dented tunnel which meets $\Real^{\dagger_2}(Z_0;\Sigma_i) $ at  $\pi_2^{-1}(\omega_{Z;i})$;
\item  $\Tunnel(Z_2;\Sigma_i)=  \bigcup_{q\in Z^{[n-1]}}\tunnel(q;\Sigma_i)$ is a collection of tunnels each of which meets 
$\Real^{\dagger_2}(Z_0;\Sigma_i)  \cup  \Tunnel^{\dagger_2}(Z_1;\Sigma_i)) $
at the $(n-1)$-cube $\pi_2^{-1}(\omega_{Z_{1}(q);i})$.
\end{enumerate}
\end{proposition}

\begin{proof}
From Lemma \ref{lemma:localized-Realization-structure-2} and the construction it follows that
\begin{align*}
\Real(Z_2;\Sigma_i)&=\left(\Refine^\nu(\Real(Z_1;\Sigma_i))-\Dent(Z_1;\Sigma_i)\right) \cup \Tunnel(Z_2;\Sigma_i)\\
&=\left(\Refine^{2\nu}(\Real(Z_0;\Sigma_i))-\Dent_\cup(Z_0,Z_1;\Sigma_i)\right)\\
&\quad \,\,  \cup \Big((\Refine^\nu(\Real(Z_1;\Sigma_i))-\Dent(Z_1;\Sigma_i)) \\
& \quad \,\,\quad  - (\Refine^{2\nu}(\Real(Z_0;\Sigma_i))-\Dent_\cup(Z_0,Z_1;\Sigma_i))\Big)
\,\, \cup  \Tunnel(Z_2;\Sigma_i)\\
&=\left(\Refine^{2\nu}(\Real(Z_0;\Sigma_i))-\Dent_\cup(Z_0,Z_1;\Sigma_i)\right)\\
&\quad \,\, \cup \left(\Refine^\nu(\Tunnel(Z_1;\Sigma_i)) - \Dent(Z_1;\Sigma_i) \right)\cup  \Tunnel(Z_2;\Sigma_i).
\end{align*}
Other claims in the proposition are restatements of those in Lemma \ref{lemma:localized-Realization-structure-2}.
\end{proof}

\subsubsection{A preference function $\rho_2$ on $Z_2$ nearly nested in $\rho_1$}\label{sec:rho-2}
The process of defining $Z_2 = \Channel(Z_1)$ readily yield a well-defined preference function 
\[\rho_2=(\rho_1)_{\Channel} \colon {Z_2}^{[n-1]}\to {\Refine^{2\nu}(K)}^{[n]},\]
nearly nested in $\rho_1\colon {Z_1}^{[n-1]}\to {\Refine^\nu(K)}^{[n]}$, for which 
\begin{enumerate}
\item $\rho_2(q^*)\in {\sfRC(Z_1)}^{[n]}$ if $q^* \in {\sfRC(Z_1)}^{[n-1]}$, and
\item $\rho_2(q^*) \in \Refine^\nu(\rho_1(q))$ if $q^* \subset q\in Z_1^{[n-1]}$ and $q^* \not \in {\sfRC(Z_1)}^{[n-1]}$.
\end{enumerate}

\subsubsection{A localization $\cL_2$ of $Z_2$ in accordance with $\cL_1$ and $\fT_1$}\label{sec:localization-2}Let $\cL_1=\{ Z_1(q') \colon q'\in Z_0^{[n-1]}\}$ be the localization  of $Z_1$, and $\fT_1=\{\cT_{q'}\colon q'\in  Z_0^{[n-1]} \}$ the family of spanning trees on $Z_1(q')$ 
previously defined.

For each $q'\in Z_0^{[n-1]}$, the local channeling $\Channel(Z_1(q'))$ has a partially ordered partition $ \{Z_2(q) \colon q\in Z_1(q')^{[n-1]}\}$ determined by the tree $\cT_{q'}\in \fT_1$ by the same rules of subdivision as in Section \ref{sec:RC-partition-UY}. Thus, 
\[\Channel(Z_1(q'))=\bigcup_{q\in Z_1(q')^{[n-1]}}  Z_2(q),\]
and the union   
\[\cL_2 =   \{ Z_2(q) \colon q\in Z_1^{[n-1]}\} = \bigcup_{q'\in Z_0^{[n-1]}} \left(\bigcup_{q\in Z_1(q')^{[n-1]}} \{ Z_2(q)\}\right) \]
is an essential partition of $Z_2$.

Similarly as  in Section \ref{sec:Z-localization}, we may choose an ambient subcomplex $U_2(q)\subset  \Refine^{2\nu}(K)$ for each $Z_2(q)\in \cL_2$ for which all pairs $(U_2(q), Z_2(q))$ are admissible in  $(\Refine^{2\nu}(K), Z_2)$ and the preference functions $\rho_2|_{Z_2(q)}$ are also admissible. Hence  $\cL_2$ is a localization and it is uniquely determined by  
  $\cL_1$ and $\fT_1$.

After fixing a spanning tree on each element in $\cL_2$, we may proceed to construct $Z_3=\Channel(Z_2)$. We iterate this construction in Section \ref{sec:evolution}.

\subsection{Evolution sequence}
\label{sec:evolution}
We now construct the sequence $(Z_k)$ by induction. Analogous to  the construction of $Z_2$, which depends on both $Z_1$ and $Z_0$, the construction of $Z_k$ depends on $Z_{k-1}$ as well as $Z_{k-2}$. 

For the record, separating complexes $Z_0 = Z \subset K$, $Z_1 = \Channel(Z) = \Channel_{\cL_0}(Z)\subset \Refine^\nu(K)$, and $Z_2 = \Channel_{\cL_1}(Z_1)\subset \Refine^{2\nu}(K)$ have been constructed, where $\cL_0 = \{Z_0\}$ and $\cL_1= \{Z_1(q)\colon q \in Z_0^{[n-1]}\}$ is a  localization of $Z_1$.

Suppose that, for some $k\ge 2$, we already have  a sequence 
\[
Z_0,\cL_0, Z_1, \cL_1,\ldots,\cL_{k-2},  Z_{k-1}
\]
of separating complexes $Z_j$'s and localizations $\cL_j$'s, for which 
\[Z_j= \Channel_{\cL_{j-1}}(Z_{j-1})=\bigcup_{q\in Z_{j-2}^{[n-1]}} \Channel(Z_{j-1}(q)) \subset \Refine^{\nu j}(K)\] 
is a localized channeling over
$\cL_{j-1}= \{Z_{j-1}(q)\colon q \in Z_{j-2}^{[n-1]}\}$, based on a family $\cP_{j-1}$ of preference functions and a family $\fT_{j-1}$ of  spanning trees on the elements on $\cL_{j-1}$, for $1\leq j \leq k-1$.

\subsubsection{Construction of $Z_k$}

To construct $Z_k$, we follow the steps for $Z_2$ in Section \ref{sec:localization} almost verbatim, and we only give a sketch.

Define as in Section \ref{sec:localization-2} a partition 
\[\cL_{k-1}= \bigcup_{ q' \in Z_{k-3}^{[n-1]}}  \{Z_{k-1}(q)\colon q \in Z_{k-2}(q')^{[n-1]}\}\]
of $Z_{k-1}=\bigcup_{q'\in Z_{k-3}^{[n-1]}} \Channel(Z_{k-2}(q'))$ in accordance with
$\cL_{k-2}$ and $\fT_{k-2}$.
Thus for $q' \in Z_{k-3}^{[n-1]}$, 
\[
\Channel_{\cL_{k-2}}(Z_{k-2}(q'))=\bigcup_{q \in Z_{k-2}(q')^{[n-1]}} Z_{k-1}(q).
\]

Fix as we may, as in Section \ref{sec:Z-localization},  for each $q\in Z_{k-2}^{[n-1]}$  an ambient complex $U_{k-1}(q)$ of $Z_{k-1}(q)$ in $\Refine^{(k-1)\nu}(K)$ for which $(U_{k-1}(q), Z_{k-1}(q))$ is admissible in $(\Refine^{(k-1)\nu}(K), Z_{k-1})$. Hence the partition $\cL_{k-1}$ is a localization.

Fix next a preference function $\rho_{k-1}  \colon Z_{k-1}^{[n-1]} \to (\Refine^{(k-1)\nu}(K))^{[n]}$ which is nearly nested in $\rho_{k-2}$ as in  Section \ref{sec:rho-2}, and let 
\[ \cP_{k-1}=\{\rho_q= \rho_{k-1}|_{Z_{k-1}(q)}\colon q \in Z_{k-2}^{[n-1]}\}.\]
Each $\rho_q$ in $\cP_{k-1}$ 
is admissible by  Lemma \ref{lemma:admissible-step-1}.

Finally fix, for each $q\in Z_{k-2}^{[n-1]}$,  a partially ordered  locally Euclidean spanning tree on $Z_{k-1}(q)$, and let 
\[\fT_{k-1}=\{\cT_q\colon q\in Z_{k-2}^{[n-1]}\}.\]

Having $(\cL_{k-1}, \cP_{k-1}, \fT_{k-1})$ at our disposal, we follow the steps in defining $Z_2=\Channel(Z_1)$ to obtain 
\[
Z_k = \Channel_{\cL_{k-1}}(Z_{k-1})=\bigcup_{q\in Z_{k-2}^{[n-1]}} \Channel(Z_{k-1}(q)).
\]
as the union of local channelings $\Channel(Z_{k-1}(q))$. 

\begin{remark}
Note that, similarly as we chose complexes $Z_1(q)$ with respect to $Z$, we may assume that, for each $q\in Z_{k-2}^{[n-1]}$, the reservoir-canal system $\sfRC(Z_{k-1}(q))$ satisfies the properties of Lemma \ref{lemma:rho-Channel-regularity} with respect to reservoir-canal system of $Z_{k-2}$ and preference function $\rho_q$.
\end{remark}

\subsubsection{Properties of $Z_k$}

Before summarizing the properties of the complex $Z_k$, we discuss the properties of the lifts $\Real(Z_k;\Sigma_i)$. As for $Z_1$ and $Z_2$, this discussion is used to show that $Z_k$ is a separating complex having $\lambda$-perturbation property.

Analogous to the case $k=2$, we have that 
\[
\Comp(Z_k;\Sigma_i) = \Rec(Z_{k-1};\Sigma_i) \cup \sfRC(Z_{k-1})_i
\]
and that
\[
\D(Z_{k-1};\Sigma_i)= \pi_k^{-1}(\sfRC(Z_{k-1})\cap \Refine^\nu(\Comp(Z_{k-1};\Sigma_i)))
\]
is  an indentation 
in $\Real(Z_{k-1};\Sigma_i)$.

The following proposition formalizes the relation between 
$\Real(Z_k;\Sigma_i)$ and $\Real(Z_{k-1};\Sigma_i)$; cf.~Lemma \ref{lemma:localized-Realization-structure-2}.

\begin{proposition} 
\label{prop:localized-Realization-structure-k}
Let $k\geq 2$. For each $i=1,\ldots, m$, 
\[
\Real(Z_k;\Sigma_i) = \Realdented(Z_{k-1};\Sigma_i) \cup \Tunnel(Z_k;\Sigma_i),
\]
where 
\begin{enumerate}
\item $\Realdented(Z_{k-1};\Sigma_i)= \Refine^\nu(\Real(Z_{k-1};\Sigma_i))-\D(Z_{k-1};\Sigma_i)$, 
\item  $\Tunnel(Z_k;\Sigma_i)= \bigcup_{q\in Z_{k-2}^{[n-1]}}\tunnel(q;\Sigma_i)$ is a family of tunnels each of which has graph size at most  $ \lambda_\loc$, and
\item $\tunnel(q;\Sigma_i) \cap  \Realdented(Z_{k-1};\Sigma_i)  = \pi_k^{-1}(\omega_{Z_{k-1}(q), i})$ is the lift of a gate. 
\end{enumerate}
\end{proposition}

Analogously to the argument in the proof of Corollary \ref{cor:localization-separating-complex-2}, we have that following properties.

\begin{corollary}
\label{cor:summary_Z-k}
The  sequence $(Z_k)$, $Z_k\subset \Refine^{k\nu}(K)$, consists of separating complexes for which 
\begin{enumerate}
\item $Z_k$ has the relative Wada property with respect to $Z_{k-1}$, 
\item $Z_k$ is core-expanding with respect to $Z_{k-1}$, and 
\item $Z_k$ is a $\lambda$-perturbation of $Z_{k-1}$, 
\end{enumerate}
where constants $\lambda=\lambda_\loc$ and $L=L_\loc$ when $k\geq 2$, and 
$\lambda$ and $L$ depend on $n$ and $K$ when $k=1$.
\end{corollary}

This completes the construction of the sequence $(Z_k)$.

\subsection{Geometry between cores}\label{sec:core-geometry}
For the geometry of Lakes of Wada, we prove that points between cores may be connected a chain of cubes not too close to the boundary. 

For the statement, recall that sequence $(Z_k)$ is core-expanding.

\begin{lemma}\label{lemma:core-to-core}
Let $k\geq 1$. Given  an $n$-cube $Q_0\in \Core(Z;\Sigma_i)$ and an  $n$-cube $Q_k$ in 
\[
P_k= \Core(Z_k;\Sigma_i) - \Refine^\nu(\Core(Z_{k-1};\Sigma_i)),
\] 
there is a chain $\mathcal C_k$ of $n$-cubes
\[
Q_k=Q_{k,1}, \ldots, Q_{k,p_k}; \,\,Q_{k-1,1},\ldots, Q_{k-1,p_{k-1}};\,\,  \ldots; \,\, Q_{0,1},\ldots, Q_{0,p_0}=Q_0,
\]
connecting $Q_k$ to $Q_0$ for which 
\begin{enumerate}
\item for each $\ell\ge 1$, $Q_{\ell,1}, \ldots, Q_{\ell,p_\ell}$ is a linearly adjacently connected sequence of cubes in 
\[P_\ell=\Core(Z_\ell;\Sigma_i) -\Refine^\nu(\Core(Z_{\ell-1};\Sigma_i)),\] 
and $Q_{\ell,p_\ell}$ is adjacent to the $n$-cube  $Q_{\ell-1,1} \in \Refine^\nu(Q_{\ell-1,1})$;
\item $p_\ell \le 3^\nu\theta(K)$ for $1\leq \ell \leq k$, and $p_0\leq \#(K^{[n]})$. 
\end{enumerate}
\end{lemma}

\begin{proof}
Given $Q_k\in P_k$, we follow the proof of  Lemma \ref{lemma:Z_1-core-expanding-John} to obtain a sequence $Q_k=Q_{k,1}, \ldots, Q_{k,p_k}$ of linearly adjacently connected $n$-cubes in $P_k$ for which  
\begin{enumerate}
\item the sequence is contained in $\Refine^\nu(\Star_{\Refine^{(k-1)\nu}(K)}(Z_{k-1}))$,
\item $Q_{k,p_k}$ has a face in $\Refine^{\nu} \left(\Core(Z_{k-1};\Sigma_i) \cap \Star_{\Refine^{(k-1)\nu}(K)}(Z_{k-1}) \right)$, and
\item  $p_k\leq 3^\nu \theta(K)$.
\end{enumerate}
Let $Q_{k-1,1}$ be the $n$-cube in $\Core(Z_{k-1};\Sigma_i)$ which meets $Q_{k,p_k}$ in a face of $Q_{k,p_k}$.
Since $Q_{k-1,1}$ is contained in $\Star_{\Refine^{(k-1)\nu}(K)}\left(\Star_{\Refine^{(k-1)\nu}(K)}(Z_{k-1})\right) $ and
\[
\Star_{\Refine^{(k-1)\nu}(K)}\left(\Star_{\Refine^{(k-1)\nu}(K)}(Z_{k-1})\right) \subset \Refine^\nu(\Star_{\Refine^{(k-2)\nu}(K)}(Z_{k-2})),
\]
the $n$-cube $Q_{k-1,1}$ is not in $\Refine^\nu(\Core(Z_{k-2};\Sigma_i))$. Thus $Q_{k-1,1}$ is in $P_{k-1}$.

We repeat the construction above for  $\ell =k-1, \ldots, 1,$ in descending order 
to obtain sequences $Q_{\ell,1},\ldots, Q_{\ell,p_{k-1}}$ in $P_\ell$ and  $n$-cubes $Q_{\ell-1,1}$ in $P_{\ell-1}$ for which 
$p_\ell \leq 3^\nu \theta(K)$,
and finally  the
 last sequence $Q_{0,1},\ldots, Q_{0,p_0}=Q_0$ is in $\Core(Z_k;\Sigma_i)$ for a number $p_0\leq \#(K^{[n]})$.
\end{proof}

Lemma \ref{lemma:core-to-core} yields a John property for cores $|\Core(Z_k;\Sigma_i)|$. We do not state this formally here, as we discuss John property of the Lakes of Wada in Section \ref{sec:Wada}. However we record a remark.

\begin{remark}\label{rmk:John}
With respect to the polyhedral
metric $d_K$ on $K$, cubes  $Q_{\ell,j}$, for $j=1,\ldots, p_\ell$, in  Lemma \ref{lemma:core-to-core} have side lengths $3^{- \ell \nu}$. Let $x_{\ell,j}$ be the center of the cube $Q_{\ell,j}$.
Then there exist constants $c_1=c_1(n,\nu)>0$ and $ c_2=c_2(n,\nu)>0$ for $\ell \geq 1$, and constants $c_1=1$ and $c_2=\#(K^{[n]})$ for $\ell=0$, for which 
\[c_1 3^{- \ell \nu}\leq \dist(x_{\ell,j}, |\partial \Core(Z_k;\Sigma_i)|)\leq c_2  3^{- \ell \nu}.\]
Thus we may fix a  PL curve $\sigma$ in $|\Core(Z_k;\Sigma_i)|$ which connects the centers 
\[
x_k= x_{k,1}, \ldots, x_{k,p_k}; x_{k-1,1},\ldots, x_{k-1,p_{k-1}};\,\,  \ldots; \,\, x_{0,1},\ldots, x_{0,p_0}= x_0,
\]
of the cubes in succession by line segments. 
\end{remark}

\subsection{Quasiconformal stability}\label{sec:quasiconformality}

We now prove the quasiconformal stability of the sequence $(Z_k)$. The quasiconformal stability, together with Corollary \ref{cor:summary_Z-k}, completes the proof of the Evolution Theorem (Theorem \ref{theorem:evolution-short}).

The proof of  stability is based on inductive relation between lifts in Proposition \ref{prop:localized-Realization-structure-k}. To control the dilatation, we flatten indentations made in all steps simultaneously, instead of iteratively. To this end, we extend the notions of  \emph{combined reservoir-canal systems} and the \emph{cumulative  indentations} from $k=2$ defined earlier to all $k$'s.

Let $k\ge 0$. We define for each $\ell\in \{1,\ldots,m\}$ a combined reservoir-canal system 
 \[\sfRC_{\cup}(Z_\ell,\ldots, Z_{k-1}) =  \bigcup_{j=\ell}^{k-1} \Refine^{\nu(k-j-1)}(\sfRC(Z_{j})) \subset \Refine^{k\nu}(K). \]
Note that  the combined system does not meet  core $|\Core_{K_\ell}(Z_\ell;\Sigma_i)|$.

The spectral cubes of $\sfRC(Z_k)$ and the spectral cubes of  $\sfRC_{\cup}(Z_\ell, \ldots, Z_{k-1})$ have the following relation.

\begin{lemma}
\label{lemma:RC_cup}
Let $k\ge 1$. A spectral cube of the complex $\sfRC(Z_{k})$ is either contained in a spectral cube of  $\sfRC_{\cup}(Z_0,\ldots, Z_{k-1})$ or has a face in $Z_0$. 
\end{lemma}
\begin{proof}For $k=1$, the claim follows from Lemma \ref{lemma:RC_cup-1}.

Analogous to the case $k=1$, we obtain from the construction that for $k\geq 2$, 
 \[Z_k \subset \sfRC_{\cup}(Z_\ell, \ldots, Z_{k-1}) \cup \left(\Refine^{k \nu}(Z_0) -  \sfRC_{\cup}(Z_\ell, \ldots, Z_{k-1}))\right)\]
 and 
 \[\Refine^{k\nu}(Z_0) -  \sfRC_{\cup}(Z_\ell, \ldots, Z_{k-1}) \subset Z_k.\] 
Thus, 
\begin{align*}
\sfRC(Z_k) \subset &\Refine^\nu(\sfRC_{\cup}(Z_\ell, \ldots, Z_{k-1}))\\ 
& \cup \left( \Refine^\nu(\rho_{k-1}(\Refine^{(k-1) \nu}(Z_0) -  \sfRC_{\cup}(Z_\ell, \ldots, Z_{k-1})\right),
\end{align*}
where $\rho_{k-1}\colon (Z_{k-1})^{[n-1]}\to \Refine^{(k-1)\nu}(K)^{[n]}$ is the preference function used in the construction of $Z_k$.
From this inclusion relation, the  claim follows.
\end{proof}

Let $k\geq 2$. We define for each $0\leq \ell \leq k-1$,  a subcomplex
\begin{align*}
&\Dent_\cup(Z_\ell, \ldots, Z_{k-1};\Sigma_i) \\
&\quad = \pi_\ell^{-1}\left(\sfRC_{\cup}(Z_\ell,\ldots, Z_{k-1}) \bigcap\Refine^{(k-\ell) \nu}(\Comp(Z_\ell; \Sigma_i))  \right)\\
&\quad \subset \Refine^{(k-\ell) \nu}(\Real(Z_\ell;\Sigma_i)).
\end{align*} \index{$\Dent_\cup( Z_\ell,\ldots, Z_k)$}
As in the case $k=2$, by Remark \ref{rmk:G-quotient-refinement},  the quotient $\pi_\ell$ may be passed to a map between  the refinements of the domain and the refinement of the target, thus $\Dent_\cup(Z_\ell, \ldots, Z_{k-1};\Sigma_i)$ is  well-defined.

Let now for $\ell \ge 1,$
\begin{align*}
&\Dent_\cup^\sfT (Z_\ell, \ldots, Z_{k-1};\Sigma_i) 
=  \Dent_\cup(Z_\ell, \ldots, Z_{k-1};\Sigma_i)\bigcap \Refine^{(k-\ell)\nu}(\Tunnel(Z_\ell;\Sigma_i)),
\end{align*}
where $\Tunnel(Z_\ell;\Sigma_i)$ is the tunnel part of $\Real(Z_\ell;\Sigma_i)$;
see Proposition \ref{prop:localized-Realization-structure-k} for notation.

\begin{lemma}
\label{lemma:Dent_cup}
Let $k \ge 1$ and $i\in \{1,\ldots, m\}$. Then $\Dent_\cup(Z_0,\ldots, Z_{k-1};\Sigma_i)$ is an indentation in $\Real(Z_0;\Sigma_i)$. Furthermore,  for each $\ell \in\{1,\ldots, k-1\}$,
$\Dent_\cup^\sfT(Z_\ell,\ldots, Z_{k-1};\Sigma_i)$ is an indentation in $\Tunnel_{K_\ell}(Z_\ell;\Sigma_i)$.
\end{lemma}
\begin{proof}
When $k=1$, $\Dent_\cup(Z_0;\Sigma_i)=\pi_0^{-1}(\sfRC(Z_0)\cap \Comp(Z_0;\Sigma_i))$ is an indentation in $\Real(Z_0;\Sigma_i)$ by Proposition \ref{prop:reservoir-canals-as-indentations}. When $k=2$, $\Dent_\cup(Z_0, Z_1;\Sigma_i)$ is an indentation in $\Real(Z_0;\Sigma_i)$ by Lemma \ref{lemma:Dent_cup-1}. 
The claim for $k\geq3$ is by induction. To reach the claim for $k$ assuming the validity of the statement for $k-1$, we apply Lemma \ref{lemma:RC_cup} and follow the same line of argument as that in Lemma \ref{lemma:Dent_cup-1}.  This completes the proof of the first claim. 

Let 
$\ell \in\{1,\ldots, k-1\}$. For the proof of the second claim, we follow the proof of the first statement with $(\Tunnel_{K_\ell}(Z_\ell;\Sigma_i);\Dent_\cup^\sfT(Z_\ell,\ldots, Z_{k-1};\Sigma_i))$ assuming the role of $(\Real(Z_0;\Sigma_i);\Dent_\cup(Z_0,\ldots, Z_{k-1};\Sigma_i))$. Other than the fact that $\Tunnel_{K_\ell}(Z_\ell;\Sigma_i)$ has multiple components, the proof follows almost verbatim. We omit the details.
\end{proof}

We record the iterative relation between the structures of $\Real(Z_k;\Sigma_i)$ and $\Real(Z_0;\Sigma_i)$ in a proposition, which reveals the non-standard tree-like structure of  $\Real(Z_k;\Sigma_i)$ for $k\geq 2$ described in 
beginning of
Section \ref{sec:quote}. Proof 
of the following 
proposition is a straightforward extension of that of Proposition \ref{prop:combined-localized-Realization-structure-2}, we omit the details; see 
Proposition \ref{prop:localized-Realization-structure-k} and Lemma \ref{lemma:Dent_cup} for notations.

\begin{proposition}
\label{prop:combined-localized-Realization-structure}
Let $k\ge 2$. For each $i\in\{1,\ldots,m\}$, $ \Real(Z_k;\Sigma_i)$ has an essential partition
\[
\Real(Z_k;\Sigma_i) = \Real^{\dagger_k}(Z_0;\Sigma_i)  
\cup \bigcup_{\ell=1}^{k-1}\Tunnel^{\dagger_k}(Z_\ell;\Sigma_i)   
 \cup \Tunnel(Z_k;\Sigma_i),
\]
where 
\begin{enumerate}
\item $\Real^{\dagger_k}(Z_0;\Sigma_i)  = \Refine^{k\nu}(\Real(Z_0;\Sigma_i))-\Dent_\cup(Z_0,\ldots, Z_{k-1};\Sigma_i)$;
\item
$\Tunnel^{\dagger_k}(Z_\ell;\Sigma_i) = \Refine^{(k-\ell)\nu}(\Tunnel(Z_\ell;\Sigma_i)) - \Dent_\cup^\sfT(Z_\ell,... Z_{k-1};\Sigma_i)$,
where 
\begin{enumerate}
\item  $\Tunnel^{\dagger_k}(Z_1;\Sigma_i)$ is a dented tunnel  which meets $\Real^{\dagger_k}(Z_0;\Sigma_i)$ at a refinement of  $\pi_k^{-1}(\omega_{Z;i})$, and 
\item for $\ell\ge 2$, 
$\Tunnel^{\dagger_k}(Z_\ell;\Sigma_i)=\bigcup_{q\in Z_{\ell-2}^{[n-1]}}  \tunnel^{\dagger_k}(q;\Sigma_i)$ is a collection of dented tunnels which  meet
\[\Real^{\dagger_k}(Z_0;\Sigma_i)  
\cup \bigcup_{j=1}^{\ell-1}(\Tunnel^{\dagger_k}(Z_j;\Sigma_i)) \]
in a family $\pi_k^{-1}(\omega_{Z_{\ell-1}(q);i})$, $q\in Z_{\ell-2}^{[n-1]}$, of lifts of gates.
\end{enumerate}
\item $\Tunnel(Z_k;\Sigma_i) =\bigcup_{q\in Z_{k-2}^{[n-1]}} \tunnel(q;\Sigma_i)$ is a family of tunnels each of which meets
\[\Real^{\dagger_k}(Z_0;\Sigma_i)  
\cup \bigcup_{j=1}^{k-1} \Tunnel^{\dagger_k}(Z_j;\Sigma_i)) \]
in the lift $\pi_k^{-1}(\omega_{Z_{k-1}(q);i})$ of a gate. 
\end{enumerate}
\end{proposition}

\begin{proof}
Since $\Real(Z_1;\Sigma_i)=(\Refine^\nu(\Real(Z_0;\Sigma_i)) - \Dent(Z_0;\Sigma_i))\cup \Tunnel(Z_1;\Sigma_i)$, the claim holds for $k=1$. The claim for  $k=2$ has been proved in Proposition \ref{prop:combined-localized-Realization-structure-2}.
For $k\ge 3$, 
the essential partition of $\Real(Z_k;\Sigma_i)$ is proved by induction. Assuming the the validity of the statement for $k-1$, the proof of the claim for $k$  follows the proof of Proposition \ref{prop:combined-localized-Realization-structure-2} almost verbatim. We omit the details.
\end{proof}

Having Proposition \ref{prop:combined-localized-Realization-structure} at our disposal, we may flatten indentations made in different generations simultaneously. Fix an integer $k\geq 2$. 
For the statement, we denote
\[W_{0;i}= \Wedge_{|\Real(Z_k;\Sigma_i))|}(\Dent_\cup(Z_0,\ldots, Z_{k-1};\Sigma_i)),\]
and for $\ell\in\{1,\ldots,k-1\}$,
\[W_{\ell;i}=\Wedge_{|\Real(Z_k;\Sigma_i))|}(\Dent_\cup^\sfT(Z_\ell,\ldots, Z_{k-1};\Sigma_i)).\]

\begin{proposition}
\label{prop:general-tree-structure-of-Realization} Fix an integer $k\geq 2$. For each $i\in\{1,\ldots,m\}$, there exists a space 
\[
X_{k;i}=R_{0;i} \cup \left(\bigcup_{\ell=1}^{k-1} T_{\ell;i}\right)\, \cup \, T_{k;i},
\]
where $R_{0;i}$ is an isometric copy of $|\Real(Z_0;\Sigma_i)|$, and $T_{1;i}, \ldots, T_{k;i}$ are isometric copies of  $|\Tunnel(Z_1;\Sigma_i)|,\ldots$, $|\Tunnel(Z_k;\Sigma_i)|$, respectively, and there exists
an $L(n,K)$-bilipschitz homeomorphism 
\[
\varphi_{k;i} \colon |\Real(Z_k;\Sigma_i)|\to |X_{k;i}|,
\]
which restricts to maps $|\Real^{\dagger_k}(Z_0;\Sigma_i)| \to R_{0;i}$, $|\Tunnel^{\dagger_k}(Z_j;\Sigma_i)| \to T_{j;i}$ for $j=1,\ldots, k-1$, and  a map $|\Tunnel(Z_k;\Sigma_i)| \to T_{k;i}$.
Moreover, 
\begin{enumerate}
\item $\varphi_{k;i}$ is an isometry in the complement of  the wedges 
$W_{0;i}\cup \cdots, \cup W_{k-1;i}$,
and an isometry on $|\pi_k^{-1}(\omega_{Z;i})| \cup \left( \bigcup_{\ell=1}^{k-1} \left( \bigcup_{q\in Z_{\ell-1}^{[n-1]}} |\pi_k^{-1}(\omega_{Z_\ell(q);i})| \right) \right)$; and
\item $T_{1;i} \cap R_{0;i}= \varphi_{k;i}(|\pi_k^{-1}(\omega_{Z;i})|)$, and for  $\ell\in\{2,\ldots, k\}$,
\[ T_{\ell;i} \cap \left(R_{0;i} \cup (\bigcup_{j=1}^{\ell-1}T_{j;i})\right)= \bigcup_{q\in Z_{\ell-2}^{[n-1]}} \varphi_{k;i}(|\pi_k^{-1}(\omega_{Z_{\ell-1}(q);i})|).\] 
\end{enumerate}
\end{proposition}

\begin{proof}
For the proof, we flatten successively the subcomplexes 
$\Real^{\dagger_k}(Z_0;\Sigma_i)$, $\Tunnel^{\dagger_k}(Z_1;\Sigma_i), \ldots,   
\Tunnel^{\dagger_k}(Z_{k-1};\Sigma_i)$
of $\Real(Z_k;\Sigma_i)$ in Proposition \ref{prop:combined-localized-Realization-structure}. We give the details only for the first step. 

Since $\Real^{\dagger_k}(Z_0;\Sigma_i) = \Refine^{k\nu}(\Real(Z_0;\Sigma_i))-\Dent_\cup(Z_0,\ldots, Z_{k-1};\Sigma_i)$, where $\Dent_\cup(Z_0,\ldots, Z_{k-1};\Sigma_i) \cap \Refine^{k\nu}(\Real(Z_0;\Sigma_i))$ is an indentation, there exists, by the Indentation-flattening theorem (Theorem \ref{theorem:flattening-indentation}) a bilipschitz map 
\[
\psi_{0;i} \colon |\Real^{\dagger_k}(Z_0;\Sigma_i)| \to |\Refine^{k\nu}(\Real(Z_0;\Sigma_i))|
\]
which is the identity in the complement of the wedge $W_{0;i}$ of the complex $\Dent_\cup(Z_0,\ldots, Z_{k-1};\Sigma_i)$. Moreover, by the construction in the proof of Theorem \ref{theorem:flattening-indentation}, we may further assume that $\psi_{0;i}$ is an isometry on all gates  in $\Real^{\dagger_k}(Z_0;\Sigma_i)$ where the tunnels are being attached. 

We replace now  subcomplex $\Real^{\dagger_k}(Z_0;\Sigma_i)$ in $\Real(Z;\Sigma_i)$ by the flattened complex $\Refine^{k\nu}(\Real(Z_0;\Sigma_i))$ and keep subcomplexes $\Tunnel^{\dagger_k}(Z_j;\Sigma_i)$ for $j=1,\ldots, k-1$, and tunnels $\Tunnel(Z_k;\Sigma_i)$ attached, canonically,  at the images of the corresponding gates.
 
Repeating this process iteratively for subcomplexes 
$\Tunnel^{\dagger_k}(Z_j;\Sigma_i)$, for $j=1,\ldots, k-1$, 
 the claim follows. Since the application of Theorem \ref{theorem:flattening-indentation} is done in essentially disjoint subcomplexes, the bilipschitz constant of $\phi_{k;i}$ is the bilipschitz constant of Theorem \ref{theorem:flattening-indentation}.
\end{proof}

After the indentations have been flattened, for the proof of quasiconformal stability, it remains to contract tunnels.

\begin{theorem}\label{theorem:quasiconformal-stable}
There exists a constant $\sK=\sK(n, K)\ge 1$ for the following. For each  $k\ge 1$ and $i\in\{1,\ldots, m\}$, there exists a $\sK$-quasiconformal homeomorphism 
\[
f_{k;i} \colon |\Real(Z_k;\Sigma_i)|\to |\Real(Z;\Sigma_i)|,
\]
which is the identity on the core $\Core(Z;\Sigma_i)$.
\end{theorem}

The following argument is a simplification of the argument in \cite{Drasin-Pankka}. For the argument, we fix first a parameter $\beta\in \N$ as follows.

The localized channeling construction yields a constant $N=N(n, \nu)\ge 1$ for which there are at most $N$ tunnels in $\Tunnel(Z_{\ell+1};\Sigma_i)$ that can be attached to the same tunnel $\tunnel(q;i)$  
in the previous collection $\Tunnel(Z_\ell;\Sigma_i)$. 
Let $\beta \in \N$ be the smallest integer for which $\Refine^{\beta \nu}([0,1]^{n-1})$ has at least $N$ mutually disjoint $(n-1)$-cubes contained in  $(0,1)^{n-1}$. 

For the argument, recall also that the tunnel-contracting map in Proposition \ref{prop:tunnel-contracting} has a bilipschitz constant depending only on $n$ and the size of the tunnel. 
By
Propositions \ref{prop:combined-localized-Realization-structure} and \ref{prop:general-tree-structure-of-Realization}, 
$\Tunnel(Z_1;\Sigma_i)$ consists of exactly one tunnel, which has graph size at most $\#(Z^{[n-1]}) 3^{(n-1)\nu}$ and that for $\ell\geq 2$, each $\tunnel(q;\Sigma_i)$ in the collection  $\Tunnel(Z_\ell;\Sigma_i)$ has  size at most $\lambda_\loc$. 
Hence the tunnel-contracting maps have a bilipschitz constant depending only $n$ and $K$.

\begin{proof}[Proof of Theorem \ref{theorem:quasiconformal-stable}]
As a preliminary step, let
\[
\varphi=\varphi_{k;i} \colon |\Real(Z_k;\Sigma_i)|\to R_{0;i} \cup \bigcup_{\ell=1}^k T_{\ell;i} 
\]
be the $L(n,K)$-bilipschitz indentation-flattening map  in Proposition \ref{prop:general-tree-structure-of-Realization}.

For each $\ell\in \{2,\ldots,k\}$, 
we take $\Omega_{q;i}$, $q\in Z_{\ell-2}^{[n-1]}$, to be the  $n$-cube in $R_{0;i} \cup \bigcup_{j=1}^{\ell-1} T_{j;i}$
which has $\mathcal{w}_{q;i}=\varphi(|\pi_k^{-1}(\omega_{Z_{\ell-1}(q);i})|)$ as a face and call it a cube over  $\mathcal{w}_{q;i}$.
Take also $\widehat \Omega_{q;i}$ to be the cube in $R_{0;i} \cup \bigcup_{j=1}^{\ell-1} T_{j;i}$ containing $\Omega_{q;i}$  for which the pair $(\widehat \Omega_{q;i}, \,\Omega_{q;i})$ is a scaling of  $([0,3]^n, [1,2]^{[n-1]}\times [0,1])$.
We call $\widehat \Omega_{q;i}$ a tent over $\mathcal{w}_{q;i}$. When $\ell=1$, we make the obvious changes on notations, and define cube $\Omega_{q;i}$ and tent $\widehat \Omega_{q;i}$ over $\varphi(|\pi_k^{-1}(\omega_{Z;i})|)$ analogously.

In view of
Proposition \ref{prop:combined-localized-Realization-structure},
we assume as we may,  after post-composing $\varphi$ with a bilipschitz map if necessary, that tents in 
\[\{\widehat \Omega_{q;i}\}\cup \bigcup_{\ell=1}^k \{\widehat \Omega_{q;i}\colon q\in Z_{\ell-2}^{[n-1]} \}\]
 are mutually disjoint.

When $k=1$, by Proposition \ref{prop:tunnel-contracting} there exists an $L'$-bilipschitz homeomorphism, hence quasiconformal, 
$\widehat \psi_{1;i} \colon R_{0;i}  \cup  T_{1;i}   \to R_{0;i}$ which is the identity on $|\Core(Z;\Sigma_i)|$, where  constant $L'=L'(n,K)$ depends only on $n$ and $\#(Z^{[n-1]})$. The composition $\widehat \psi_{1;i} \circ \varphi_{1;i}$ is the claimed map in the proposition.

Let now $k\geq 2$. We fix a tunnel-contracting bilipschitz homeomorphism 
\[
\psi_{k;i} \colon R_{0;i} \cup \left(\bigcup_{\ell=1}^{k-1} T_{\ell;i}\right) \cup T_{k;i} \to R_{0;i} \cup \bigcup_{\ell=1}^{k-1} T_{\ell;i}
\]
as in Proposition \ref{prop:tunnel-contracting}, which 
\begin{enumerate}
\item is the identity in the complement of 
$T_{k;i}\, \cup \left(\bigcup_{q\in Z_{k-2}^{[n-1]}}|\widehat \Omega_{q;i}|\right)$, and
\item maps each   $\tau_{q;i}=\varphi(|\tunnel(q;\Sigma_i)|)$ in $T_{k;i}$, $q\in Z_{k-2}^{[n-1]}$,  onto the cube $\Omega_{q;i}$ adjacent to $\tau_{q;i}$. 
\end{enumerate}
Since the graph size of $\tunnel(q;\Sigma_i)$ is at most $\lambda_\loc$, the bilipschitz constant $L''=L''(n,K)$ may be chosen to depend only on $n$ and $\lambda_\loc$.

We fix next a $L'''$-bilipschitz homeomorphism
\[
\psi_{k-1;i}\colon R_{0;i} \cup \bigcup_{\ell=1}^{k-1} T_{\ell;i} \to R_{0;i} \cup \bigcup_{\ell=1}^{k-2} T_{\ell;i},
\]
which 
\begin{enumerate}
\item is the identity in the complement of  $T_{k-1;i} \cup \left(\bigcup_{q\in Z_{k-3}^{[n-1]}} \widehat \Omega_{q;i}\right)$,
 \item maps each $\tau_{q;i}=\varphi(|\tunnel(q;\Sigma_i)|)$ in $T_{k-1;i}$, 
for
$q\in Z_{k-3}^{[n-1]}$, onto  $\Omega_{q;i}$,  and
 \item  is a scaling on each tent $ \widehat \Omega_{q';i} $  
for 
$q'\in Z_{k-2}^{[n-1]}$, \label{item:scaling-on-tent}
\end{enumerate}
where the bilipschitz constant $L'''=L'''(n,K)\geq L''$ may be chosen to depend only on $n$, $\lambda_\loc$, and $\nu$.

Follow the construction of $\psi_{k-1;i}$ inductively for  $\ell= k-2, \ldots, 2$, and replace condition \eqref{item:scaling-on-tent} by
\begin{enumerate}
\item [(3)'] $\psi_{\ell;i}$ is a scaling on each $ \widehat \Omega_{q';i}$,
for
$q'\in Z_j^{[n-1]}$ and $j\in\{\ell-1, \ldots, k-2\}$.
\end{enumerate}
We obtain then a sequence of 
$L'''$-bilipschitz homeomorphisms 
\[
\psi_{\ell;i}\colon R_{0;i} \cup \bigcup_{j=1}^{\ell} T_{j;i} \to R_{0;i} \cup \bigcup_{j=1}^{\ell-1} T_{j;i},\quad  \ell= k-2, \ldots, 2,
\]
for which the sequence 
\[
\psi_{k;i},\,  \psi_{k-1;i}, \, \psi_{k-2;i},\, \ldots, \,\psi_{1;i},
\]
has an essential property needed for the composition to be quasiconformal with a distortion constant independent of $k$. Namely,  for each $\ell\in \{1,\ldots, k\}$,  points in $T_{\ell;i}$ are not moved by $\psi_{\ell+2;i}\circ\cdots \circ \psi_{k;i}$. They are first moved  by 
$\psi_{\ell+1;i}$ to points in $T_{\ell;i}$, and 
then by the tunnel-contracting map $\psi_{\ell;i}$ into tents adjacent to  $T_{\ell;i}$.  From then on, on each of these tents,   $\psi_{1;i} \circ\cdots \circ \psi_{\ell-1;i}$ is a scaling. 
In the above, we have taken  $\psi_{k+1;i}$ and  $\psi_{k+2;i}$  to be the identity map. Points in $R_{0;i}$ stay fixed under $\psi_{2;i}\circ \cdots, \circ \psi_{k-1;i} \circ \psi_{k;i}$ and are only moved by $\psi_{1;i}$.

Therefore the  composition  
\[f_{k;i}= \psi_{1;i} \circ \cdots, \circ \psi_{k;i} \circ \varphi\colon |\Real(Z_k;\Sigma_i)|\to |\Real(Z;\Sigma_i)|,\]
 is a $\sfK$-quasiconformal homeomorphism for a constant  $\sK$ depending only on $L(n,K), L', L'', L'''$,  and is the identity on the core $\Core_K(Z;\Sigma_i)$.
\end{proof}

\begin{proof}[Proof of Evolution Theorem \ref{theorem:evolution-short}]
Quasiconformal stability of the sequence $(Z_k)$ follows by taking  the inverse $f_{k;i}^{-1}$ of the mapping $f_{k;i}$ in Theorem \ref{theorem:quasiconformal-stable}.

The quasiconformal stability, together with Corollary  \ref{cor:summary_Z-k}, completes the proof of Theorem \ref{theorem:evolution-short}.
\end{proof}



\section{Lakes of Wada}\label{sec:Wada}

\subsection{Proof of the Main Theorem}

Evolution of separating complexes repeated indefinitely yields a continuum whose complement consists of Lakes of Wada. The following theorem, including manifolds with two boundary components, is an extension of the main theorem (Theorem \ref{intro-thm:Wada-Riemannian-manifold}).

\begin{theorem}
\label{thm:Wada-Riemannian-manifold} 
 Let $n\ge 3$, $m\geq 2$, and $M$ be a compact connected Riemannian $n$-manifold with $m$ boundary components.Then there exist a constant $\sK=\sK(n,M)>1$ and a continuum $X\subset \interior M$ having the following properties:
\begin{enumerate}
\item each $x\in X$ is a common boundary point of all connected components of $M \setminus X$; \label{item:Wada-2}
\item  each  component  $\Omega$ of $M \setminus X$ is a John domain in $M$ for which $\Omega \cap \partial M$ contains exactly one component of $\partial M$ and $\Omega$ is $\sK$-quasiconformal to $(\Omega \cap \partial M)  \times [0,1)$.\label{item:John+QC-2}
\end{enumerate}
\end{theorem}

\begin{proof}
For the proof we pass from the Riemannian manifold $(M,g)$ to a cubical $n$-complex $K$ with a polyhedral 
metric $d_K$ in which each cube is isometric to a Euclidean unit cube, and for which $(|K|, d_K)$ is quasisimilar to $(M,g)$;
see Proposition \ref{prop:Riemannian-to-cubical}.   
Thus it suffices to prove Theorem \ref{intro-thm:Wada-Riemannian-manifold} in the setting of $(|K|, d_K)$. 

We continue to denote  the boundary components of $K$ by $\Sigma_1, \ldots, \Sigma_m$ and refer freely to the construction in Section \ref{sec:evolution-seq}. Let $Z_0, Z_1,\ldots, Z_k, \ldots$ be 
separating complexes in $K, \Refine^\nu(K),\ldots, \Refine^{k\nu}(K), \ldots $, respectively, constructed in Section \ref{sec:evolution-seq}, which have the properties in Theorem \ref{theorem:evolution-short}. 

\subsubsection*{Common boundary}
For each $i\in \{1,\ldots,m\}$, the cores 
\[
\Core(Z_k;\Sigma_i)= \Comp(Z_k;\Sigma_i) - \Star_{\Refine^{k\nu}(K)}(Z_k)
\]
are expanding, that is, $|\Core(Z_{k-1};\Sigma_i)| \subset \interior |\Core(Z_k;\Sigma_i)|$. Thus 
\[
M_i=\bigcup_{k=1}^\infty |\Core(Z_k;\Sigma_i)|,\quad i=1,\ldots, m,
\]
are mutually disjoint open sets in $M$. Let
\[
X= \bigcap_{k=1}^\infty |\Star_{\Refine^{k\nu}(K)}(Z_k)|.
\]
It has been proved in Proposition \ref{prop:topological-lakes-of-wada} that 
the topological boundaries $\partialtop M_i$ of domains $M_i$ satisfy 
$\partialtop M_1 = \ldots = \partialtop M_m =X.$

\subsubsection*{Quasiconformality.}
Recall from Convention \ref{convention:metric-subcomplex} and Lemma \ref{lemma:metric-Realization} that metrics on $|\Real(Z_k;\Sigma_i)|$ and on  $|\Comp(Z_k;\Sigma_i)|$ are conformally equivalent (in fact, locally isometric) when 
restricted to 
interiors $\interior (|\Real(Z_k;\Sigma_i)|)$ and  $\interior (|\Comp(Z_k;\Sigma_i)|\setminus |Z_k|)$.

Recall also that core $\Core(Z_k;\Sigma_i)$ and its  lift $\pi_k^{-1}(\Core(Z_k;\Sigma_i))$ are isomorphic, hence are identified with the same notation.

For each $i\in \{1,\ldots, m\}$ and  $k\ge 2$, let
\[
g_{k;i}= f_{k;i}^{-1}\, \colon |\Real(Z;\Sigma_i)|\to |\Real(Z_k;\Sigma_i)|
\]
be the inverse of the $\sfK$-quasiconformal mapping $f_{k;i}$ in Theorem \ref{theorem:quasiconformal-stable}, and let 
\[
h_{k;i}\colon \interior (|\Comp(Z;\Sigma_i)|\setminus |Z|) \to \interior(|\Comp(Z_k;\Sigma_i)|\setminus |Z_k|)
\]
be the mapping $h_{k;i} =\pi_k\circ  g_{k;i}\circ \pi_0^{-1} $.

For a fixed $i\in \{1,\ldots, m\}$, each mapping $h_{k;i}$ is the identity
on $|\Core(Z;\Sigma_i)|$. Hence the sequence $(h_{k;i})$ is equicontinuous and
they form a normal family of $K$-quasiconformal mappings $ \interior (|\Comp(Z;\Sigma_i)|\setminus |Z|) \to M$. 
Therefore, there is  a subsequence $(h_{k_j;i})$ which converges locally uniformly to a $\sfK$-quasiconformal map 
\[h_i\colon \interior(|\Comp(Z;\Sigma_i)|\setminus |Z|) \to \ker\left(\interior(|\Comp(Z_{k_j};\Sigma_i)|\setminus |Z_{k_j}|)\right);\]
see V\"ais\"al\"a \cite[Section 20]{Vaisala-book}. Recall that the kernel $\ker (A_j)$ of a sequence of sets $(A_j)$  in $|K|$ is the set of all points in $|K|$ which has a neighborhood that is contained in all but finitely many sets $A_j$. 
Since 
\[
|\Core(Z_k;\Sigma_i)|\subset |\Comp(Z_k;\Sigma_i)|\setminus |Z_k| \subset |\Core(Z_k;\Sigma_i)|\cup |\Star_{\Refine^{k\nu}(K)}(Z_k)|,
\]
and cores are expanding,  it is straightforward to check that 
\[
\ker\left(\interior(|\Comp(Z_{k_j};\Sigma_i)|\setminus |Z_{k_j}|)\right)= M_i \setminus |\partial K|.
\]

Since  $ \interior (|\Comp(Z;\Sigma_i)|\setminus |Z|)$ is $L(n,K)$-bilipschitz to $\Sigma_i\times (0,1)$, each $M_i\setminus |\partial K|$ is $\sfK$-quasiconformal to the space $\Sigma_i\times (0,1)$ for some constant $\sfK=\sfK(n,K)\geq 1$.

\subsubsection*{John domain.}
Let $i\in \{1,\ldots,m\}$. We fix an $n$-cube $Q'_0$ in $\Core(Z_0;\Sigma_i)$. Let also $x'_0$ be the center of $Q'_0$.

We construct first, for each $a\in M_i$, a path $\gamma_a$ from $a$ to $x'_0$.
Given $a\in M_i$,  
let $k\ge 0$ be the smallest integer for which $a \in |P_k|$, where 
\[
P_k=\Core(Z_k;\Sigma_i) - \Refine^\nu(\Core(Z_{k-1};\Sigma_i)).
\]
Let $Q_k\in P_k^{[n]}$ be an $n$-cube which contains $a$, and let $x_k$ be the center of $Q_k$. 

Suppose first that $k\ge 1$ and let 
\[
Q_k=Q_{k,1}, \ldots, Q_{k,p_k}; \,\,Q_{k-1,1},\ldots, Q_{k-1,p_{k-1}};\,\,  \ldots; \,\, Q_{0,1},\ldots, Q_{0,p_0}=Q'_0,
\]
be the $n$-cubes in Lemma \ref{lemma:core-to-core} and let points $x_{\ell,j}$ be centers of cubes $Q_{\ell,j}$, respectively.
By the definition of cores and the distance estimates in Remark \ref{rmk:John}, there exist constants $c=c(n,K) >0$ and $ C=C(n,K)>0$ such that 
\[
c 3^{- \ell \nu}\leq \dist (x_{\ell,j}, M\setminus M_i)\leq C  3^{- \ell \nu},
\]
for $1\leq \ell\leq k$ and $j\in\{1,\ldots, p_\ell\}$.
In this case, we fix a path $\gamma_a$ from the center $x_k$ of $Q_k$ to the center $x'_0$ of $Q'_0$ as follows. First, let $\sigma_a$ be the PL curve which connects $x_k=x_{k,1}$ to $x_0=x_{0,p_0}$ as in Remark \ref{rmk:John}. Let also $\sigma_a'$ be the line segment which connects $a$ to $x_k$. We take $\gamma_a = \sigma_a \cup \sigma'_a$.

For $k=0$, we observe that, since $P_0$ is a (product) collar of the boundary component $\Sigma_i$, that there exists a chain of $n$-cubes 
\[
Q_0 = Q_{0,1},\ldots, Q_{0,p_0}=Q'_0,
\]
where $p_0$ depends only on $\Sigma_i$, and hence again only on $K$. Thus, we may fix a path $\gamma_a$ from $x_0$ to $x'_0$ as in the case $k\ge 1$.

It  is  straightforward to check by  Lemma \ref{lemma:core-to-core} that there exists a constant $C=C(n, \mu, \nu, \#(K^{[n]}))>0$ for which
\[ 
s(\gamma_a(a,x))\leq C  \, \dist(x, M\setminus M_i)\quad \text{for all}\,\, x\in \gamma_a, 
\]
where $\gamma_a(x_k,x)$ is the part of $\gamma_a$ between $x_k$ and $x$, $s(\cdot)$ is the length.

Let now $a,b\in M_i$ and define $\gamma = \gamma_a \cup \gamma_b$. Then the length-distance estimates of $\gamma$ in the definition of John domain follows immediately.

This completes the proof of the theorem.
\end{proof}

\begin{corollary}
\label{cor:Wada-sphere} 
Let $n\ge 3$, $m \ge 2$, and $D_1,\ldots, D_m$ be mutually disjoint connected closed PL $n$-dimensional submanifolds of a closed and connected Riemannian $n$-manifold $M$. 
Then there exists a continuum $X$ in $M\setminus \bigcup_{i=1}^m D_i $, whose complement in $M$ has exactly $m$ components $ M_1,\ldots, M_m$ for which $D_i \subset M_i$, and  
\[ 
\partialtop M_1= \cdots = \partialtop M_m =\bigcap_{i=1}^m\,  \overline{M_i}=X,
\]  
and each $M_i$  is quasiconformal to $\interior D_i$ and is a John domain. 
\end{corollary}

The main theorem (Theorem \ref{intro-thm:Wada-on-sphere})  follows from Corollary \ref{cor:Wada-sphere} by choosing $D_1,\ldots, D_m$ to be Euclidean balls. We restate the theorem for reference.

\introthmWadasphere*

\subsection{Lakes of Wada in $\bS^2$} In dimension $n=2$ the situation is different. We finish this paper with the result that Lakes of Wada in $\bS^2$ are never John domains. Note that Lakes of Wada in $\bS^2$ are always open cells and hence conformal to the unit disk.

\begin{proposition}\label{prop:Wada_2dim} 
Let $m\geq 3$, and $M_1,\ldots, M_m$ be disjoint connected open sets in $\bS^2$ for which
\[ 
\partialtop M_1= \cdots = \partialtop  M_m =\bigcap_{i=1}^m\,  \overline{M_i}.
\]  
Then none of the domains $M_1,\ldots, M_m$ is a John domain.
\end{proposition}

\begin{proof}
Let $X=\bigcap_{i=1}^m\, \overline{M_i}$. Suppose towards contradiction that one of the domains $M_1,\ldots, M_m$ is a John domain. We assume, as we may, that $M_1$ is a John domain. We may assume that $\infty \in M_m$ and hence that domain $M_1$ is a $C$-John domain in $\R^2$. Note that, now $X$ is a continuum in $\R^2$.

Fix a point $a\in M_1$. Let $S^1(a,r)$ (resp. $S^1(a,r')$) be the smallest (resp. the largest) circle centered at $a$ that meets $X$. Thus,  $B^2(a,r) \subset M_1$ and  $ M_1 \cup X\subset \overline{B^2}(a,r')$.

Fix points $p\in S^1(a,r)\cap X$ and $p'\in S^1(a,r')\cap X$ and let 
\[
\epsilon = \min\{ |p-p'|/10, r/(10C), (r'-r)/(10 C)\}.
\] 
Let also $D=B^2(p,\epsilon)$ and $D'=B^2(p',\epsilon)$.
Since $p$ and $p'$ are in $X$, we have that each $M_i$ meets the interior of $D$ and also the interior of $D'$. 

We fix, for $i=2, 3$, a Jordan arc $\gamma_i \subset M_i\setminus (D\cup D')$ having one endpoint $q_i$ in $\partial D$, the other $q_i'$ in $\partial D'$; such Jordan arc $\gamma_i$ is given by a subarc of a Jordan arc in $M_i$ connecting points in $D$ and $D'$.

Let $\tau$ be the subarc of $\partial D$ having endpoints $q_2$ and $q_3$ and not intersecting the disk $B^2(a,r)$. Now $\gamma_2 \cup \tau \cup \gamma_3$ is a Jordan arc in $\R^2\setminus (D \cup D')$. Let now $\tau'$ be one of the subarcs of $\partial D'$ connecting $q_2'$ and $q_3'$. Then $J = \gamma_2\cup \tau \cup \gamma_3 \cup \tau'$ is a Jordan curve having the property that one of the components of $\R^2\setminus J$ contains the disk $B^2(a,r)$. Let $\mathcal Q$ be the component of $\R^2\setminus J$ which does not contain $B^2(a,r)$.

Since both $\gamma_2$ and $\gamma_3$ meet the circle $S=S^1(a, (r+r')/2)$, the intersection $\mathcal Q\cap S$  contains points in $X$. Thus we may fix a point $b \in M_1\cap \mathcal Q$ satisfying $\dist(b,S)<\epsilon$.

Let $\gamma \subset M_1$ be an arc connecting $a$ to $b$. Since $a$ and $b$ belong to different components of $\R^2\setminus J$ and $\gamma$ does not meet arcs $\gamma_2$ and $\gamma_3$, we conclude that $\gamma$ intersects either $\tau$ or $\tau'$. Let $t\in \gamma \cap (\tau \cup \tau')$. Since $\tau \subset \overline D$ and $\tau'\subset \overline{D'}$, we conclude that $\dist(t,X)\leq \epsilon$. Thus 
\[ 
\min\{|t-a|,|t-b|\} \geq \min\{ r, (r'-r)/2-2\epsilon\} >  C \epsilon \geq  C \dist(t, X).
\]
This is a contradiction, since we assumed that $M_1$ is $C$-John.
\end{proof}




\end{document}